\documentclass[
11pt,  reqno]{amsart}
\usepackage[margin=1.2in,marginparwidth=1.5cm, marginparsep=0.5cm]{geometry}

\usepackage{amsmath,amssymb,amscd,amsthm,amsxtra, esint}

\usepackage{booktabs} 
\usepackage{microtype}
\usepackage{amssymb}
\usepackage{mathrsfs}

\usepackage{upgreek} 

\usepackage{color}
\usepackage[implicit=true]{hyperref}

\usepackage{xsavebox}

\usepackage{cases}

\allowdisplaybreaks[2]

\newcommand{\rhoo}{\vec{\rho}}
\newcommand{\muu}{\vec{\mu}}

\newcommand{\fN}{\mathfrak{N}}

\newtheorem{theorem}{Theorem} [section]

\newtheorem{lemma}[theorem]{Lemma}

\newtheorem{remark}[theorem]{Remark}

\newtheorem{definition}[theorem]{Definition}

\DeclareMathOperator*{\supp}{supp}
\DeclareMathOperator{\med}{med}

\DeclareMathOperator{\Law}{Law}
\DeclareMathOperator{\Id}{Id}

\newcommand{\II}{\text{I \hspace{-2.8mm} I} }

\newcommand{\W}{\mathcal{W}}
\newcommand{\U}{\mathcal{U}}

\newcommand{\dr}{\theta}
\newcommand{\Dr}{\Theta}

\newcommand{\noi}{\noindent}
\newcommand{\Z}{\mathbb{Z}}
\newcommand{\R}{\mathbb{R}}
\newcommand{\C}{\mathbb{C}}
\newcommand{\T}{\mathbb{T}}
\newcommand{\N}{\mathbb{N}}

\let\Re=\undefined\DeclareMathOperator*{\Re}{Re}
\let\Im=\undefined\DeclareMathOperator*{\Im}{Im}

\let\P= \undefined
\newcommand{\P}{\mathbf{P}}
\newcommand{\PP}{\mathbb{P}}

\newcommand{\F}{\mathcal{F}}
\newcommand{\NN}{\mathcal{N}}
\newcommand{\RR}{\mathcal{R}}

\newcommand{\Gf}{\mathfrak{G}}

\newcommand{\nb}{\nabla}

\newcommand{\al}{\alpha}
\newcommand{\be}{\beta}
\newcommand{\dl}{\delta}
\newcommand{\Dl}{\Delta}
\newcommand{\eps}{\varepsilon}

\newcommand{\g}{\gamma}
\newcommand{\G}{\Gamma}
\newcommand{\ld}{\lambda}
\newcommand{\Ld}{\Lambda}
\newcommand{\s}{\sigma}
\newcommand{\Si}{\Sigma}
\newcommand{\ft}{\widehat}
\newcommand{\wt}{\widetilde}
\newcommand{\cj}{\overline}

\newcommand{\dt}{\partial_t}

\newcommand{\jb}[1]
{\langle #1 \rangle}

\renewcommand{\l}{\ell}

\newcommand{\les}{\lesssim}
\newcommand{\ges}{\gtrsim}

\newcommand{\ind}{\mathbf 1}

\newcommand{\E}{\mathbb{E}}

\renewcommand{\o}{\omega}
\renewcommand{\O}{\Omega}

\numberwithin{equation}{section}
\numberwithin{theorem}{section}

\newcommand{\too}{\longrightarrow}

\newtheorem*{ackno}{Acknowledgements}

\begin{document}

\baselineskip = 14pt

\title[Invariant Gibbs measure dynamics for the 2-$d$ Zakharov-Yukawa system]
{Invariant Gibbs dynamics for 
the two-dimensional Zakharov-Yukawa system}

\author[K.~Seong]
{Kihoon Seong}

\address{Kihoon Seong\\
Department of Mathematics\\
Cornell University\\ 
310 Malott Hall\\ 
Cornell University\\
Ithaca\\ New York 14853\\ 
USA }

\email{kihoonseong@cornell.edu}





\begin{abstract}
We study the Gibbs dynamics  for the Zakharov-Yukawa system
 on the two-dimensional torus $\T^2$, 
namely a Schr\"odinger-wave system
 with a Zakharov-type coupling $(-\Dl)^\g$.
We first construct 
the Gibbs measure in the weakly nonlinear coupling case ($0 \le \g<1$).
Combined  with the non-construction 
 of the Gibbs measure in the strongly nonlinear coupling case ($\g=1$)
 by Oh, Tolomeo, and the author (2020), this exhibits
 a phase transition at $\g = 1$.
We also study the dynamical problem and prove almost sure global
well-posedness of the Zakharov-Yukawa system and invariance of the
Gibbs measure under the resulting dynamics for the range $ 0 \le \g <
\frac 13$. 
In this dynamical part, the main step is to prove local well-posedness.
Our argument is based on 
the first order expansion 
and the operator norm approach
via  the random matrix/tensor estimate
from a recent work  Deng, Nahmod, and Yue (2020). 
In the appendix, we briefly discuss the Hilbert-Schmidt norm approach
and compare it with the operator norm approach.

\end{abstract}


\maketitle


\tableofcontents

\section{Introduction}
\label{SEC:1}

\subsection{Invariant Gibbs dynamics for the Zakharov-Yukawa system}
\label{SUBSEC:intro}
In this paper, our main goal is to construct invariant Gibbs
dynamics for the Schr\"odinger-wave systems on $\T^2 = (\R/2\pi\Z)^2$ with a Zakharov-type coupling\footnote{The flow of \eqref{Zak1} preserves the $L^2$-norm of the Schr\"odinger component $u$ but not the $L^2$-norm of the wave component $w$. Hence, to avoid a problem at the zeroth frequency in the Gibbs measure construction, we work with the massive linear part $\dt^2w-\Dl w+w$ and accordingly consider $\jb{\nb}^{2\g}$ instead of $(-\Dl)^\g$ as a coupling to complete the Hamiltonian formulation.}
\begin{align}
\begin{cases}
i \dt u +\Dl u = uw\\
\dt^2 w +  (1-\Dl) w = -\jb{\nb}^{2\g}(|u|^2)\\
(u,w,\dt w)|_{t=0}=(u_0, w_0, w_1),
\end{cases}
\label{Zak1}
\end{align}

\noi
where $ 0\le \g \le 1$ and $\jb{\nb}:=(1-\Dl)^{\frac 12}$. The case $\g=1$ corresponds to the well-known Zakharov system, while the case $\g=0$ corresponds to the Yukawa system.
The Zakharov-Yukawa system \eqref{Zak1} is a special case of the models introduced in \cite[Section 3]{SZ2} 
\begin{equation}
\begin{cases}
i \partial_t u + L_1 u =  u w \\
L_2 w = L_3 {|u|}^2 \, ,
\end{cases}
\label{DS}
\end{equation}

\noi
where $L_1,L_2$ and $L_3$ are constant coefficient differential operators.  
This class of systems is referred to as Davey-Stewartson (DS) systems in the work of Zakahrov-Schulman \cite[Section 3]{SZ2}.
In particular, (DS) systems is associated with a specific 2 dimensional system of the form \eqref{DS}, modeling the evolution of weakly nonlinear water waves travelling predominantly in one direction, in which the wave amplitude is modulated slowly in two horizontal directions. See, for example, \cite{GhS1, GhS2}. As for the global dynamics of \eqref{Zak1}, see \cite{BPSW}.

Notice that the Zakharov-Yukawa system \eqref{Zak1} is a Hamiltonian system 
associated with the Hamiltonian \eqref{Ham}:
\begin{align}
&H(u,w,\dt w) \notag \\
&=\frac 12 \int_{\T^2}|\nb u|^2 \, dx+\frac 12 \int_{\T^2} |u|^2w \,dx
+\frac 14 \int_{\T^2}| \jb{\nb}^{1-\g} w|^2 \, dx+\frac 14\int_{\T^2}|  \jb{\nb}^{-\g} \dt w |^2 \,dx 
\label{Ham}.
\end{align}

\noi
Moreover, the wave energy, namely, the $L^2$-norm of the Schr\"odinger component 
\begin{align*}
M(u)  = \int_{\T^2} |u|^2 dx
\end{align*}

\noi
is known to be conserved. See \cite{CCS}.
Then, the corresponding grand-canonical (Gibbs) measure $d\vec \rho_\g$ for the Hamiltonian system \eqref{Zak1} is formally given by\footnote{
A typical function $u$ in the support of $d\mu_1$ (on $\T^2$) is merely a distribution (not a function) and so a proper renormalization procedure is required.	
In this introduction, we keep our discussion at a formal level and  do not worry about renormalizations.} 
\begin{align}
\begin{split}
d \vec \rho_\g 
& = Z^{-1} e^{- H(u, w, \dt w)-M(u)} du \otimes dw \otimes d (\dt w)\\
& = Z^{-1} e^{- \frac{1}{2} \int_{\T^2}
|u|^2 w \, dx }d(\mu_1 \otimes \mu_{1-\g} \otimes \mu_{-\g} )(u,w,\dt w),
\end{split}
\label{Gibbs}
\end{align}

\noi
where for any given $ s \in \R$, $d\mu_s$ denote
a Gaussian field, formally defined by
\begin{align}
d \mu_s 
= Z_s^{-1} e^{-\frac 12 \| u\|_{{H}^{s}}^2} du
& =  Z_s^{-1} \prod_{n \in \Z^2} 
e^{-\frac 12 \jb{n}^{2s} |\ft u(n)|^2}   
d\ft u(n), 
\label{gauss0}
\end{align}

\noi
where 
$\jb{\,\cdot\,} = \big(1+|\,\cdot\,|^2\big)^\frac{1}{2}$
and $\ft u(n)$  denotes the Fourier transforms of $u$.
Note that 
$d\mu_s$ corresponds to 
the massive Gaussian free field $d\mu_1$ when  $s  = 1$
and to the white noise measure $d\mu_0$ when $s = 0$. 
Notice that the Gibbs measure $ \vec \rho_\g$ on
the vector  $(u,w, \dt w )$, formally defined in \eqref{Gibbs}, decouples as the Gibbs measure $\rho_\g$ on the vector $(u,w)$ and the Gaussian measure $\mu_{-\g}$ on the third component $\dt w$:
\begin{align}
 \vec \rho_{\g}= \rho_\g \otimes \mu_{-\g}.
\label{Gibbs3}
\end{align} 

\noi
In terms of the conservation of the Hamiltonian $H(u, w, \dt w)$ and the wave energy $M(u)$, the Gibbs measures $d\vec \rho_\g$ in \eqref{Gibbs} are expected to be invariant under the Zakharov-Yukawa  dynamics \eqref{Zak1}.


The main issue in constructing
the Gibbs measure $d\rhoo_\g$ in~\eqref{Gibbs} comes
from the focusing nature of the potential, 
i.e.~the interaction potential $ \int_{\T^2} |u|^2 w \,dx$ is unbounded\footnote{
In this paper, by ``focusing'', we mean ``non-defocusing'' (non-repulsive).
Namely, the interaction potential  (for example, $\int_{\T^2} |u|^2w$ in \eqref{Gibbs} or $\int_\T |u|^k$ \eqref{fGibbs3} ) is unbounded from above. }.
In the seminal work \cite{LRS}, Lebowitz, Rose, and Speer initiated
the study of focusing Gibbs measures in the one-dimensional setting. In this work, they
constructed the one-dimensional (focusing) Gibbs measures with an $L^2$-cutoff 
\begin{align}
d\rho_1(u,w,\dt w)=Z^{-1} \ind_{\{\int_\T |u|^2 dx \le K \}} e^{-\frac 12\int_\T  |u|^2w \, dx} d(\mu_1 \otimes \mu_{0} \otimes \mu_{-1} )(u,w,\dt w)
\label{fGibbs1}
\end{align}

\noi
and also the focusing $\Phi_1^k$-measure (Gibbs measures) in the $L^2$-(sub)critical setting (i.e. $2< k \le 6$)\footnote{Here, we consider the case where $k$ is an integer. In particular, $k$ is an even integer when $d\mu_1$ is the complex Gaussian free field.} with an $L^2$-cutoff
\begin{align}
d\rho(u) = Z^{-1} 
\ind_{\{\int_\T |u|^2 dx \le K \}} e^{\frac{1}{k} \int_{\T} |u|^k \, dx} d\mu_1(u)
\label{fGibbs3}
\end{align}

\noi
where $d\mu_1$ and $d\mu_0$ denote the periodic Wiener measure and the white noise on $\T$ , respectively.
The (focusing) Gibbs measures \eqref{fGibbs1} and \eqref{fGibbs3} were then proved to be invariant under the Zakharov system ($\g=1$ in \eqref{Zak1}) and cubic NLS on $\T$ by Bourgain \cite{BO94, BO94b}, respectively.
See Remark \ref{REM:phi} (i) for more explanations about the focusing $\Phi_1^k$-measure.

In the two-dimensional setting $\T^2$, Oh, Tolomeo, and the author \cite{OST} continued the study on the (focusing) Gibbs measures \eqref{fGibbs1}  and proved that the Gibbs
measure \eqref{fGibbs1}  (even with proper renormalization on the potential energy $\int_{\T^2} |u|^2w$ and on the $L^2$
-cutoff) is not normalizable as a probability measure:
\begin{align*}
\E_{\mu_1 \otimes \mu_0\otimes \mu_{-1}} \Big[   \ind_{\{\int_{\T^2} \; :|u|^2: \;  dx \, \leq K\}}   e^{ -\frac 12 \int_{\T^2}   \; :  |u|^2 : \;  w \,dx } \Big]=\infty
\end{align*}

\noi
for any $K>0$, where $d\mu_1$ and $d\mu_0$ denote the massive Gaussian free field and the white noise on $\T^2$ , respectively. 
As for the focusing $\Phi_2^k$-measure (Gibbs measure) \eqref{fGibbs3} on $\T^2$, Brydges and Slade \cite{BS} proved that the focusing $\Phi_2^4$-measure (i.e. the quartic interaction $k=4$ ) (even with proper renormalization on the potential energy $\frac 14\int_{\T^2} |u|^4$ and on the $L^2$
-cutoff) is not normalizable as a probability measure:
\begin{align*}
\E_{\mu_1} \Big[   \ind_{\{\int_{\T^2} \; :|u|^2: \; dx \, \leq K\}}   e^{ \frac 14\int_{\T^2}  \;  :  |u|^4  : \;  \, dx} \Big]=\infty
\end{align*}

\noi
for any $K>0$. An alternative proof was also given by Oh, Tolomeo, and the author \cite{OST}. 
We also notice that with the cubic interaction ($k=3$),
Jaffe constructed a (renormalized) $\Phi_2^3$-measure with a Wick-ordered $L^2$--cutoff. See \cite{BO95, OST}. See Remark \ref{REM:phi} (ii) for more explanations about the focusing $\Phi_2^k$-measure.

In this paper, our first goal is to establish the following phase transition (Theorem \ref{THM:1}) at the critical value $\g=1$:

\begin{itemize}
\item[\textup{(i)}] \textup{(weakly nonlinear coupling).}
Let $0\le\g<1$. Then, we have the normalizability of the (focusing) Gibbs measure 
\begin{align*}
\E_{\mu_1 \otimes \mu_{1-\g}\otimes \mu_{-\g}} \Big[   \ind_{\{\int_{\T^2} \; :|u|^2: \;  dx \, \leq K\}}   e^{ -\frac 12 \int_{\T^2}  \;  :  |u|^2   :  \; w \,dx } \Big]<\infty
\end{align*}

\noi
for any $K>0$.	
	
\smallskip
	
\item[\textup{(ii)}] \textup{(strongly nonlinear coupling).}
Let $\g=1$. Then, we have the non-normalizability of the (focusing) Gibbs measure 
\begin{align*}
\E_{\mu_1 \otimes \mu_0\otimes \mu_{-1}} \Big[   \ind_{\{\int_{\T^2} \; :|u|^2: \;  dx \, \leq K\}}   e^{ -\frac 12 \int_{\T^2}  \;   :  |u|^2 :  \; w \, dx } \Big]=\infty
\end{align*}

\noi
for any $K>0$.
 
\end{itemize}

\noi
Therefore, the two-dimensional Zakharov system $(\g=1)$ turns out to be critical, exhibiting the phase transition in terms of the measure construction.

We then study the dynamical problem \eqref{Zak1} 
i.e. construct invariant Gibbs dynamics ($d\rhoo_\g$-almost sure global well-posedness and invariance of the Gibbs measure; see Theorems \ref{THM:2} and \ref{THM:3}).

\begin{remark}\rm
In a recent work \cite{Seo23} by the author, the phase transition phenomenon in the three-dimensional setting $\T^3$ was explored for the Gibbs measure with $\g=0$, formally expressed as
\begin{align}
d\rhoo(u,w, \dt w) = Z^{-1}  
\exp\bigg( -\frac \ld2 \int_{\T^3}   \; :  |u|^2w : \;   \,dx -\infty \bigg) \ind_{ \{|\int_{\T^3} : | u|^2 :   dx| \le K\}}  d \muu_1(u,w,\dt w)
\label{SingGibbs3}
\end{align}

\noi
where the coupling constant\footnote{Since $|u|^2w$ is not sign definite, the sign of $\ld$ does not play any role.} $\ld \in \R\setminus \{0\}$ measures the strength
of the interaction potential and $\muu_1=\mu_1\otimes \mu_1 \otimes \mu_0$.  In the three-dimensional scenario, the  Gaussian free field $\mu_1$ is supported on a significantly rougher space, specifically $\mathcal{C}^{s}(\T^3)\setminus \mathcal{C}^{-\frac 12}(\T^3)$ for any $s<-\frac 12$, where $\mathcal{C}^s (\T^d)$ denotes the H\"older-Besov space. To accommodate this difference from the two-dimensional case, an additional (non-Wick) renormalization denoted by $-\infty$ is necessary for constructing the measure. While the focusing Gibbs measure on $\T^2$ in \eqref{Gibbs} can be constructed regardless of the coupling size $|\ld|$ (see Theorem \ref{THM:1}), in \cite{Seo23} the author demonstrated a phase transition for the Gibbs measure \eqref{SingGibbs3} in the three-dimensional setting. Specifically, normalizability ($Z<\infty$) was established in the weak coupling regime ($0<|\ld| \ll 1$) for every $K>0$, while non-normalizability ($Z=\infty$) was proven in the strong coupling case ($|\ld|\gg 1$) for every $K>0$. Similar phase transitions at critical exponents between strong and weak coupling regimes have been observed in other focusing models (see \cite{OOT1, OOT2}), but not with the Wick-ordered $L^2$-cutoff (i.e.~the generalized grand-canonical formulation). Therefore, we point out that the size of the coupling constant $\ld$ plays a crucial role in the analysis of the focusing Gibbs measure \eqref{SingGibbs3}, particularly concerning the phase transition with respect to $|\ld|$, marking a significant departure from focusing Gibbs measures on $\T^2$. In particular, in the weak coupling regime ($0<|\ld|\ll 1$) the Gibbs measure \eqref{SingGibbs3} is singular with respect to the base Gaussian field $\muu_1$ while the Gibbs measure \eqref{Gibbs} on $\T^2$ is absolutely continuous with respect the base Gaussian field. This singularity of the Gibbs measure introduced additional difficulties, compared to the Gibbs measures on $\T^2$, studied in this paper for the measure (non-)construction part.

\end{remark}

\begin{remark}\rm
\label{REM:phi}	
(i) In the one-dimensional setting $\T$,
in \cite{LRS}, Lebowitz, Rose, and Speer
also proved non-normalizability of the focusing $\Phi_1^k$-measure \eqref{fGibbs3} 
\begin{align*}
\E_{\mu_1}\Big[\ind_{\{\int_\T |u|^2 dx \le K \}} e^{\frac{1}{k} \int_{\T} |u|^k \, dx} \Big]=\infty
\end{align*}

\noi
in (i) the $L^2$-supercritical case $( k>6)$ for any $K > 0$
and 
(ii) the  $L^2$-critical case ($ k=6$), 
provided that $K > \|Q\|_{L^2(\R)}^2$, 
where $Q$ is the (unique\footnote{Up to the symmetries.}) optimizer for the Gagliardo-Nirenberg-Sobolev inequality
on $\R$
such that $\|Q\|_{L^6(\R)}^6 = 3\|Q'\|_{L^2(\R)}^2$.
In a recent work \cite{OSoT}, Oh, Sosoe and Tolomeo completed the focusing Gibbs measure construction program in the one-dimensional setting, including
the critical case $(k = 6)$ at the optimal mass threshold $K = \|Q\|_{L^2(\R)}^2$. See \cite{OSoT} for more details on the (non-)construction of the focusing $\Phi_1^k$- measure in the one-dimensional setting.


\smallskip

\noi
(ii) In the two-dimensional setting $\T^2$,	
the non-normalizability of the focusing $\Phi^k_2$-measure also holds for the higher order interactions $k \geq 5$ (see \cite[Remark 1.4]{OST}) 
\begin{align*}
\E_{\mu_1} \Big[   \ind_{\{\int_{\T^2} \; :|u|^2: \; dx \, \leq K\}}   e^{ \frac 1k\int_{\T^2}    \; :  |u|^k  :   \; \, dx} \Big]=\infty.
\end{align*}

We point out that typical elements $u$ in the support of massive Gaussian free field $d\mu_1$ (on $\T^2$) is log-correlated, namely 
\begin{align*}
\E_{\mu_1} \big[u(x) \cj u(y)\big]  \sim 
\log|x-y|
\end{align*}

\noi
for any $x, y \in \T^2$ with $x\ne y$. 
See \cite{ORSW, OST} for a related discussion.
In particular, in \cite{OST} Oh, Tolomeo, and the author studied the non-normalizability of the Gibbs measure with log-correlated base Gaussian fields on $\T^d$ for any $d \ge 1$.

\smallskip

\noi
(iii) In the three-dimensional setting $\T^3$, more complicated phenomena appear.
In \cite{OOT1}, Oh, Okamoto and Tolomeo studied the (non-)construction of the focusing $\Phi_3^3$-measure.
More precisely, in the weakly nonlinear regime, they proved normalizability of the $\Phi_3^3$-measure and show that it is
singular with respect to the massive Gaussian free field on $\T^3$. Furthermore, they proved that there exists a shifted measure with respect to which $\Phi_3^3$-measure is absolutely continuous. In the
strongly nonlinear regime, they established non-normalizability
of $\Phi_3^3$-measure. 

In the case of a higher order focusing interaction $(k \ge 4)$,
the focusing nonlinear interaction is worse than the cubic interaction $(k=3)$
and so non-normalizability would be satisfied for the higher order interactions $(k \ge 4)$. See also \cite{OOT2} for the non-normalizability of the focusing Hartree $\Phi_3^4$-measure.

\end{remark}

\begin{remark}\rm
Removing the infrared cut-off (i.e. the finite volume $\mathbb{T}^d$) to investigate Gibbs measures on the infinite volume $\mathbb{R}^d$ poses a highly intricate challenge. The construction of these measures, initially achieved in finite volumes and subsequently extended to infinite volume, stands as a significant accomplishment in constructive quantum field theory. In the context of focusing Gibbs measures on the infinite volume,  in \cite{Rider, TW} it was observed that the focusing $\Phi^4_1$-measure, an invariant measure for the focusing cubic Schrödinger equation, collapses onto the unit mass on the trivial path. In other words, it converges weakly to $\delta_0$, placing unit mass on the zero path, when taking a large torus limit. Consequently, we anticipate a triviality phenomenon in the large torus limit of the Gibbs measure $\rhoo_\g$ \eqref{Gibbs}, analogous to the one-dimensional focusing case, because of its focusing nature.

\end{remark}

\begin{remark}\rm
The equation \eqref{Zak1} is also known as the Schr\"odinger-Klein-Gordon system
with a Zakharov-type coupling. In the following, however, we simply refer to \eqref{Zak1} as the Schr\"odinger-wave system.

We point out that for our first main result
(Theorem \ref{THM:1}), we need to work with the massive linear part $\dt^2w-\Dl w+w$ in order to avoid a problem\footnote{Since the zeroth frequency is not
controlled due to the lack of the conservation of the $L^2$-mass under the dynamics.}
at the zeroth frequency in the Gibbs measure construction. Hence, due to this reason, we
work with the massive case in this paper.

\end{remark}


\subsection{Phase transition of the Gibbs measure}
In this subsection, we explain a renormalization procedure
required to give a proper meaning to the Gibbs measure $d\rhoo_\g$ defined in \eqref{Gibbs} and present our first main theorem for the phase transition i.e. (non-)construction of the focusing Gibbs measure.

The  Gibbs measure for the Zakharov-Yukawa   system \eqref{Zak1}
is formally given by 
\begin{align}
\begin{split}
d \rhoo_\g 
& = Z^{-1} e^{-Q(u,w)} d(\mu_1 \otimes \mu_{1-\g} \otimes \mu_{-\g} )(u,w, \dt w),
\end{split}
\label{Gibbs1}
\end{align}

\noi
where the potential energy $Q(u, w)$ is given by 
\begin{align*} 
Q(u, w) = \frac{1}2 \int_{\T^2}
|u|^2 w \, dx.
\end{align*}

\noi
Recall that on $\T^2$, the Gaussian field $d\mu_s$ in \eqref{gauss0}
is a probability measure supported
on $W^{s - 1 - \eps, p}(\T^2)$ for any $\eps > 0$ and $1 \leq p \leq \infty$.
For simplicity, we set $d\mu = d\mu_1$.
We now go over the Fourier representation
of functions distributed by $\mu \otimes \mu_{1-\g} \otimes \mu_{-\g}$\footnote{Gaussian masure $Z^{-1}\exp\big(-\frac 12 \| (u,w, \dt w)  \|_{H^{1} \times H^{1-\g} \times H^{-\g} } \big) du \otimes dw \otimes  d(\dt w) $ for which $H^{1}(\T^2) \times H^{1-\g}(\T^2) \times H^{-\g}(\T^2) $ is the Cameron-Martin space.}.
We now define random distributions $u= u^\o$, $w_0 = w_0^\o$, and $w_1=w_1^\o$ by 
the following  Gaussian Fourier series:\footnote{By convention, 
we endow $\T^2$ with the normalized Lebesgue measure $dx_{\T^2}= (2\pi)^{-2} dx$.}
\begin{equation} 
u^\o = \sum_{n \in \Z^2 } \frac{ g_n(\o)}{\jb{n}} e^{i \jb{n,x} }, \;
w_0^\o = \sum_{n\in \Z^2}  \frac {h_n(\o)} {\jb{n}^{1-\g}} e^{i \jb{n,x}   },
\;
\text{and}
\;
w_1^\o = \sum_{n\in \Z^2}  \frac {\l_n(\o)} {\jb{n}^{-\g}} e^{i \jb{n,x}   },
\label{IV2}
\end{equation}

\noi
where $\{g_n, h_n, \l_n\}_{n \in \Z^2}$ is a sequence of “independent standard” complex-valued\footnote{This means that $h_0, \l_0 \sim\NN_\R(0,1)$
and $\Re g_0, \Im g_0, \Re g_n, \Im g_n, \Re h_n, \Im h_n, \Re \l_n, \Im \l_n \sim \NN_\R(0,\tfrac12)$
for $n \ne 0$.}
Gaussian random variables on 
a probability space $(\O,\F,\PP)$ conditioned that $h_{-n}:=\cj h_n$ and $\l_{-n}:=\cj \l_n$. 
More precisely, with the index set $\Ld$ defined by
\begin{align*}
\Ld:=(\Z \times \Z_{+}) \cup (\Z_{+}\times \{ 0\}) \cup \{(0,0) \},
\end{align*}  

\noi
We define $\{h_n, \l_n \}_{n \in \Ld}$ to 
be a sequence of independent standard complex-valued Gaussian random variables (with $h_0, \l_0$ real-valued) and set $h_{-n} := \cj{h_n}$ and $\l_{-n}:=\cj \l_n$ for $n \in \Ld$\footnote{As for the Gaussian free field $d\mu$ on the first component $u$, we mean the complex Gaussian free field. 
On the other hand, the Gaussian fields $d\mu_{1-\g}$, $d\mu_{-\g}$ on the second and third component $w, \dt w$ mean the real Gaussian fields.}.

Denoting the law of a random variable $X$ by $\Law(X)$, 
we then have
\begin{align*}
\Law((u, w_0, w_1 )) = \mu \otimes \mu_{1-\g} \otimes \mu_{-\g} 
\end{align*}

\noi
for $(u, w_0, w_1)$ in \eqref{IV2}.
Note that  $\Law((u,w_0, w_1)) = \mu \otimes \mu_{1-\g} \otimes \mu_{-\g}$ is supported on
\begin{align*}
H^{s_1}(\T^2)\times H^{s_2}(\T^2) \times H^{s_2-1}(\T^2)
\end{align*}

\noi
for $s_1 <0$ and $s_2<-\g$ but not for $s_1 \geq 0$ and $s_2 \geq -\g$, respectively; see \cite{BO96, Zhid, OTh}.

As we pointed out, the key issue in constructing
the Gibbs measure $d\rhoo$ in~\eqref{Gibbs} comes
from the focusing nature of the potential, 
i.e.~the potential $Q(u, w)$ is unbounded from above.
In a seminal paper \cite{LRS}, 
Lebowitz, Rose, and Speer constructed the Gibbs measure $d\rhoo_\g$ when $d = 1$ and $\g=1$, 
by inserting a cutoff in terms of the conserved  wave energy $M(u) = \|u\|_{L^2}^2$.
Then, a natural question is to consider the construction
of the Gibbs measure $d\rhoo_\g$ in the two-dimensional setting $\T^2$.
We point out that in \cite{OST} Oh, Tolomeo, and the author 
proved that the (renormalized) Gibbs measure
for the Zakharov system ($\g=1$) on $\T^2$ is not normalizable, even with a Wick-ordered $L^2$-cutoff
i.e. the Gibbs measure $d\rhoo_1$  \eqref{Gibbs} for the Zakharov system on $\T^2$
cannot be realized as a probability measure even with a Wick-ordered $L^2$-cutoff on the Schr\"odinger component $u$.
Despite this non-normalizability result,
in this paper we construct the Gibbs measure $d\rhoo_\g$ for all $\g<1$.

In view of \eqref{Gibbs1}, 
we can write the formal expression~\eqref{Gibbs1}
for the Gibbs measure $d\rhoo_\g$ as\footnote{Hereafter, 
	we simply use $Z$, $Z_N$, etc. to denote various normalization constants.}
\begin{align}
\text{``}\, d \rhoo_\g(u,w,\dt w) 
= Z^{-1} \exp \bigg( - \frac 1 2 \int_{\T^2} |u|^2 w  \, dx \bigg) 
d(\mu \otimes \mu_{1-\g} \otimes \mu_{-\g} )(u,w, \dt w)
\, \text{"}.
\label{Gibbs0a}
\end{align}

\noi
Since $u$ in the support of the Gaussian free field $d\mu$ (on $\T^2$) is not a function\footnote{A typical element $u$ in the support of $d\mu$  
does not belong to $L^2(\T^2)$.}, 
the potential energy in \eqref{Gibbs0a} is not well defined
and thus a proper renormalization is required to give a meaning to 
\eqref{Gibbs0a}.
In order to explain the renormalization process, we first study the regularized model.
Given $N \in \N$, we define the (spatial) frequency projector $\pi_N$  by 
\begin{align*}
\pi_N f = 
\sum_{ |n| \leq N}  \ft f (n)  e^{i\jb{n,x}}.
\end{align*}

\noi
Let   $u$  be as in \eqref{IV2}
and 
set $u_N = \pi_N u$.
Note that, for each fixed $x \in \T^3$, 
$u_N(x)$ is
a mean-zero real-valued Gaussian random variable with variance
\begin{align}
\s_N = \E\big[|u_N(x)|^2\big] = \sum _{|n| \le N} \frac1{\jb{n}^2}
\sim \log N \too \infty, 
\label{sigma1}
\end{align}

\noi
as $N \to \infty$.
We then define the Wick power $:\! |u_N|^2 \!: $ by 
\begin{align}
:\! | u_N|^2 \!:  \, =  | u_N|^2  -\s_N.
\label{Wick1}
\end{align}

\noi
We point out that the Wick renormalization \eqref{Wick1} removes certain singularities (i.e. subtract a divergent contribution; see \cite{BO96, OTh}). 
This suggests us to consider the renormalized potential energy:
\begin{align}
Q_N(u, w) =  \frac{1}{ 2} \int_{\T^2}
:\! | u_N|^2 \!: w \,dx
\label{poten}
\end{align}

\noi
where $u_N = \pi_N u $ as in Subsection \ref{SUBSEC:notation}.
Thanks to the presence of $\s_N$ in \eqref{Wick1}, 
we can show that $Q_N $ 
converges to some limit $Q$
in $L^p(\mu \otimes \mu_{1-\g} \otimes \mu_{-\g})$ if $ \g<1$.
We now define the truncated renormalized Gibbs measure $d\rhoo_{\g, N}$, endowed with  a Wick-ordered $L^2$-cutoff, by 
\begin{align}
\begin{split}
d \rhoo_{\g, N} 
& = Z_{N}^{-1} \ind_{ \{|\int_{\T^2} 
: | u_N|^2 :   dx| \le K\}}
e^{-Q_N(u, w)} d(\mu \otimes \mu_{1-\g} \otimes \mu_{-\g})(u,w, \dt w). 
\end{split}
\label{Gibbs2}
\end{align}

\noi
As we have already pointed out, it is well known that a typical function $u$ in the support of $d\mu$ (on $\T^2$) does not belong to $L^2(\T^2)$ and so 
we put a Wick-ordered $L^2$-cutoff instead of considering a $L^2$-cutoff.

In \cite{OST}, Oh, Tolomeo, and the author proved  the non-construction 
of the Gibbs measure in the strongly nonlinear coupling case ($\g=1$)
i.e. the Gibbs measure $d\rhoo_1$ \textup{(}even with proper renormalization on the potential energy $\int_{\T^2} |u|^2w$ and on the $L^2$-cutoff\textup{)} can not be defined as a probability measure. In particular, we proved 
\begin{align}
\begin{split}
\sup_{N \in \N}  \E_{\mu \otimes \mu_0 \otimes \mu_{-1}} \Big[ \ind_{\{ |\int_{\T^2} \; :|u_N|^2: \;  dx \, | \leq K\}}   e^{ -Q_N (u,w)} \Big]
=\infty. 
\label{nonnor}
\end{split}
\end{align}

We, however, notice that once $\g<1$ (i.e. the weakly nonlinear coupling case),
we obtain the following uniform exponential integrability of the density, 
which allows us to construct the limiting Gibbs measure $d\rhoo_\g$. We now state our first main result.

\begin{theorem}[Normalizability of the Gibbs measure]
\label{THM:1}
Let $0\le \g <1$. Then, given any finite $ p \ge 1$, 
$Q_N $  in \eqref{poten} converges to some limit $Q$ in $L^p(\mu \otimes \mu_{1-\g} \otimes \mu_{-\g} )$.
Moreover, there exists $C_{p} > 0$ such that 
\begin{equation}
\sup_{N\in \N} \Big\| \ind_{ \{|\int_{\T^2} : | u_N|^2 :   dx| \le K\}}    e^{-Q_N(u,w)}\Big\|_{L^p(\mu \otimes \mu_{1-\g} \otimes \mu_{-\g} )}
\leq C_{p}  < \infty. 
\label{unexp1}
\end{equation}

\noi
In particular,  we have
\begin{equation}
\lim_{N\rightarrow\infty} \ind_{ \{|\int_{\T^2} : | u_N|^2 :   dx| \le K\}}  e^{ -Q_N(u,w)}= \ind_{ \{|\int_{\T^2} : | u|^2 :   dx| \le K\}}  e^{-Q(u,w)}
\qquad \text{in } L^p(\mu \otimes \mu_{1-\g} \otimes \mu_{-\g}  ).
\label{c11}
\end{equation}

\noi
As a consequence, 
the truncated renormalized Gibbs measure $d\rhoo_{\g,N}$ defined in \eqref{Gibbs2} converges, in the sense of \eqref{c11}, 
to the Gibbs measure $d\rhoo_\g$ given by
\begin{align}
d\rhoo_\g(u,w, \dt w)= Z^{-1}  \ind_{ \{|\int_{\T^2} : | u|^2 :   dx| \le K\}} e^{-Q(u,w)}d(\mu \otimes \mu_{1-\g} \otimes \mu_{-\g} )(u,w,\dt w).
\label{Gibbs4}
\end{align}

\noi
Furthermore, the resulting Gibbs measure $d\rhoo_\g$ is absolutely continuous with respect to 
the Gaussian  field~d$(\mu \otimes \mu_{1-\g} \otimes \mu_{-\g})$.

\end{theorem}

Theorem \ref{THM:1} shows that a phase transition occurs: (i) the (focusing) Gibbs measure is not constructible as a probability measure for $\g=1$ (see \eqref{nonnor}) and (ii) the (focusing) Gibbs measure is constructible as a probability measure for $\g<1$. 

The main task  in proving Theorem \ref{THM:1} 
is to show  the uniform exponential integrability \eqref{unexp1}.
We establish the bound \eqref{unexp1} by applying the variational formulation of the partition function introduced by Barashkov and Gubinelli \cite{BG} in the construction of
the $\Phi^4_3$-measure. See also \cite{GOTW, ORSW2, OOT1, OST, FS, OOT2}. We point out that the partition function $Z_N$ comes from the expectation with respect to the product measure $\mu \otimes \mu_{1-\g} \otimes \mu_{-\g}$ i.e. 
\begin{align*}
Z_N=\E_{\mu \otimes \mu_{1-\g} \otimes \mu_{-\g}  }\Big[  \ind_{ \{|\int_{\T^2} : | u_N|^2 :   dx| \le K\}} e^{-Q_N(u, w)}    \Big].
\end{align*}
   
\noi
Hence, it requires more careful analysis than dealing with  
one component measure in the variational formulation. We, however, notice that random field $u$ under $d\mu$ and random field $w$ under $d\mu_{1-\g}$ are independent and so we exploit some cancellation from the independence.  
Once the uniform bound \eqref{unexp1} is established, the $L^p$-convergence \eqref{c11} of the densities follows from (softer) convergence in measure of the densities. See \cite[Remark 3.8]{Tz08}.

\begin{remark}\rm
The phase transition from \eqref{nonnor} and Theorem \ref{THM:1} shows that the two-dimensional Zakharov system $(\g=1)$ is critical in terms of the Gibbs measure construction. 
\end{remark}


	
	



\subsection{Renormalized Zakharov-Yukawa system}
In this subsection, we now consider the following Zakharov-Yukawa system on $\T^2$ associated with the (renormalized) Hamiltonian
\begin{align}
\begin{cases}
i \dt u +\Dl u = uw\\
\dt^2 w +(1- \Dl) w = -\jb{\nb}^{2\g}(|u|^2-\fint |u|^2)
\label{Zak2}
\end{cases}
\end{align}

\noi
where $ \fint  f(x) dx := \frac{1}{(2\pi)^2} \int_{\T^2} f(x) dx$ 
denotes integration with respect to the normalized Lebesgue measure $(2\pi)^{-2} dx$ on $\T^2$
and the initial data $(u^\o, w_0^\o, w_1^\o)$ is distributed according to the Gibbs measure $d\rhoo_\g$ \eqref{Gibbs5}.
In view of the absolute continuity of $d\rhoo_\g$ with respect to the Gaussian field $d(  \mu \otimes \mu_{1-\g} \otimes d\mu_{-\g})$
(Theorem~\ref{THM:1}), we consider the random initial data $(u_0^\o, w_0^\o, w_1^\o)$
distributed according to $d(  \mu \otimes \mu_{1-\g} \otimes d\mu_{-\g})$ in the following discussion.


The important point is that  the renormalization  removes a certain singular component from the nonlinearity $|u|^2$
(see \eqref{nonlinear} below).
This allows us to study well-posedness
of the renormalized the Zakharov-Yukawa system \eqref{Zak2} on the support of the Gibbs measure $d\rhoo_\g$ in \eqref{Gibbs4}.
In \cite{Kishi2, Kishi1}, the Zakharov system was shown to be locally and globally well-posed in $H^s(\T^2) \times H^{\l}(\T^2) \times H^{\l-1}(\T^2)$
for some $s \ge 0$ and $\l \ge 0$.
It turns out, however, that Gibbs measures $d\rhoo_\g$ in \eqref{Gibbs4} are supported on 
$H^{-\eps}(\T^2) \times H^{-\g-\eps}(\T^2) \times H^{-\g-1-\eps}(\T^2)$ i.e. negative Sobolev spaces, which is beyond the scope of the known deterministic well-posedness results in~\cite{Kishi2, Kishi1}.
For this reason,  the main part of our analysis is devoted to the 
probabilistic construction of  local-in-time and global-in-time solutions to \eqref{Zak2} on the (low regularity) support of the Gibbs measure $d\rhoo_\g$.

Next, we interpret the nonlinearity $\NN(u)$:
\begin{align*}
\NN(u)=|u|^2 -\fint |u|^2.
\end{align*}

\noi
Namely, define bilinear operator $\NN(u_1,u_2)$ by setting
\begin{align}
& \mathcal{N}(u_1, u_2) (x, t) 
: = \sum_{n \in \Z^2}
\sum_{ \substack {n_1-n_2=n, \\n_1 \neq n_2  } }
\ft{u}_1(n_1, t)\cj{\ft{u} }_2(n_2, t)e^{i(n_1-n_2)x}
\label{nonlinear}
\end{align}

\noi
When all the arguments coincide, we simply  write $\NN(u) =\NN(u, u)$.
Notice that the most problematic interactions (the high-high interactions) are removed 
in the renormalized nonlinearity $\NN(u)$.
Note that this renormalization of the nonlinearity in \eqref{Zak2} comes from the Euclidean quantum field theory (see, for example, \cite{Simon}).\footnote{To be precise, it is an equivalent formulation  to the Wick renormalization in handling rough Gaussian initial data.}
This formulation first appeared  in the work of Bourgain \cite{BO96}
for studying  the invariant Gibbs measure for the defocusing cubic NLS on $\T^2$. See \cite{CO, OS, GO, OTh} for more discussion in the context of the (usual) nonlinear Schr\"odinger equations.

\begin{remark}\rm	
We briefly look into the relation between the renormalized Zakharov-Yukawa system
\begin{align}
\begin{cases}
i \dt u +\Dl u = uw\\
c^{-2}\dt^2 w+ (1-  \Dl) w = -\jb{\nb}^{2\g}(|u|^2-\fint |u|^2)
\label{Zak3}
\end{cases}
\end{align}

\noi
and the renormalized focusing Hartree NLS
\begin{align*}
i \dt u +\Dl u = - \big[(|u|^2- \fint |u|^2) * V \big]u
\end{align*}

\noi
where $V$ is a convolution potential with $\ft V(n)= \jb{n}^{-2+2\g}$. Notice that
by sending the wave speed $c$ in \eqref{Zak3} to $\infty$, 
the Zakharov-Yukawa system (Zakharov system) converges, at a formal level, to the renormalized focusing Hartree NLS (cubic NLS, respectively). As for the Zakharov system $(\g=1)$, see, for example, \cite{OzT, MN} for rigorous convergence results on $\R^d$.
We also refer to \cite{RSoh1, RSoh2} for a detailed explanation of how Gibbs measures, specifically in the context of focusing, can be microscopically derived from the perspective of many-body quantum mechanics.

\end{remark}

\subsection{Invariant dynamics for the Zakharov-Yukawa system}

In this subsection, we establish global-in-time flow on the support of the Gibbs measure $d\rhoo_\g$ \eqref{Gibbs5} and its invariance.
The main difficulty in studying these problems, even locally in time, comes from the
roughness of the support of the Gibbs measure. 
We first present $d\rhoo_\g$-almost sure local well-posedness result.
\begin{theorem}[Almost sure local well-posedness]\label{THM:2}
Let $0 \le \g <\frac 13$.
Then, the renormalized Zakharov-Yukawa system \eqref{Zak2} on $\T^2$ is $d\rhoo_\g$-almost surely locally well-posed. More precisely, for any $\eps>0$, there exists a set $\Si\subset H^{-\eps}(\T^2) \times H^{-\g-\eps}(\T^2) \times H^{-\g-1-\eps}(\T^2) $ of full $d\rhoo_\g$-measure such that for any $(u_0^\o,w_0^\o,w_1^\o)\in \Si$, there exists $\dl>0$ and a solution to the Cauchy problem for \eqref{Zak2}  on $[-\dl,\dl]$ with data $(u_0^\o,w_0^\o,w_1^\o)$, unique in the class 
\begin{align*}
z^{S,\o} +X^{s,b}_S(\dl)  \subset   & C\big([-\dl, \dl ]; H^{-\eps}(\T^2)  \big)  \\ 
z^{W,\o} +X^{\l,b}_{W}(\dl) \subset   & C\big([-\dl, \dl ]; H^{-\g-\eps}(\T^2)  \big) \cap C^1\big([-\dl,\dl]; H^{-\g-1-\eps}(\T^2) \big)
\end{align*} 

\noi
for some $b>\frac 12$ and
\begin{align*}
s-1<\l<1-2\g
\end{align*}

\noi
with $0<s<\frac 14-\frac \g2$ and $\l>0$, where $X_S^{s,b}(\dl)$, $X_W^{\l,b}(\dl)$,  $z^{S,\o}$, and $z^{W,\o}$ are defined in Subsection \ref{SUBSEC:Xsb} and \ref{SUBSEC:notation}, respectively.



\end{theorem}

The proof of Theorem~\ref{THM:2} is based on the first order expansion (McKean \cite{McKean}, Bourgain \cite{BO96}, Da Prato-Debussche \cite{DP1, DP2} type argument), which exploits the propagation of randomness under the corresponding linear flow. This method first decompose the solution by writing
\begin{align}
u=\text{random linear term+smoother term}
\label{decomp}
\end{align}

\noi
i.e. the decomposition of solution as the sum of the random linear evolution terms plus a smoother remainder
\footnote{In the field of stochastic PDEs, a well-posedness argument based on the decomposition~\eqref{decomp}
is usually referred to as the Da Prato-Debussche trick \cite{DP1, DP2}, where the random linear solution is replaced by the  solution to a linear stochastic PDE. It
is worthwhile to point out that the paper \cite{McKean, BO96} by McKean and Bourgain precede \cite{DP1, DP2}.}. 
Gaussian initial data in the first term of \eqref{decomp} are propagated linearly,
which preserves all the independence properties of the initial data $u^\o(0)$, $w^\o(0)$ in \eqref{Zak2} \footnote{This is no longer satisfied for nonlinear solutions $u(t), w(t)$ as soon as $t > 0$.} and in the second term of \eqref{decomp} the remainder term is treated as a perturbation.
The main idea is to use the propagation of the randomness in such a way that cancellations from the randomness happen and solve a fixed point problem
for the remainder term (see Subsection \ref{SUBSEC:firstor} for more explanations about Bourgain and Da Prato-Debussche trick).

To exploit better the independene strucutre of the random linear term involved, in \cite{BO96} Bourgain used the $TT^*$-argument with the Hilbert-Schmidt norm of the random matrices (kernel of $TT^*$ operator).
In Appendix \ref{SEC:A}, the estimates of random matrices are based on the Hilbert-Schmidt norm with the Wiener chaos estimate (Lemma~\ref{LEM:multigauss}). 
In order to attain a stronger coupling region, we need to go beyond the Bourgain's argument \cite{BO96}. 
We point out that a certain room exists between the operator norm and the Hilbert-Schmidt norm as follows:
\begin{align}
\| T \|^2_{\text{OP}}=\| TT^* \|_{\text{OP}} \le \| TT^* \|_{\text{HS}}. 
\label{HSS12}
\end{align}

\noi
Therefore, instead of using the Hilbert-Schmidt norm of the kernel matrix,
we use the operator norm approach based on the random tensor theory. This method rely on higher order versions of Bourgain’s $TT^*$ argument introduced by Deng, Nahmod, and Yue \cite{DNY}
\begin{align*}
\| T \|^{2m}_{\text{OP}}=\| (TT^*)^m \|_{\text{OP}}, 
\end{align*}

\noi
which makes us exploit better the independent structure.
More precisely, the essential difference between our analysis and that in \cite{BO96} is that the Hilbert-Schmidt norm of kernel matrices and the Wiener chaos estimate  will be replaced by the operator norm bound (Lemma \ref{LEM:tens}) coming from the random tensor theory in \cite{DNY} and the counting estimate will be also replaced by deterministic tensor estimates (see Subsection \ref{SUBSEC:dettens}), which allows us to reach a  stronger coupling region (see also Remark \ref{RM:essen}). In \cite{DNY}, Deng, Nahmod, and Yue recently developed a theory of random tensors, which forms a comprehensive framework for random dispersive equations. In recent years, the random tensor theory has played a crucial role in the well-posedness study of random dispersive equations; see \cite{DNY, Bring, OWZ}.

Notice that the low regularity nature of solutions in Theorem \ref{THM:1} comes from the random linear terms. We, however, point out that the random linear solutions also enjoy enhanced integrability than the smoother remainder terms thanks to the  independent structure involved.

\begin{remark}\rm
The decomposition \eqref{decomp} states that 
in the high frequency regime (i.e. at small spatial scales on the physical side), 
the dynamics is essentially governed by that of the random linear solution.
\end{remark}

\begin{remark}\rm
The extension of Theorem \ref{THM:2} to $\g <1$ would require more sophisticated arguments. We mention recent breakthrough works (random averaging operators/random
tensor) by Deng, Nahmod, and Yue \cite{DNY0, DNY}.

\end{remark}

\begin{remark}\rm\label{RM:essen}
We point out that it is essential to use the operator norm approach with the random matrix/tensor theory in proving Lemma \ref{LEM:zSzW} (in particular, Subsubcase 2.b.(ii)), Lemma \ref{LEM:zSRW} (in particular, Subcase 2.b), and Lemma \ref{LEM:RszW} (in particular, Subcase 2.b), where the cases cannot be proven by only using the Hilbert-Schmidt norm approach even when $\g=0$ (see also Lemma \ref{LEM:zSzW}'s Subsubcase 2.b.(i) and Appendix \ref{SEC:A} to compare the methods).   
\end{remark}

\begin{remark}\rm
In \cite{BO96, Richards, NS, GKO2, OOT1, OOT2}, the (random) operators with kernel (random) matrices appear and the Hilbert-Schmidt norm approach was used in dealing with them as in \eqref{HSS12}; see also Appendix \ref{SEC:A}. 
Hence, one can expect some improvements by using the operator norm with the random tensor theory as in this paper.	
\end{remark}

In constructing almost sure global-in-time dynamics, we exploit Bourgain’s invariant measure argument \cite{BO94, BO96, BT2, OTzW} to our setting.
More precisely, we use invariance of the Gibbs measure under the finite-dimensional
approximation of \eqref{Zak2} to obtain a uniform control on the solutions, and then
apply a PDE approximation argument to extend the local solutions to \eqref{Zak2} obtained from
Theorem \ref{THM:2} to global ones. As a consequence, we also obtain invariance of the Gibbs measure $d\rhoo_\g$
under the global flow of the renormalized Zakharov-Yukawa system \eqref{Zak2}.

\begin{theorem}[Almost sure global well-posedness and invariance of the Gibbs measure]
\label{THM:3}
Let $0\le \g <\frac 13$. Then, the Cauchy problem for \eqref{Zak2} is $d\rhoo_{\g}$-almost surely globally well-posed, and the Gibbs measure $d\rhoo_\g$ is invariant under the flow. More precisely, if $\Si$ is as in Theorem \ref{THM:2}, and if $ \vec\Phi(t) $ denotes the flow map of \eqref{Zak2} on $\Si$, then $d\rhoo_\g$ is invariant under $\vec \Phi(t)$ in the sense that for any $d\rho_\g$-measurable $A\subset \Si$, it holds 
\begin{align*}
\rhoo_\g\big(\Phi(t)(A)\big)=\rhoo_{\g}(A)
\end{align*}

\noi
for any $t \in \R$.
\end{theorem}

The proof of Theorem \ref{THM:3} follows
from a standard application of Bourgain’s invariant measure argument \cite{BO94, BO96, BT2, OTzW}, which is presented in Section \ref{SEC:INV}.


	

\subsection{Organization of the paper} 
In Section \ref{SEC:2}, we introduce some notations and preliminary (deterministic and probabilistic) lemmas. 
In Section \ref{SEC:conmea}, we present the variational formulation of the partition
function and prove the phase transition (Theorem \ref{THM:1}) by showing the uniform exponential integrability \eqref{unexp1}. In Section \ref{SEC:ga12}, we discuss probabilistic well-posedness with the random tensor theory. 
In Section \ref{SEC:rant}, we provide the basic definition and 
some lemmas on random tensors and present the proof of random tensor estimates.
In Section \ref{SEC:INV}, we prove the almost sure global well-posedness and invariance of the Gibbs measure by exploiting the Bourgain's invariant measure argument.
In Appendix \ref{SEC:A}, we implement the Hilbert-Schmidt norm approach with the Wiener chaos estimate and compare its result with when using the operator norm bound with the random tensor theory.


\section{Notations and preliminary lemmas}
\label{SEC:2}
In this section, we recall and prove basic lemmas to be used in this paper.

\subsection{Notations}
\label{SUBSEC:notation}
If a function $f$ is random, we may use the superscript $f^\o$ to show the dependence on $\o \in \O$.
Let $\eta \in C^\infty_c(\mathbb{R})$ be a smooth non-negative cutoff function supported on $[-2, 2]$ with $\eta \equiv 1$ on $[-1, 1]$ and set
\begin{align}
\eta_{_\dl}(t) =\eta(\dl^{-1}t)
\label{eta1}
\end{align}

\noi
for $\dl > 0$.
We also denote by   $\chi(t) $ another smooth non-negative cutoff function 
and let $\chi_{_\dl}(t) =\chi(\dl^{-1}t) $.

Let $\Z_{\geq 0} := \Z \, \cap\,  [0, \infty)$.
Given a dyadic number $N \in 2^{\Z_{\geq 0}}$, 
let $\P_N$
be the (non-homogeneous) Littlewood-Paley projector
onto the (spatial) frequencies $\{n \in \Z: |n|\sim N\}$
such that 
\[ f = \sum_{\substack{N\geq 1\\
\text{dyadic}}}^\infty \P_N f.\]

\noi
We also denote by $P_Q$ the Fourier projector onto $Q$ where $Q$ is a spatial frequency ball of radius $N$ (not necessarily centered at the origin). 
Given a non-negative integer $N \in \Z_{\geq 0}$, 
we also define the Dirichlet projector $\pi_N$ 
onto the frequencies $\{|n|\leq N\}$ by setting
\begin{align}
\pi_N f(x)=
\sum_{ |n| \leq N}  \ft f(n)    e^{in x}.
\label{Dir}
\end{align}

\noi
Moreover, we set
\begin{align*}
\pi_N^\perp  = \Id - \pi_N. 
\end{align*}

\noi
By convention, we also set $\pi_{-1}^\perp = \Id$.
By abuse of the notation, we also use the notation $\P_N$ to denote the operator on functions in $(t,x)$.
Also, define the operators $Q_L^S$, $Q_L^{W_\pm}$ on spacetime functions by
\begin{align*}
\F _{t,x}(Q_L^Su)(\tau ,k):=\eta _L(\tau +|k|^2)\F_{t,x}{u}(\tau ,k),\quad \F _{t,x}(Q_L^{W_\pm}w)(\tau ,k):=\eta _L(\tau \pm |k|)\F_{t,x}{w}(\tau ,k)
\end{align*}

\noi
for dyadic numbers $L\ge 1$.
We will write $P^S_{N,L}=P_NQ^S_L$, $P^{W_{\pm}}_{N,L}=P_NQ^{W_{\pm}}_L$ for brevity.
In what follows, capital letters $N$ and $L$ are always used to denote dyadic numbers $\ge 1$.
We will often use these capital letters with various subscripts, and also the notation
\begin{align*}
N_{\max}:=\max \big\{N_1,N_2, N    \big \} \qquad \text{and} \qquad L_{\max }:= \max \big\{ L_1, L_2, L  \big \}.
\end{align*}

We denote by $z^{S,\o}$ and $z^{W,\o}$ the linear solutions of Zakharov-Yukawa system  with initial data $(u_0^\o,w_0^\o,w_1^\o)$ as follows:
\begin{align*}
z^{S,\o}:&=e^{it\Dl}u_0^\o  \\
z^{W,\o}:&=\cos(t  \jb{\nb} )w_0^\o + \frac{\sin( t   \jb{\nb} )}{ \jb{\nb} }w_1^\o.
\end{align*}

We use $c,$ $ C$ to denote various constants, usually depending only on $\al$ and $s$. If a constant depends on other quantities, we will make it explicit. 
For two quantities $A$ and $B$,  we use $A\lesssim B$ to denote an estimate of the form $A\leq CB$, where $C$ is a universal constant,  independent of particular realization of $A$ or $B$. 
Similarly, we use $A\sim B$ to denote $A\lesssim B$ and $B\lesssim A$ . 
The notation $A\ll B$ means $A \leq cB$ for some sufficiently small constant $c$. 
We also use the notation $a+$ (and $a-$) to denote $a + \eps$ (and $a - \eps$, respectively)
for arbitrarily small  $\eps >0$ (this notation is often used 
when there is an implicit constant  which diverges in the limit $\eps\to 0$).

\subsection{Fourier restriction norm method}
\label{SUBSEC:Xsb}
In this subsection, we introduce the following Bourgain spaces for the Schr\"odinger and the wave equation. 
\begin{definition}[Bourgain spaces]
For $s,b\in \R$, define the Bourgain space for the Schr\"odinger equation $X^{s,b}_S$ and that for reduced wave equations $X^{s,b}_{W_{\pm}}$ by the completion of functions $C^\infty$ in space and Schwartz in time with respect to
\begin{align*}
\|u \|_{X_S^{s,b}(\R \times \T^2)}:&= \| \jb{n}^s \jb{\tau - |n|^2}^b \ft u(\tau,n ) \|_{\l_n^2 L^2_\tau(\Z^2\times \R ) }, \\
\|w \|_{X_{W_\pm}^{\l,b}  (\R \times \T^2) } :&= \| \jb{n}^\l \jb{\tau \mp \jb{n} }^b \ft w(\tau,n ) \|_{\l_n^2 L^2_\tau (\Z^2 \times \R^2) }.
\end{align*}

\noi
Here $\pm$ corresponds to the norm of $w_{\pm}$ in the system \eqref{reduced}.
We also define the Bourgain space for the wave equation $X^{s,b,p}_W$ by setting
\begin{align*}
\| w\|_{X^{s,b}_{W}}:=\| \jb{n}^\l \jb{|\tau| - \jb{n} }^b \ft w(\tau,n ) \|_{\l_n^2 L^2_\tau}
\end{align*}

\noi
i.e. replacing $W_\pm$ with $W$ in the above definition of $X^{s,b}_{W_{\pm}}$.
For $\dl>0$, define the restricted space $X^{s,b}_*(\dl)$ ($*=S$ or $W_{\pm}$ or $W$) by the restrictions of distributions in $X^{s,b}_*$ to $(-\dl,\dl)\times \T^2$, with the norm
\begin{align}
\| u\|_{X^{s,b}_*(\dl)}:=\inf \big\{ \| v \|_{X^{s,b}_*} : \;  v|_{(-\dl,\dl)}=u   \big\}.
\label{localtime}
\end{align}
\end{definition}

We note that for any $b>\frac 12$, we have $X_{*}^{s,b} \hookrightarrow C(\R; H^{s}(\T^2))$.
Next, we recall the linear estimates.  See \cite{BO93, GTV}.

\begin{lemma}\label{LEM:lin}
Let $s \in \R$, $\l \in \R$ and $0 < \dl \leq 1$.
	
\smallskip
	
\noi
\textup{(i)} For any $b \in \R$, we have
\begin{align*}
\|  e^{it\Dl} u_0 \|_{X_{S}^{s, b}(\dl)}
&\leq C_b \|u_0\|_{H^{s}},\\
\|  e^{it \jb{\nb} }w_0 \|_{X_{W_\pm}^{\l, b}(\dl)}
&\leq C_b \|w_0\|_{H^{\l}}.
\end{align*}

\smallskip
	
\noi
\textup{(ii)}
Let $ - \frac 12 < b' \leq 0 \leq b \leq b'+1$.
Then,  we have 
\begin{align*}
\bigg\|  \int_0^t e^{i(t-t')\Dl } F(t') dt'\bigg\|_{X_{S}^{s, b}(\dl)}
&\leq C_{b, b'} \dl^{1-b+b'} \|F\|_{X_{S}^{s, b'}(\dl)}\\
\bigg\|  \int_0^t e^{i(t-t')\jb{\nb} } F(t') dt'\bigg\|_{X_{W_\pm}^{\l, b}(\dl)}
&\leq C_{b, b'} \dl^{1-b+b'} \|F\|_{X_{W_\pm}^{\l, b'}(\dl)}.
\end{align*}
	
\end{lemma}

By restricting the $X^{s,b}$-spaces onto a small time interval $(-\dl,\dl)$, we can gain a small power of $\dl$ (at a slight loss in the modulation).
\begin{lemma}\label{LEM:dlpower}
Let $s\in \R$ and $b<\frac 12$. Then, there exists $C=C(b)>0$ such that
\begin{align*}
\| \eta_\dl(t) \cdot u  \|_{X^{s,b}}+\| \chi_\dl(t) \cdot u  \|_{X^{s,b}}\le C \dl^{\frac 12-b-}\| u\|_{X^{s,\frac 12-}}.
\end{align*}
\end{lemma}

For the proof of Lemma \ref{LEM:dlpower}, see \cite{CO}.

\subsection{Product estimates}
In this subsection, we present product estimates.
We first recall the following interpolation
and fractional Leibniz rule. As for the second estimate \eqref{bilinear+} below, 
see \cite[Lemma 3.4]{GKO}.

\begin{lemma}\label{LEM:prod}
The following estimates hold.

\noi
\textup{(i) (interpolation)} 
For  $0 < s_1  < s_2$, we have
\begin{equation*}
\| u \|_{H^{s_1}} \les \| u \|_{H^{s_2}}^{\frac{s_1}{s_2}} \| u \|_{L^2}^{\frac{s_2-s_1}{s_2}}.
\end{equation*}

\smallskip
	
\noi 
\textup{(ii) (fractional Leibniz rule)} Let $0\le s \le 1$. Suppose that 
$1<p_j,q_j,r < \infty$, $\frac1{p_j} + \frac1{q_j}= \frac1r$, $j = 1, 2$. 
Then, we have  
\begin{equation}  
\| \jb{\nb}^s (fg) \|_{L^r(\T^2)} 
\les \Big( \| f \|_{L^{p_1}(\T^2)} 
\| \jb{\nb}^s g \|_{L^{q_1}(\T^2)} + \| \jb{\nb}^s f \|_{L^{p_2}(\T^2)} 
\|  g \|_{L^{q_2}(\T^2)}\Big),
\label{bilinear+}
\end{equation}

\noi
where $\jb{\nb} = \sqrt{1 - \Dl}$.

\end{lemma}

\subsection{Hilbert-Schmidt norm}
In this subsection, we present the Hilbert-Schmidt norm estimate of (random) matrices (kernel of $T^* T$ operator).
To exploit better the independene strucutre of the random variable involved, we use the following lemma.  
\begin{lemma}[Hilbert-Schmidt norm]
\label{LEM:matrix}
The matrix $\Gf$ is given by 
\begin{align}
\Gf \{b_{n_1} \} =\sum_{n_1 \in \Z^d} \s(n,n_1)b_{n_1}
\label{ss1}
\end{align}
	
\noi
for any $\{b_{n_1} \} \in \l^2$. Then, we have
\begin{align*}
\| \Gf \|_{\l^2 \to \l^2}^2&=\| \Gf^* \Gf \|_{\l^2 \to \l^2}\\
&\les \max_n \sum_{n_1 } |\s(n_1,n)|^2+\bigg(\sum_{n,n': n \neq n'} \Big| \sum_{n_1 } \s(n_1,n') \cj\s(n_1,n)  \Big|^2   \bigg)^\frac 12 
\end{align*}
	
\end{lemma}

\begin{proof}
From \eqref{ss1}, $\Gf^*$ is given by 	
\begin{align*}
\Gf^* \{b_{n_1}\}=\sum_{n_1 } \cj\s(n_1,n)b_{n_1}
\end{align*}
	
\noi
and so
\begin{align}
\Gf^*\Gf \{ b_{n'} \}= \sum_{n'} \bigg( \sum_{n_1} \s(n_1,n')\cj\s(n_1,n) \bigg)b_{n'}.
\label{GG0}
\end{align}
	
\noi
We split \eqref{GG0} into two cases as follows
\begin{align}
\Gf^*\Gf \{ b_{n'} \}=\underbrace{  \sum_{n_1} |\s(n_1,n) |^2 b_{n}  }_{\text{diagonal term}} +\underbrace{\sum_{n':n' \neq n} \bigg( \sum_{n_1} \s(n_1,n')\cj\s(n_1,n) \bigg)b_{n'} }_{\text{non-diagonal term} }
\label{GG00}
\end{align}
	
\noi
Hence, from \eqref{GG00} and the Cauchy-Schwarz inequality in $n'$, we have
\begin{align*}
\| \Gf^*\Gf \{ b_{n'} \} \|_{\l^2}&\les  \bigg[ \max_n \sum_{n_1 } |\s(n_1,n)|^2+\bigg(\sum_{n,n': n \neq n'} \Big| \sum_{n_1 } \s(n_1,n') \cj\s(n_1,n)  \Big|^2   \bigg)^\frac 12  \bigg] \|b_n \|_{\l^2}. 
\end{align*}
	
\noi
This completes the proof of Lemma \ref{LEM:matrix}.
\end{proof}

\subsection{On discrete convolutions}

Next, we recall the following basic lemma on a discrete convolution.

\begin{lemma}\label{LEM:SUM}
\textup{(i)}
Let $d \geq 1$ and $\al, \be \in \R$ satisfy
\[ \al+ \be > d  \qquad \text{and}\qquad \al, \be < d.\]
\noi
Then, we have
\[
\sum_{n = n_1 + n_2} \frac{1}{\jb{n_1}^\al \jb{n_2}^\be}
\les \jb{n}^{d - \al - \be}\]
	
\noi
for any $n \in \Z^d$.
	
\smallskip
	
\noi
\textup{(ii)}
Let $d \geq 1$ and $\al, \be \in \R$ satisfy $\al+ \be > d$.
\noi
Then, we have
\[
\sum_{\substack{n = n_1 + n_2\\|n_1|\sim|n_2|}} \frac{1}{\jb{n_1}^\al \jb{n_2}^\be}
\les \jb{n}^{d - \al - \be}\]
	
\noi
for any $n \in \Z^d$.
	
\end{lemma}

Namely, in the resonant case (ii), we do not have the restriction $\al, \be < d$.
Lemma \ref{LEM:SUM} follows
from elementary  computations.
See, for example,  Lemmas 4.1 and 4.2 in \cite{MWX} for the proof.

\subsection{Strichartz estimates on $\T^2$ }
In this subsection, we record the $L^4_{t,x}$-Strichartz estimates on $\T^2$. We first recall (see \cite{BO93}) 
\begin{align}
\bigg\|  \sum_{n \in Q} a_n e^{i(\jb{n,x}+|n|^2 t) }  \bigg\|_{L^4_{t,x}([0,1]\times \T^2) } \les_\eps |Q|^\eps \bigg(\sum_{n \in Q } |a_n|^2 \bigg)^\frac 12
\label{linL4}
\end{align} 

\noi
where $Q$ is a spatial frequency ball of radius $N$ (not necessarily centered at the origin) and $|Q|=\#Q$. Then, from \eqref{linL4} and H\"older's inequality, we have
\begin{align}
\begin{split}
\|P_{Q} u\|_{L_{t,x}^4([0,1]\times \T^2 )} &\le \int_\R \bigg\|   \sum_{n \in Q} \ft u(\tau+|n|^2,n ) e^{ i(\jb{n,x}+|n|^2t)  } \bigg\|_{L^4_{t,x}([0,1] \times \T^2) } d \tau\\ 
&\les  |Q|^\eps \int_\R \bigg( \sum_{n \in Q} |\ft u(\tau+|n|^2,n) |^2   \bigg)^\frac 12 d\tau\\
&\les |Q|^\eps \bigg( \sum_{n \in Q} \int \jb{\tau+|n|^2}^{2b} |\ft u(\tau,n)|^2  \bigg)^\frac 12 \les N^{\eps } \| u\|_{X_S^{0,b}}
\label{L41}
\end{split}
\end{align}

\noi
for any $b> \frac 12$, where $P_Q$ is the Fourier projector onto $Q$. In the following lemma, we improve the $L^4_{t,x} $-Strichartz estimates \eqref{L41} by using the Hausdorff-Young inequality and an interpolation.
\begin{lemma}
\label{LEM:L4}
Let $\eps>0$. Then, we have
\begin{align*}
\| P_Q u \|_{L^4_{t,x}([0,1] \times \T^2 ) } \les_{\eps} N^\eps \|P_Q u \|_{X_S^{0,\frac 12-\frac \eps4}}, 
\end{align*}
	
\noi
where $P_Q$ is the Fourier projector onto $Q$  and $Q$ is a spatial frequency ball of radius $N$ (not necessarily centered at the origin). 
\end{lemma}

\begin{proof}
From the Hausdorff-Young inequality, we have
\begin{align}
\begin{split}
\|P_Q u \|_{L^4_{t,x} } &\le \bigg( \sum_{n\in Q} \int_\R |\ft u(\tau, n) |^\frac 43 d\tau   \bigg)^{\frac 43}\\
&\le \bigg( \sum_{n\in Q}  \Big(  \int_\R  \jb{\tau+|n|^2}^{2b'} |\ft u(\tau, n)|^2     \Big)^{\frac 23}     \bigg)^{\frac 34}\\
&\le N^{\frac 12} \bigg( \sum_{n \in Q} \int \jb{\tau+|n|^2}^{2b'} |\ft u(\tau,n)|^2  \bigg)^\frac 12=N^{\frac 12} \| P_Q u \|_{X_S^{0,b'}}.
\label{L42}
\end{split}
\end{align}

\noi
where $N$ is the size of $Q$ and $b'>\frac 14$.
By interpolating \eqref{L41} and \eqref{L42}, we obtain the desired result.

\end{proof}

For further use in the following sections, we also record another estimate \eqref{interop}. 
By interpolating the following two estimates coming from \eqref{L41} and the definition of $X_S^{s,b}$ 
\begin{align*}
\| P_Q u \|_{L^4_{t,x} ([0,1] \times \T^2) } &\les_{\eps} N^\eps \|P_Q u \|_{X_S^{0,\frac 12+\eps'}}, \\
\| P_Q u \|_{L_{t,x}^2([0,1] \times \T^2) } &\les \| P_Q u \|_{X_S^{0,0}},
\end{align*} 

\noi
we have
\begin{align}
\| P_Q u\|_{L^{2+}_{t,x} ([0,1] \times \T^2) } \les N^{0+} \| P_Q u\|_{X_S^{0,0+}}.
\label{interop}
\end{align}

\subsection{Counting estimates for lattice points and a key multilinear estimate}
In this subsection, we first record the following counting estimates.

	
		

\begin{lemma}[high-high interactions and low-modulation]	
\label{LEM:count1}	
Let
\begin{align*}
1\ll N \les N_1 \sim N_2 \quad \text{and} \quad    M \ll N_1.
\end{align*}

Then, for any fixed $n \in \Z^2$ with $|n| \sim N$, we have
\begin{align*}
\begin{split}
&\# \big\{n_1 \in \Z^2 :  |n_1|^2\pm |n|-|n_2|^2=O(M), \; n_1+n_2=n, \; |n_1| \sim N_1, \; \text{and} \; |n_2| \sim N_2  \big\} \\
& \hphantom{XX} \les   \Big(\frac MN +1\Big)N_1
\end{split}
\end{align*}

\end{lemma}

\begin{proof}	
Note that $|n_1|^2\pm |n|-|n_2|^2=O(M)$ and $n_1+n_2=n$ imply 
\begin{align*}
-|n|^2+2\jb{n,n_1} \pm |n|=O(M).
\end{align*}

\noi
Hence, from $|n| \sim N $, we have
\begin{align*}
\frac n{|n|} \cdot n_1= \frac {|n|}2 \mp \frac 12 + O\Big(\frac M{N} \Big)
\end{align*}

\noi
i.e. the component of $n_1$ parallel  to $n$ is restricted in an interval of length $O\Big(\frac M{N} \Big)$. Hence, we have 
\begin{align*}
&\# \big\{n_1 \in \Z^2 :  |n_1|^2\pm |n|-|n_2|^2=O(M), \; n_1+n_2=n, \; |n_1| \sim N_1, \; \text{and} \; |n_2| \sim N_2  \big\}\\
& \hphantom{X} \les  \Big(\frac M{N} +1  \Big)N_1,
\end{align*}

\noi
which proves the desired result.

\end{proof}

Next, we state the following lattice point counting bound that will be used in the proof of multilinear estimates in Subsection \ref{SUBSEC:bis} and \ref{SUBSEC:biw}. For the proof, see Lemma 4.3 in \cite{DNY0}.
\begin{lemma}\label{LEM:circle}
Given $0 \neq m \in \Z[i]$, $a_0, b_0 \in \C$, and $M,N>0$, the number of tuples $(a,b)\in \Z[i]^2$ that satisfies 
\begin{align*}
ab=m,\; |a-a_0|\le M, \; |b-b_0| \le N
\end{align*}

\noi
is $O(M^\eps N^\eps)$ for any small $\eps>0$, where the constant depends only on $\eps>0$.
\end{lemma}

To introduce Lemma \ref{LEM:multi}, we define several dyadic frequency regions:
\begin{align*}
\mathfrak{P}_1:=\big\{(\tau ,k): |k|\le 2 \big\},\qquad &\mathfrak{P}_N:=\big\{(\tau ,k): \tfrac{N}{2}\le |k|\le 2N \big\},\qquad N\ge 2,\\
\mathfrak{S}_1:=\big \{(\tau ,k): |\tau +|k|^2|\le 2  \big \},\qquad &\mathfrak{S}_L:=\big \{(\tau ,k): \tfrac{L}{2}\le |\tau +|k|^2|\le 2L \big \},\qquad L\ge 2, \\
\mathfrak{W}^\pm _1:=\big \{(\tau ,k): |\tau \pm |k||\le 2 \big \},\qquad &\mathfrak{W}^\pm _L:=\big \{(\tau ,k): \tfrac{L}{2}\le |\tau \pm |k||\le 2L \big \}  ,\qquad L\ge 2.
\end{align*}

\noi
We now state the following multilinear estimate in \cite[Proposition 3.2]{Kishi1}.
This multilinear estimate will be used in Lemmas \ref{LEM:RSRW} and \ref{LEM:RSRW}. 
More precisely, if an interaction is the high-low interactions where one Schr\"odinger frequency is much greater than the other Schr\"odinger frequency (i.e. high-modulation cases $L_{\max} \ges N_{\max}^2$), then we can use the following lemma.
\begin{lemma}[\cite{Kishi1}]
\label{LEM:multi}
Let $N_j,L_j\ge 1$ be dyadic numbers and $f,g_1,g_2\in L^2(\R \times \Z^2 )$ be real-valued nonnegative functions with the support properties
$\supp {f} \subset {\mathfrak{P}_{N_0}\cap \mathfrak{W}_{L_0}^\pm},\; \supp{g_j} \subset \mathfrak{P}_{N_j}\cap \mathfrak{S}_{L_j},\; j=1,2 $.
Moreover, assume $N_1 \gg N_2$ or $N_2\gg N_1$. Then, we have
\begin{align*}
&\sum_{n,n_1: n_1+n_2=n}\int_{\tau, \tau_1: \tau=\tau_1+\tau_2} f(\tau,n) g_1(\tau_1,n_1)   g_2(\tau_2,n_2)\\
&\les  L_{\max}^{\frac 12} L_{\med}^{\frac 38} L_{\min}^{\frac 38}N_{\min}^{\frac 12} N_{\max}^{-1} \| f\|_{L^2_\tau \ell^2_n} \|g_1 \|_{L^2_\tau \ell^2_n} \|g_2 \|_{L^2_\tau \ell^2_n}.
\end{align*}

\end{lemma}

\subsection{Tools from stochastic analysis}
In this subsection, we present the probabilistic tools. 
We first recall the Wiener chaos estimate (Lemma~\ref{LEM:multigauss}).
For this purpose, we first recall 
basic definitions
from stochastic analysis;
see \cite{Bog, Shige}. 
Let $(H, B, \mu)$ be an abstract Wiener space.
Namely, $\mu$ is a Gaussian measure on a separable Banach space $B$
with $H \subset B$ as its Cameron-Martin space.
Given  a complete orthonormal system $\{e_j \}_{ j \in \N} \subset B^*$ of $H^* = H$, 
we  define a polynomial chaos of order
$k$ to be an element of the form $\prod_{j = 1}^\infty H_{k_j}(\jb{x, e_j})$, 
where $x \in B$, $k_j \ne 0$ for only finitely many $j$'s, $k= \sum_{j = 1}^\infty k_j$, 
$H_{k_j}$ is the Hermite polynomial of degree $k_j$, 
and $\jb{\cdot, \cdot} = \vphantom{|}_B \jb{\cdot, \cdot}_{B^*}$ denotes the $B$--$B^*$ duality pairing.
We then 
denote the closure  of 
polynomial chaoses of order $k$ 
under $L^2(B, \mu)$ by $\mathcal{H}_k$.
The elements in $\mathcal{H}_k$ 
are called homogeneous Wiener chaoses of order $k$.
We also set
\begin{align}
\mathcal{H}_{\leq k} = \bigoplus_{j = 0}^k \mathcal{H}_j
\notag
\end{align}

\noi
for $k \in \N$.

Let $L = \Dl -x \cdot \nabla$ be 
the Ornstein-Uhlenbeck operator.\footnote{For simplicity, 
	we write the definition of the Ornstein-Uhlenbeck operator $L$
	when $B = \R^d$.}
Then, 
it is known that 
any element in $\mathcal H_k$ 
is an eigenfunction of $L$ with eigenvalue $-k$.
Then, as a consequence
of the  hypercontractivity of the Ornstein-Uhlenbeck
semigroup $U(t) = e^{tL}$ due to Nelson \cite{Nelson}, 
we have the following Wiener chaos estimate
\cite[Theorem~I.22]{Simon}.

\begin{lemma}\label{LEM:multigauss}	
Let $k \in \N$.
Then, we have
\begin{equation*}
\|X \|_{L^p(\O)} \leq (p-1)^\frac{k}{2} \|X\|_{L^2(\O)}
\end{equation*}
	
\noi
for any $p \geq 2$ and any $X \in \mathcal{H}_{\leq k}$.
As a consequence, the multilinear Gaussian expression
\begin{align*}
F_k(\o):=\sum_{n_1,\dots,n_k}c_{n_1,\dots, n_k} g_{n_1}(\o)g_{n_2}(\o)\cdots g_{n_k}(\o)
\end{align*}	
	
\noi
for some $k \ge 1$ and $\{c_{n_1,\dots,n_k} \} \in \l^2 \big( (\Z^2)^k \big)$ satisfies
\begin{align*}
\|  F_k \|_{L^p(\O)} \leq \sqrt{k+1}(p-1)^{\frac k2} \| F_k\|_{L^2(\O)}
\end{align*}

\noi
for any $p \ge  2$.
Moreover, there exists $c>0$ such that for any $\ld>0$, we have
\begin{align*}
\PP \big\{ |F_k|  > \ld  \big \} \le \exp\big(  -c\ld^\frac 2k  \|F_k \|_{L^2(\O)}^{-\frac 2k}     \big).
\end{align*}

\end{lemma}

We next present a well known fact (see also for example \cite{OHDIE, CO}).

\begin{lemma}\label{LEM:prob}
Let $\eps>0$ and $Q$ be a lattice ball of radius $N$ in $\R^2$ (not necessarily centered at the origin). Then, given $\be>0$, there exists a constant $c>0$ such that we have
\begin{align*}
\PP  \Big(   \max_{n \in Q} |g_n| >\dl^{-\be} ( \#Q)^\eps    \Big) \les N^{0-}e^{-\frac 1{\dl^c}}
\end{align*} 

\noi
for any $\dl>0$.	
\end{lemma}

\begin{proof}
Note that
\begin{align*}
\PP  \Big(   \max_{n \in Q} |g_n| >\dl^{-\be} ( \#Q)^\eps    \Big) &\le \sum_{n\in Q} \PP \Big( |g_n| >\dl^{-\be}( \#Q)^\eps  \Big)\\
&=\sum_{n \in Q} \int_{|g_n| \ge \dl^{-\be}(\#Q)^\eps } e^{-\frac {|g_n|^2}{2} } dg_n\\
&\les \sum_{n \in Q} e^{-c\frac{(\#Q)^\eps }{\dl^c}}\\
&\les N^2 e^{-c\frac{(\#Q)^\eps }{\dl^c} } \les N^{0-} e^{-c\frac{N^{2\eps} }{\dl^c} +(2+)\log N  } \les N^{0-}e^{-\frac 1{\dl^c} }.
\end{align*}

\noi
Hence, we obtain the desired result.
\end{proof}

	

In probabilistic well-posedness theory, a probabilistic improvement of Strichartz estimates for random linear solutions plays an important role.
\begin{lemma}
\label{LEM:prstricha}
Let 
\begin{align*}
f^\o(t,x)=\sum_{n \in \Z^2} c_n g_n(\o) e^{i(\jb{n,x}-|n|^2t )} \qquad \text{or} \qquad \sum_{n \in \Z^2} c_n g_n(\o) e^{i(\jb{n,x} \pm |n|t )}. 
\end{align*}

\noi	
for $\{c_n \}_{n \in \Z^2} \in \l^2(\Z^2)$.	
Then, for $2\le p <\infty$,  there exists $\dl_0>0$ and $c>0$ such that 
\begin{align*}
\PP\big\{    \| f^\o \|_{L^p([-\dl,\dl] \times \T^2) } > \|c_n \|_{\ell_n^2}  \big\} \le e^{- \frac 1{\dl^c} }.
\end{align*}	
	
\noi
for $\dl \le \dl_0$.
\end{lemma}

One way to prove Lemma \ref{LEM:prstricha} would be to directly apply the Wiener chaos estimate \ref{LEM:multigauss}. For the proof, see \cite{CO}.


\section{The construction of the Gibbs measure for the Zakharov-Yukawa system}
\label{SEC:conmea}

In this section, we present the proof of Theorem \ref{THM:1}. 
The main step  in proving Theorem \ref{THM:1} 
is to prove the following lemma (uniform exponential integrability \eqref{unexp1}).
We establish the bound \eqref{unexp1} by applying the variational formulation of the partition function by Barashkov-Gubinelli \cite{BG}.

\begin{lemma}
\label{LEM:uni2}
Let $0\le \g <1$. Then, given any finite $ p \ge 1$, 
$Q_N $  in \eqref{poten} converges to some limit $Q$ in $L^p(\mu \otimes \mu_{1-\g})$.
Moreover, there exists $C_{p} > 0$ such that 
\begin{equation}
\sup_{N\in \N} \Big\| \ind_{ \{|\int_{\T^2} : | u_N|^2 :   dx| \le K\}} e^{-Q_N(u,w)}\Big\|_{L^p(\mu \otimes \mu_{1-\g} )}
\leq C_{p}  < \infty. 
\label{uni2}
\end{equation}

\noi
In particular,  we have
\begin{equation}
\lim_{N\rightarrow\infty} \ind_{ \{|\int_{\T^2} : | u_N|^2 :   dx| \le K\}}  e^{ -Q_N(u,w)}= \ind_{ \{|\int_{\T^2} : | u|^2 :   dx| \le K\}} e^{-Q(u,w)}
\qquad \text{in } L^p(\mu \otimes \mu_{1-\g} ).
\label{Lpp}
\end{equation}

\end{lemma}

The convergence of $Q_N $  in \eqref{poten} follows
from a standard computation and thus we omit details.
See, for example, 
\cite[Lemma 2.5]{OTh}
for the related result.
As we pointed out, once the uniform bound \eqref{uni2} is established, the $L^p$-convergence \eqref{Lpp} of the densities follows from (softer) convergence in measure of the densities. See \cite[Remark 3.8]{Tz08}.

\begin{remark}\rm
We notice that the Gibbs measure $d \rhoo_\g$ on $ (u,w, \dt w)$, formally defined in \eqref{Gibbs}, decouples as the Gibbs measure $d \rho_\g$ \eqref{Gibbs3} on the component $(u,w)$ and the Gaussian measure $d \mu_{-\g}$ on the component $\dt w$.
Therefore, once the Gibbs measure $d\rho_\g$ \eqref{Gibbs3} on $(u,w)$ is established,  
we can construct the Gibbs measure $d\rhoo_\g$ on $(u,w,\dt w)$ by setting 
\begin{align}
d\rhoo_\g(u,w,\dt w) 
=  d ( \rho_\g \otimes \mu_{-\g})(u,w, \dt w). 
\label{Gibbs5}
\end{align}

\noi
Hence, in the following, we only discuss the construction 
of the (renormalized) Gibbs measure~$d\rho_\g$ on $(u,w)$, written in \eqref{Gibbs3}.
\end{remark}

\subsection{Stochastic variational formulation}
\label{SUBSEC:var}

We use a variational formula for the partition function
as in \cite{BG, OOT1, OST, OOT2, RSTW, Seo23}. The main idea is to write the partition function as a stochastic optimization problem over time-dependent processes.

We begin by representing the measure
$\mu_1 \otimes \mu_{1-\g} $ as the distribution of a pair of cylindrical processes at the time $1$.
Let $\vec{W}(t)$ be a cylindrical Brownian motion in $L^2(\T^2) \times L^2(\T^2)$.
Namely, we have
\begin{align*}
\vec{W}(t) = (W_1(t),W_2(t))= \bigg(\sum_{n \in \Z^2} B_1^n(t) e^{i\jb{n,x}}, \sum_{n \in \Z^2} B_2^n(t) e^{i\jb{n,x}} \bigg),
\end{align*}

\noi
where  
$\{B_1^n\}_{n \in \Z^2}$ and $\{B_2^n\}_{n \in \Z^2}$ are 
two sequences of mutually independent complex-valued\footnote{By convention, we normalize $B_n$ such that $\text{Var}(B_n(t)) = t$. In particular, $B_0$ is  a standard real-valued Brownian motion.} Brownian motions\footnote{While we keep the discussion only to the real-valued setting, 
the results also hold in the complex-valued setting. In the complex-valued setting, we use the Laguerre polynomial to define the Wick renormalization.} such that 
$\cj{B^n_j}= B_{j}^{-n}$, $n \in \Z^2$. 
Then, we define a centered Gaussian process $\vec{Y}(t)=(Y_1(t),Y_2(t))$
by 
\begin{align}
\vec{Y}(t)
= \Big(  \jb{\nb}^{-1}W_1(t), \jb{\nb}^{-1+\g} W_2(t) \Big)
\label{P2}
\end{align}

\noi
Note that we have 
\begin{align*}
\Law(\vec{Y}(1)) = \mu_1 \otimes \mu_{1-\g}, 
\end{align*}

\noi where $d\mu_1$ and $d\mu_{1-\g}$ are the (fractional) Gaussian fields in \eqref{gauss0}. 
By setting  $\vec{Y}_N = \pi_N \vec{Y} $, 
we have   $\Law(\vec{Y}_N(1)) = (\pi_N)_*(\mu_1 \otimes \mu_{1-\g})$, 
i.e.~the pushforward of $\mu_1 \otimes \mu_{1-\g}$ under $\pi_N$.
In particular, 
we have  $\E [Y_{1,N}^2(1)] = \s_N$,
where $\s_N$ is as in~\eqref{sigma1}.

Next, let $\vec{\mathbb{H}}_{a}$ denote the space of drifts, 
which are progressively measurable processes 
belonging to 
\begin{align*}
L^2([0,1]; L^2(\T^2) \times L^2(\T^2) ), \quad \text{$\PP$-almost surely.}
\end{align*}

\noi
We now state the  Bou\'e-Dupuis variational formula \cite{BD, Ust};
in particular, see Theorem 7 in \cite{Ust}.

\begin{lemma}\label{LEM:var3}
Let $\vec{Y}$ be as in \eqref{P2}.
Fix $N \in \N$.
Suppose that  $F:C^\infty(\T^2) \times C^\infty(\T^2) \to \R$
is measurable such that $\E\big[|F(\pi_N \vec{Y}(1))|^p\big] < \infty$
and $\E\big[|e^{-F(\pi_N \vec{Y}(1))}|^q \big] < \infty$ for some $1 < p, q < \infty$ with $\frac 1p + \frac 1q = 1$.
Then, we have
\begin{align*}
- \log \E\Big[e^{-F(\pi_N \vec{Y}(1))}\Big]
= \inf_{\vec{\dr} \in \vec{\mathbb{H}}_{a} }
\E\bigg[ F(\pi_N \vec{Y}(1) + \pi_N \vec{I}(\dr)(1)) + \frac{1}{2} \int_0^1 \| \vec{\dr}(t) \|_{L^2_x \times L^2_x}^2 dt \bigg], 
\end{align*}

\noi
where  $\vec{I}(\vec{\dr})=(I_1(\dr_1),I_2(\dr_2) )$ is  defined by 
\begin{align*}
\vec{I}(\vec{\dr})(t) =  \big(I_1(\dr_1)(t), I_2(\dr_2)(t) \big)   =\bigg( \int_0^t \jb{\nb}^{-1} \dr_1(t') dt', \int_0^t \jb{\nb}^{-1+\g} \dr_2(t') dt' \bigg)
\end{align*}

\noi
and the expectation $\E = \E_\PP$
is an 
expectation with respect to the underlying probability measure~$\PP$.

\end{lemma}


Before proceeding to the proof of Theorem \ref{THM:2}, 
we state a lemma on the  pathwise regularity bounds  of 
$\vec{Y}(1)$ and $\vec{I}(\vec{\dr})(1)$.

\begin{lemma}  \label{LEM:Dr}
\textup{(i)} 
Let $\eps > 0$. Then, given any finite $p \ge 1$, 
we have 
\begin{align*}
\begin{split}
\E \Big[ & \| \big( Y_{1,N}(1), Y_{2,N}(1) \big) \|_{W^{-\eps,\infty} \times W^{-\eps-\g,\infty} }^p
+ \|:\!  Y_{1,N}(1)^2\!:\|_{W^{-\eps,\infty}}^p
\Big]
\leq C_{\eps, p} <\infty,
\end{split}
\end{align*}

\noi
uniformly in $N \in \N$.
	
\smallskip
	
\noi
\textup{(ii)} For any $\vec{\dr} \in \vec{\mathbb{H}}_a$, we have
\begin{align}
\| \vec{I}(\vec{\dr})(1) \|_{H^1 \times H^{1-\g}}^2 \leq \int_0^1 \| \vec{\dr}(t) \|_{L^2_x \times L^2_x}^2dt.
\label{CM}
\end{align}

\noi
In particualr, we have
\begin{align*}
\| I_1( \dr_1)(1) \|_{H^1 }^2 \leq \int_0^1 \|  \dr_1(t) \|_{L^2_x}^2dt,\\
\| I_2( \dr_2)(1) \|_{H^{1-\g} }^2 \leq \int_0^1 \|  \dr_2(t) \|_{L^2_x}^2dt.
\end{align*}

\end{lemma}

Part (i) of Lemma \ref{LEM:Dr} follows
from a standard computation and thus we omit details.
See, for example, 
\cite[Proposition 2.3]{OTh2}
and 
\cite[Proposition 2.1]{GKO}
for related results when $d = 2$.
As for Part (ii), 
the estimate \eqref{CM}
follows  from Minkowski's and Cauchy-Schwarz' inequalities. 
See the proof of Lemma 4.7 in \cite{GOTW}.

\subsection{Uniform exponential integrability}

In this section, we present the proof of Lemma \ref{LEM:uni2}. 
Since the argument is identical for any finite $p \geq 1$, we only present details for the case $p =1$. 
Note that
\begin{align}
\ind_{\{|\,\cdot \,| \le K\}}(x) \le \exp\big( -  A |x|^\al \big) \exp(A K^\al )
\label{H6}
\end{align}

\noi
for any $K, A , \al > 0$.
Given $N \in \N$, $\al \gg 1$, and $A\gg 1$ sufficiently large, we define 
\begin{align}
\begin{split}
\RR_N (u,w)
&=  \int_{\T^2}   :\!  u_N ^2 :\!  w_N   dx 
+ A \,  \bigg| \int_{\T^2}  :\!  u_N^2 : \! dx \bigg|^\al.
\end{split}
\label{K2}
\end{align}

\noi 
Then, the following uniform exponential bound \eqref{u2} with \eqref{H6}
\begin{equation}
\sup_{N\in \NN} \Big\| 
e^{-\RR_N(u)}\Big\|_{L^p( d\muu_\g)}
\leq C_{p,A,\al}  < \infty
\label{u2}
\end{equation}

\noi
implies the uniform exponential integrability \eqref{uni2}. Hence, it remains to prove the uniform exponential integrability \eqref{u2}.
In view of the Bou\'e-Dupuis formula (Lemma \ref{LEM:var3}), 
it suffices to  establish a  lower bound on 
\begin{equation}
\W_N(\vec{\dr}) = \E
\bigg[\RR_N(\vec{Y}(1) + \vec{I}(\vec{\dr})(1)) + \frac{1}{2} \int_0^1 \| \vec{\dr}(t) \|_{L^2_x \times L^2_x}^2 dt \bigg], 
\label{v_N0}
\end{equation}

\noi 
uniformly in $N \in \N$ and  $\vec{\dr} \in \vec{\mathbb{H}}_a$.
We  set $\vec{Y}_N = \pi_N \vec{Y} = \pi_N \vec{Y}(1)$ and $\vec{\Dr}_N = \pi_N  \vec{\Dr} = \pi_N \vec{I}(\vec{\dr})(1)=( \Dr_{1,N}, \Dr_{2,N} )$.

From \eqref{K2}, we have
\begin{align}
\RR_N (\vec{Y} + \vec{\Dr})  & = 
\int_{\T^2}   :\! (Y_{1,N}+\Dr_{1,N})^2 :\!  (Y_{2,N}+\Dr_{2,N}) dx+ A\bigg( \int_{\T^2} :\! (Y_{1,N}+\Dr_{1,N})^2 :\! dx \bigg)^\al \notag \\
&=\int_{\T^2} :\! Y_{1,N}^2 :\! Y_{2,N} dx + \int_{\T^2} :\! Y_{1,N}^2 :\! \Dr_{2,N} dx + 2\int_{\T^2} Y_{1,N} \Dr_{1,N}Y_{2,N} dx \notag \\
&\hphantom{X}+ 2\int_{\T^2}   Y_{1,N} \Dr_{1,N} \Dr_{2,N} dx
+\int_{\T^2}  \Dr_{1,N}^2 Y_{2,N} dx + \int_{\T^2} \Dr_{1,N}^2 \Dr_{2,N}  dx \notag \\
&\hphantom{X}
+  A \bigg\{ \int_{\T^2} \Big(  :\! Y_{1,N}^2 :\!  + 2 Y_{1,N} \Dr_{1,N} + \Dr_{1,N}^2 \Big) dx \bigg\}^\al. 
\label{Y0}
\end{align}

\noi
Hence, from  \eqref{v_N0} and \eqref{Y0}, we have
\begin{align}
\W_N(\vec{\dr})
&=\E
\bigg[
\int_{\T^2} :\! Y_{1,N}^2 :\! Y_{2,N} dx +  \int_{\T^2} :\! Y_{1,N}^2 :\! \Dr_{2,N} dx + 2\int_{\T^2} Y_{1,N} \Dr_{1,N} Y_{2,N} dx \notag \\
&\hphantom{X}+ 2\int_{\T^2}   Y_{1,N} \Dr_{1,N}  \Dr_{2,N} dx
+ \int_{\T^2}  \Dr_{1,N}^2  Y_{2,N}  dx + \int_{\T^2}   \Dr_{1,N}^2  \Dr_{2,N} dx \notag \\
&\hphantom{X}
+  A \bigg\{ \int_{\T^2} \Big(  :\! Y_{1,N}^2 :\!  + 2 Y_{1,N} \Dr_{1,N} + \Dr_{1,N}^2 \Big) dx \bigg\}^\al
+ \frac{1}{2} \int_0^1 \| \vec{\dr}(t) \|_{L^2_x\times L^2_x}^2 dt 
\bigg].
\label{v_N0a}
\end{align}

We first state 
a lemma, controlling the terms appearing in \eqref{v_N0a}.
We present the proof of this lemma at the end of this section.

\begin{lemma} \label{LEM:Dr2}
	
Let $0\le\g<1$. Then, we have the following:

\noi
\textup{(i)} For any $\dl>0$, we have
\begin{align}
\E\bigg[ \int_{\T^2} :\! Y_{1,N}^2 :\! Y_{2,N}  dx \bigg]&=0 
\label{1} \\
\E \bigg[ \int_{\T^2} Y_{1,N} Y_{2,N}   \Dr_{1,N} dx \bigg] 
&\le c(\dl) + \dl  \E \Big[  \| \Dr_{1,N}\|_{ H^1}^2 \Big]  
\label{Y2} 
\end{align}

\noi
uniformly in $N \in \N$.

\smallskip
	
\noi
\textup{(ii)}
There exists a small $\eps>0$, a constant $c >0$ and $\al \gg 1$ such that
for any $\dl>0$, we have
\begin{align}
\bigg| \int_{\T^2}  :\! Y_{1,N}^2 \!:  \Dr_{2,N}dx  \bigg|
&\le c(\dl) \| :\! Y_{1,N}^2 \!: \|_{W^{-\eps,\infty}}^2  
+ \dl \| \Dr_{2,N} \|_{ H^{1-\g}}^2,   
\label{Y1} \\
\bigg| \int_{\T^2}   Y_{1,N} \Dr_{1,N}  \Dr_{2,N} dx \bigg| &\le c(\dl)\| Y_{1,N} \|_{W^{-\eps, \infty} }^{c}+\dl \Big( \|  \Dr_{1,N}  \|_{H^1}^2  +  \| \Dr_{1,N}\|_{L^2}^{2\al}   +\| \Dr_{2,N}  \|_{H^{1-\g}}^2  \Big)  \label{YY2}\\
\bigg| \int_{\T^2}   \Dr_{1,N}^2  Y_{2,N}  dx  \bigg| &\le  c(\dl) \|  Y_{2,N}  \|_{W^{-\g-\eps,\infty}}^c +\dl \Big( \|  \Dr_{1,N}  \|_{H^1}^2  +  \| \Dr_{1,N}\|_{L^2}^{2\al} \Big)   \label{YYY2}  \\
\bigg| \int_{\T^2}    \Dr_{1,N}^2 \Dr_{2,N}  dx \bigg| & \le  \frac A{100} \| \Dr_{1,N}  \|_{L^2}^{2\al} +\dl \Big( \| \Dr_{2,N} \|_{H^{1-\g}}^2 + \| \Dr_{2,N} \|_{L^2}^2+\| \Dr_{1,N} \|_{H^1}^2   \Big)
\label{Y3}
\end{align}

\noi  
for any sufficiently large $A>0$, uniformly in $N \in \N$.

\smallskip

\noi
\textup{(ii)}	
Let $A> 0$ and $\al>0$. Then,  
there exists $c = c(A,\al)>0$ such that
\begin{align}
\begin{split}
&A \bigg| \int_{\T^2} \Big(  :\! Y_{1,N}^2 :\!  + 2 Y_{1,N} \Dr_{1,N} + \Dr_{1,N}^2 \Big) dx \bigg|^\al\\
&\hphantom{XX}\ge \frac A4 \| \Dr_{1,N} \|_{L^2}^{2\al} 
- \frac 1{100} \|  \Dr_{1,N} \|_{H^1}^2   -c \bigg\{
\| Y_{1,N}  \|_{W^{-\eps,\infty}}^{\frac {2\al}{1-(\al-1)\eps}}  +\bigg| \int_{\T^2} :\!  Y_{1,N}^2 :\!  dx \bigg|^{\al} \bigg\}, 
\end{split}
\label{YY5}
\end{align}

\noi
uniformly in $N \in \N$.

\end{lemma}

Then, as in \cite{BG, GOTW, ORSW2, OOT1, OST, OOT2}, 
the main strategy is 
to establish a pathwise lower bound on $\W_N(\vec{\dr})$ in~\eqref{v_N0a}, 
uniformly in $N \in \N$ and $ \vec{\dr} \in \vec{\mathbb{H}}_a$, 
by making use of the 
positive terms:
\begin{equation}
\U_N(\vec{\dr}) =
\E \bigg[\frac A 4  \| \Dr_{1,N} \|_{L^2}^{2\al} + \frac{1}{2} \int_0^1 \| \vec{\dr}(t) \|_{L^2_x\times L^2_x }^2 dt\bigg].
\label{v_N1}
\end{equation}

\noi
coming from \eqref{v_N0a} and \eqref{YY5}.
From \eqref{v_N0a} and \eqref{v_N1} together with Lemmas  \ref{LEM:Dr2} and \ref{LEM:Dr}, we obtain
\begin{align}
\inf_{N \in \mathbb{N}} \inf_{ \vec{\dr} \in \vec{\mathbb{H}}_a} \W_N(\vec{\dr}) 
\geq 
\inf_{N \in \mathbb{N}} \inf_{ \vec{\dr} \in \vec{\mathbb{H}}_a}
\Big\{ -C_0 + \frac{1}{10}\U_N(\vec{\dr})\Big\}
\geq - C_0 >-\infty.
\label{YYY5}
\end{align}

\noi
Then,  
the uniform exponential integrability \eqref{uni2} 
follows from 
\eqref{YYY5} and Lemma \ref{LEM:var3}.
This completes the proof of  
Lemma \ref{LEM:uni2}.

\medskip

We conclude this section by 
presenting the proof of Lemma \ref{LEM:Dr2}.

\begin{proof}[Proof of Lemma \ref{LEM:Dr2}]

(i) From the independence of $Y_{1,N}$ and $Y_{2,N}$, we have
\begin{align*}
\E\bigg[ \int_{\T^2} :\! Y_{1,N}^2 :\! Y_{2,N}  dx \bigg]&=0.
\end{align*}

\noi
This yields \eqref{1}.

From H\"older's inequality and Young's inequality, we have
\begin{align*}
\bigg| \int_{\T^2}  Y_{2,N}  Y_{1,N}  \Dr_{1,N} dx \bigg| &\le \| \jb{\nb}^{-1}(Y_{1,N} Y_{2,N} )   \|_{L^2} \|  \jb{\nb} \Dr_{1,N} \|_{L^2}\\
&\les  C(\dl)\| \jb{\nb}^{-1}(Y_{1,N} Y_{2,N} )   \|_{L^2}^2+ \dl \|  \Dr_{1,N} \|_{H^1}^2.
\end{align*}

\noi
We now consider $\|  \jb{\nb}^{-1} (Y_{1,N} Y_{2,N})  \|_{L^2}^2 $. Note that
\begin{align*}
\|  \jb{\nb}^{-1} (Y_{1,N} Y_{2,N})  \|_{L^2}^2  =  \sum_{\substack{ n \in \Z^2, \\ |n|\le 2N   } } \frac 1{\jb{n}^2}  \bigg|  \sum_{\substack{ n_1 \in \Z^2, \\ |n_1| \le N, |n_2| \le N, \\ n_1+n_2=n   } }    \frac{ B_1^{n_1}(1)  B_2^{n_2}(1) }{  \jb{n_1} \jb{n_2}^{1-\g} }     \bigg|^2.
\end{align*}

\noi
From taking the expectation, using the independence of $\{B_1^n\}_{n \in \Z^2}$ and $\{B_2^n\}_{n \in \Z^2}$, and Lemma \ref{LEM:SUM}, we have
\begin{align*}
\E\Big[   \|  \jb{\nb}^{-1} (Y_{1,N} Y_{2,N})  \|_{L^2}^2  \Big]&\les \sum_{\substack{ n \in \Z^2, \\ |n|\le 2N   } } \frac 1{\jb{n}^2}  \sum_{\substack{ n_1 \in \Z^2, \\ |n_1| \le N, |n_2| \le N, \\ n_1+n_2=n   } } \frac 1{\jb{n_1}^2 \jb{n_2}^{2(1-\g)} }\\
&\les\sum_{\substack{ n \in \Z^2, \\ |n|\le 2N   } } \frac 1{\jb{n}^2}  \sum_{\substack{ n_1 \in \Z^2, \\ |n_1| \le N, |n_2| \le N, \\ n_1+n_2=n   } } \frac 1{\jb{n_1}^{2-\eta} \jb{n_2}^{2(1-\g)} }\\
&\les  \sum_{\substack{ n \in \Z^2, \\ |n|\le 2N   } } \frac 1{\jb{n}^2   \jb{n}^{2-2\g-\eta}}<\infty,
\end{align*}

\noi
uniformly in $N$, where we choose $\eta>0$ such that $2-2\g-\eta>0$ by using the condition $\g<1$. This yields \eqref{Y2}.

\medskip

\noi
(ii) 
From Young's inequality and the condition $\g<1$, we have 	
\begin{align*}
\bigg| \int_{\T^2}  :\! Y_{1,N}^2 \!:  \Dr_{2,N}dx  \bigg|&\le \| :\! Y_{1,N}^2 \!: \|_{W^{-\eps,\infty}}    \| \jb{\nb}^\eps \Dr_{2,N} \|_{ L^1}\\
&\le c(\dl) \| :\! Y_{1,N}^2 \!: \|_{W^{-\eps,\infty}}^2  
+ \dl \| \Dr_{2,N} \|_{ H^{1-\g}}^2.   
\end{align*}	

\noi
This yields \eqref{Y1}.

From the fractional Leibniz rule (Lemma \ref{LEM:prod} (ii)) and Sobolev's inequality, we have
\begin{align}
&\bigg| \int_{\T^2}   Y_{1,N} \Dr_{1,N}   \Dr_{2,N} dx \bigg| \notag \\
&\le \|  Y_{1,N}  \|_{W^{-\eps,\infty}} \|   \Dr_{1,N} \Dr_{2,N}  \|_{W^{\eps,1}}  \notag \\
& \le \|  Y_{1,N} \|_{W^{-\eps,\infty}} \bigg[ \| \Dr_{1,N}  \|_{L^{2+}}
\| \jb{\nb}^\eps \Dr_{2,N}  \|_{L^2}   +  \| \jb{\nb}^\eps \Dr_{1,N}  \|_{L^{2+}}   \|  \Dr_{2,N}  \|_{L^2}    \bigg]  \notag  \\
& \le \|  Y_{1,N}  \|_{W^{-\eps,\infty}} \bigg[  \| \Dr_{1,N}  \|_{H^\eta}
\|  \Dr_{2,N}  \|_{H^{1-\g}}   +    \|  \Dr_{1,N}  \|_{H^{\eta}}    \|  \Dr_{2,N}  \|_{L^2}   \bigg]
\label{dr2dr1}
\end{align}

\noi
where $0<\eps<1-\g$ and $0<\eps<\eta $ for some small $\eta>0$.
From the interpolation inequality (Lemma \ref{LEM:prod} (i)), we have
\begin{align}
\| \Dr_{1,N} \|_{H^\eta } \les \|  \Dr_{1,N}  \|_{H^1}^\eta \|  \Dr_{1,N}  \|_{L^2}^{1-\eta}.
\label{intep} 
\end{align}

\noi
Hence, from \eqref{intep} and Young's inequality, we have
\begin{align*}
\eqref{dr2dr1} &\le \|  Y_{1,N}  \|_{W^{-\eps,\infty}}  \|  \Dr_{2,N}  \|_{H^{1-\g}} \| \Dr_{1,N}  \|_{H^1}^\eta \| \Dr_{1,N}  \|_{L^2}^{1-\eta}\\
&\le c(\dl)\| Y_{1,N} \|_{W^{-\eps, \infty} }^{c}+\dl \Big( \|  \Dr_{1,N}  \|_{H^1}^2  +  \| \Dr_{1,N}\|_{L^2}^{2\al}   +\| \Dr_{2,N}  \|_{H^{1-\g}}^2  \Big),  
\end{align*}
	
\noi
where $\frac 12+\frac \eta{2}+\frac {1-\eta}{2\al}<1$ if $\al \gg 1$. This yields \eqref{YY2}.

From the fractional Leibniz rule (Lemma \ref{LEM:prod} (ii)) (with $\frac{1}{1+\dl_0}=\frac{1}{2+\dl_1}+\frac 12$), Sobolev's inequality (with $\frac 1{2+\dl_1}=\frac 12 -\frac \eta2$), and Young's inequality, we have	
\begin{align*}
\bigg| \int_{\T^2} Y_{2,N} \Dr_{1,N}^2 dx  \bigg|  &\le \|  Y_{2,N}  \|_{W^{-\g-\eps,\infty}} \|  \jb{\nb}^{\g+\eps}( \Dr_{1,N}^2 ) \|_{L^{1+\dl_0} }  \notag \\
&\le \|  Y_{2,N}  \|_{W^{-\g-\eps,\infty}}  \| \jb{\nb}^{\g+\eps} \Dr_{1,N}  \|_{L^{2+\dl_1}} \| \Dr_{1,N}  \|_{L^{2}} \\
&\le   \|  Y_{2,N}  \|_{W^{-\g-\eps,\infty}}  \| \jb{\nb}^{\g+\eps+\eta} \Dr_{1,N}  \|_{L^{2}} \| \Dr_{1,N}  \|_{L^{2}} \\ 
&\le|  Y_{2,N}  \|_{W^{-\g-\eps,\infty}}  \| \Dr_{1,N}  \|_{H^1} \| \Dr_{1,N}  \|_{L^{2}} \\ 
&\le c(\dl) \|  Y_{2,N}  \|_{W^{-\g-\eps,\infty}}^c +\dl \Big( \|  \Dr_{1,N}  \|_{H^1}^2  +  \| \Dr_{1,N}\|_{L^2}^{2\al} \Big)    
\end{align*}	

\noi
for some sufficiently small $\dl_0, \dl_1$, and $\eta>0$,
where $\frac 12+\frac 1{2\al}<1$ if $\al \gg 1$. Notice that in the fourth step, we used the condition $\g<1$. This yields \eqref{YYY2}.

From H\"older's inequality (with $1=\frac 1{2+\dl_0}+\frac 1{2-\dl_1}$), we have
\begin{align}
\bigg| \int_{\T^2} \Dr_{2,N} \Dr_{1,N}^2 dx \bigg| &\le \|  \Dr_{2,N} \|_{L^{2+\dl_0}} \| \Dr_{1,N}^2 \|_{L^{2-\dl_1}}\notag\\
& \le \|  \Dr_{2,N} \|_{L^{2+\dl_0}} \| \Dr_{1,N} \|_{L^{4-2\dl_1}}^2.
\label{Y01}
\end{align}	

\noi
We first consider $\| \Dr_{2,N} \|_{L^{2+\dl_0}}$.
From the Sobolev's inequality (with $\frac 1{2+\dl_0}=\frac 12 -\frac \eps2$) and the interpolation inequality (Lemma \ref{LEM:prod} (i)), we have
\begin{align}
\| \Dr_{2,N} \|_{L^{2+\dl_0}} \les \| \Dr_{2,N} \|_{H^\eps } \le \| \Dr_{2,N} \|_{H^{1-\g}}^{\frac \eps{1-\g}}  \| \Dr_{2,N} \|_{L^2}^{\frac {1-\g-\eps}{1-\g} }
\label{Y02}
\end{align}

\noi
where $0<\eps<1-\g$. 
We next consider $\| \Dr_{1,N} \|_{L^{4-2\dl_1} }$.
From the Sobolev's inequality (with $\frac 1{4-2\dl_1}=\frac 12 -\frac {1-\dl_1}{2(2-\dl_1)}$) and the interpolation inequality (Lemma \ref{LEM:prod} (i)), we have
\begin{align}
\| \Dr_{1,N} \|_{L^{4-2\dl_1} } \les \| \Dr_{1,N} \|_{H^{\frac{1-\dl_1}{2-\dl_1}}} \les \| \Dr_{1,N} \|_{H^1}^{\frac{1-\dl_1}{2-\dl_1}}  \| \Dr_{1,N} \|_{L^2}^{\frac 1{2-\dl_1}}.
\label{Y03}
\end{align}

\noi
From \eqref{Y01}, \eqref{Y02}, \eqref{Y03}, and Young's inequality, we have
\begin{align*}
\bigg| \int_{\T^2} \Dr_{2,N} \Dr_{1,N}^2 dx \bigg| 
& \le \|  \Dr_{2,N} \|_{L^{2+\dl_0}} \| \Dr_{1,N} \|_{L^{4-2\dl_1}}^2\\
& \le  \| \Dr_{2,N} \|_{H^{1-\g}}^{\frac \eps{1-\g}}  \| \Dr_{2,N} \|_{L^2}^{\frac {1-\g-\eps}{1-\g} }   \| \Dr_{1,N} \|_{H^1}^{\frac{2(1-\dl_1)}{2-\dl_1}}  \| \Dr_{1,N} \|_{L^2}^{\frac 2{2-\dl_1}}\\
&\le  \frac A{100} \| \Dr_{1,N}  \|_{L^2}^{2\al} +\dl \Big( \| \Dr_{2,N} \|_{H^{1-\g}}^2 + \| \Dr_{2,N} \|_{L^2}^2+\| \Dr_{1,N} \|_{H^1}^2   \Big)
\end{align*}

\noi
where
$ \frac \eps{2(1-\g)}+\frac {1-\g-\eps}{2(1-\g)}+\frac {1-\dl_1}{2-\dl_1}+\frac 1{(2-\dl_1)\al}=1 $ if $\al \gg 1$. Notie that $\frac {1-\dl_1}{2-\dl_1}<\frac 12$ if and only if $\dl_1>0$. This yields \eqref{Y3}.

\medskip
	
\noi
(iii) 
Note that there exists a constant $C_\al>0$ such that 
\begin{align}
|a+b+c|^\al
\ge \frac 12 |c|^\al - C_\g(|a|^\al+|b|^\al)
\label{OOT}
\end{align}
	
\noi
for any $a,b,c \in \R$ (see Lemma 5.8 in \cite{OOT1}). Then, from \eqref{OOT}, we have 
\begin{align}
\begin{split}
A \bigg| \int_{\T^2}& \Big(  :\! Y_{1,N}^2 :\! + 2  Y_{1,N} \Dr_{1,N}  + \Dr_{1,N}^2 \Big) dx \bigg|^{\al} \\
&\ge \frac A2 \bigg( \int_{\T^2} \Dr_{1,N}^2dx \bigg)^\al
- AC_{\al} \bigg\{ \bigg| \int_{\T^2} :\!  Y_{1,N}^2 :\!  dx \bigg|^{\al}
+ \bigg| \int_{\T^2} Y_{1,N} \Dr_{1,N} dx \bigg|^{\al} \bigg\}.
\end{split}
\label{YZ3}
\end{align}

\noi
From the interpolation inequality (Lemma \ref{LEM:prod} (i)) and Young's inequality, we have
\begin{align}
\begin{split}
\bigg| \int_{\T^2}  Y_{1,N} \Dr_{1,N}  dx \bigg|^\al
&\le \| Y_{1,N} \|_{W^{-\eps,\infty }}^\al \| \jb{\nb}^\eps \Dr_{1,N} \|_{L^2}^\al\\
& \le  \| Y_{1,N} \|_{W^{-\eps,\infty }}^\al  \|  \Dr_{1,N} \|_{L^2}^{\al(1-\eps)} \| \Dr_{1,N} \|_{H^1}^{\al\eps}\\ 
&\le  c \| Y_{1,N}  \|_{W^{-\eps,\infty}}^{\frac {2\al}{1-(\al-1)\eps}}+ \frac 1{100C_\al} \| \Dr_{1,N} \|_{L^2}^{2\al} +\frac 1{100C_\al A} \| \Dr_{1,N} \|_{H^1}^2 .
\end{split}
\label{YZ4}
\end{align}
	
\noi
Notice that $\frac {1-\eps}2+\frac {\al \eps}2<1$ if $\eps$ is sufficiently small.
Hence, \eqref{YY5} follows from \eqref{YZ3} and \eqref{YZ4}.
\end{proof}

\section{Probabilistic well-posedness}
\label{SEC:ga12}
In this section, we present the proof of Theorem \ref{THM:2} by assuming random tensor estimates (Lemmas \ref{LEM:zSzWy}, \ref{LEM:Tens1}, and \ref{LEM:Tens2})    in Section \ref{SEC:rant}.
In particular, we show that the renormalized Zakharov-Yukawa system \eqref{Zak2}
is $d\rhoo_\g$-almost surely locally well-posed.

\subsection{Reduced first-order system}
\label{SUBSEC:reduce}
In this subsection, we consider reduced first-order system for the wave part.
For the Zakharov-Yukawa system, there is a standard procedure to factor the wave operator in order to derive a first-order system. We first look at the renormalized Zakharov-Yukawa system:
\begin{align}
\begin{cases}
i \dt u +\Dl u = uw\\
\dt^2 w + (1- \Dl) w = -\jb{\nb}^{2\g} \NN(u) \\
(u,w,\dt w)|_{t=0}=(u^\o,w_0^\o, w_1^\o),
\end{cases}
\label{Zak4}
\end{align}

\noi
where we recall $\NN(u)=|u|^2-\fint |u|^2$ and $(u^\o,w_0^\o, w_1^\o)$ is distributed according to $d\muu_\g=d\mu \otimes d\mu_{1-\g} \otimes d\mu_{-\g} $.
It is convenient to reduce the renormalized Zakharov-Yukawa system \eqref{Zak4} to a first-order system
by setting $ w_{\pm}:=w \pm i\jb{\nb}^{-1} \dt w $, where $(u,w, \dt w)$ is a solution to \eqref{Zak4}. The new system is then given by
\begin{align}
\begin{cases}
i \dt u +\Dl u = \frac 12 (w_{+}+w_{-})u \\
i\dt  w_{\pm} \mp  \jb{\nb}  w_{\pm} = \pm  \jb{\nb}^{-1+2\g} \NN(u) \\
(u, w_{\pm})|_{t=0}=(u_0^\o,  w_{\pm,0}^\o),
\end{cases}
\label{reduced}
\end{align}

\noi
where $(u_0^\o, w_{\pm,0}^\o)$ is given by
\begin{equation*} 
u_0^\o = \sum_{n \in \Z^2 } \frac{ g_n(\o)}{\jb{n}} e^{i n\cdot x}
\qquad
\text{and}
\qquad
w_{\pm,0}^\o = \sum_{n\in \Z^2}  \frac {\wt h_{n,\pm}(\o)} {\jb{n}^{1-\g}} e^{i n\cdot x}, 
\end{equation*}

\noi
where $\{g_n\}_{n \in \Z^2} $ and $\{\wt h_{n,\pm}=h_n \pm i \l_n \}_{n\in \Z^2}$, see \eqref{IV2}.
Since $w$ is real-valued, we can recover the solution of the original system \eqref{Zak4}
by setting $w=\Re  w_{\pm}$ and this consideration allows us to claim the statement of Theorem \ref{THM:2} by
proving a similar statement about the reduced system \eqref{reduced}. 
The corresponding system can be expressed in Duhamel formulation:
\begin{align}
u(t)&=e^{it\Dl}u_0^\o-\frac i2 \int_0^t e^{i(t-t') \Dl }\big[(w_{+}+ w_{-} )u  \big](t') dt' \notag\\
w_{\pm}(t)&=e^{\mp it\jb{\nb}}w_{\pm, 0}^\o \mp i\int_0^t e^{-i(t-t')\jb{\nb}} \jb{\nb}^{-1+2\g}\NN(u) (t') dt'.
\label{Duha}
\end{align}

\noi
For further use, we set
\begin{align}
\NN_S(u_1, w_1):=u_1  w_1 \qquad \text{and} \qquad 
\NN_W(u_2, u_3):= \jb{\nb}^{-1+2\g} \NN(u_2,u_3)
\label{bi0}
\end{align}


\subsection{First order expansion}
\label{SUBSEC:firstor}
Let us first go over the basic idea of the probabilistic local well-posedness, as developed for instance in  \cite{BO96, BT1, Th1, CO, BOP2, NS}; see also \cite{McKean}.
This argument is based on the following first-order expansion: random linear term+smoother term 
\begin{align*}
u &= z^S + R^S,  \\
w_{\pm} &= z^{W_\pm} + R^{W_\pm}, 
\end{align*}

\noi
where $z^S=z^{S,\o}$ and $z^{W_\pm}=z^{W_{\pm} ,\o}$ denote the random linear solutions defined by 
\begin{align}
z^S(t)=z^{S,\o}(t) :&= e^{it\Dl} u_0^\o, \notag \\
z^{W_\pm}(t)=z^{W_\pm,\o}(t) :&=e^{\mp it\jb{\nb} }w_{\pm, 0}^\o.
\label{lin1}
\end{align}

\noi
By rewriting  \eqref{Duha} as a  fixed point problem
for the residual terms $R^S: = u - z^S$ and $R^{W_\pm}:=w_{\pm}-z^{W_\pm}$, we obtain
the following perturbed renormalized Zakharov-Yukawa system   :
\begin{align}
R^S(t) &=  - i \int_0^t e^{i(t-t') \Dl } \NN_S(z^S + R^S, z^{W_\pm}+R^{W_\pm})(t') dt', \notag \\
R^{W_\pm}(t) &=  \mp i \int_0^t e^{\mp i(t-t') \jb{\nb} }  \NN_W(z^S + R^S, z^S+R^S) (t') dt'.
\label{pZak}
\end{align}

\noi
By viewing $(z^S, z^{W_\pm}, z^S z^{W_\pm})$ as a given enhanced data set, we study the fixed point
problem for the smoother term $(R^S, R^{W_\pm})$ in  $H^{s}(\T^2) \times H^{\l}(\T^2)$ for some $s>0$ and $\l>0$ in Theorem \ref{THM:2}.
In particular, for $0\le \g <\frac 13$, we will show that for each small $\dl > 0$, there exists an event $\O_\dl \subset \O$
with $P(\O_\dl^c) < Ce^{-\frac{1}{\dl^c}}$
such that  for each $\o \in \O_\dl$, there exists
a solution $u = z^S+ R^S$ and $w_{\pm}=z^{W_{\pm} }+R^{W_{\pm}}$ to the perturbed renormalized Zakharov-Yukawa system   \eqref{pZak}  in the class: 
\begin{align*}
z^S + X_S^{s,\frac 12+ }(\dl)
&\subset  C([-\dl, \dl]; H^{-\eps}(\T^2)),\\
z^{W_\pm} + X_{W_\pm}^{\l ,\frac 12+}(\dl)
&\subset  C([-\dl, \dl]; H^{-\g -\eps}(\T^2)).
\end{align*}

\noi
for some $s >0$ and $\l>0$ in Theorem \ref{THM:2}.

\subsection{The proof of Theorem \ref{THM:2} }
\label{SUBSEC:proof2}
In this subsection,  we prove Theorem \ref{THM:2} by assuming Lemma \ref{LEM:bi2}.
We first define the following Duhamel map 
\begin{align*}
\G^\o(R^S,R^{W_{\pm}} )=\bigg( \int_0^t e^{i(t-t')\Dl} \fN_S^\o(R^S, R^{W_{\pm}} ) (t') d t'  , \mp  \int_0^t e^{\mp i(t-t')\jb{\nb}} \fN_{W}^\o(R^S, R^S) (t') d t'   \bigg)
\end{align*}

\noi
where
\begin{align*}
\fN_S^\o (R^S,R^{W_\pm}):&=\chi_\dl \cdot \NN_S(\wt z^S + R^S, \wt z^{W_\pm}+R^{W_\pm}) \notag \\
\fN_{W}^\o (R^S, R^{S} ):&=\chi_\dl  \cdot  \NN_W( \wt z^S + R^S,  \wt z^S + R^S) 
\end{align*}

\noi
with extensions $\wt z^S$ and $\wt z^{W_\pm}$  of the truncated random linear solution $\chi_{_\dl}\cdot z^S$ and $\chi_{_\dl}\cdot z^{W_\pm} $. 
Given $R^S$ and $R^{W_\pm}$ on $\T^2\times [-\dl, \dl]$, 
let $\wt R^S$ and $\wt R^{W_\pm}$ be extensions of $R^S$ and $R^{W_\pm}$ onto $\T^2 \times \R$.
By the non-homogeneous linear estimate (Lemma~\ref{LEM:lin}),  we have
\begin{align*}
\bigg\| \int_0^t e^{i(t-t')\Dl} \fN_S^\o(R^S, R^{W_\pm} ) (t') d t'\bigg\|_{X_S^{s, \frac{1}{2}+, }(\dl) }
& \leq \bigg\|\eta_{_\dl}(t) \int_0^t e^{i(t-t')\Dl } \fN_S^\o(\wt R^S, \wt R^{W_\pm}) (t') d t'\bigg\|_{X_S^{s, \frac{1}{2}+}} \notag \\
& \les \| \fN_S^\o(\wt R^S, \wt R^{W_\pm} ) \|_{X_S^{s, -\frac{1}{2}+} }, 
\end{align*}

\noi 
and
\begin{align}
\bigg\| \int_0^t e^{i(t-t')\jb{\nb}} \fN_W^\o(R^S, R^{S }) (t') d t'\bigg\|_{X_{W_\pm}^{\l, \frac{1}{2}+, }(\dl) }
& \leq \bigg\|\eta_{_\dl}(t) \int_0^t e^{-i(t-t')\jb{\nb} } \fN_W^\o(\wt R^S, \wt R^{S}) (t') d t'\bigg\|_{X_{W_\pm}^{\l, \frac{1}{2}+}} \notag \\
& \les \| \fN_W^\o(\wt R^S, \wt R^{S} ) \|_{X_{W_\pm}^{\l, -\frac{1}{2}+}}, 
\label{D2}
\end{align}

\noi
where  $\eta_{_\dl}$ is a smooth cutoff on $[-2\dl, 2\dl]$ 
as in \eqref{eta1}.
The main goal of next two subsections \ref{SUBSEC:bis}, \ref{SUBSEC:biw} is to prove the following bilinear estimates by viewing $(\wt z^S, \wt z^{W_\pm},  \wt z^S \wt z^{W_\pm})$ as a given (enhanced) data set.

\begin{lemma}
\label{LEM:bi2}
Let $ 0\le \g<\frac 13 $ and $\dl>0$. 
Then, there exists an event $\O_\dl \subset \O$ and $c'>0$  with $P(\Omega^c_\dl) < e^{-\frac{1}{\dl^{c'} }}$ such that 
\begin{align}
\| \fN_S^\o(\wt R^S, \wt R^{W_\pm} ) \|_{X_S^{s, -\frac{1}{2}+} } 
&\les \dl^{0+}
\Big( 1+  \| \wt R^S\|_{X_S^{s, \frac 12 + }} +  \|\wt R^{W_\pm}\|_{X_{W_\pm}^{\l, \frac 12 + }}  + \| \wt R^S\|_{X_S^{s, \frac 12 + }}\|\wt R^{W_\pm} \|_{X_{W_{\pm}}^{\l, \frac 12 + }}  \Big) 
\label{bilin11} \\
\| \fN_W^\o(\wt R^S, \wt R^{S} ) \|_{X_{W_{\pm}}^{\l, -\frac{1}{2}+} } 
&\les \dl^{0+}
\Big( 1+  \| \wt R^S\|_{X_S^{s, \frac 12 + }} +    \| \wt R^S\|_{X_S^{s, \frac 12 + }}^2   \Big)
\label{bilin22}  
\end{align}
	
\noi 
for all $\o \in \Omega_\dl$ 
and any  extension $(\wt R^S, \wt R^W) $ of $(R^S, R^W)$, provided
\begin{align*}
s-1<\l<1-2\g 
\end{align*}
	
\noi
with $0<s<\frac 14-\frac \g2$ and $\l>0$.
\end{lemma}


From the definition \eqref{localtime} of the local-in-time norm, 
we then conclude from \eqref{D2} and Lemma \ref{LEM:bi2} that 
\begin{align*}
&\bigg\| \int_0^t e^{i(t-t')\Dl} \fN_S^\o(R^S, R^{W_\pm} ) (t') d t'\bigg\|_{X_S^{s, \frac{1}{2}+ }(\dl) }\\ 
&\les \dl^{0+}
\Big( 1+  \|  R^S\|_{X_S^{s, \frac 12 + } (\dl)} +\| R^{W_\pm}\|_{X_{W_\pm }^{\l, \frac 12 + } (\dl)}  + \|  R^S\|_{X_S^{s, \frac 12 + } (\dl)} \| R^{W_\pm} \|_{X_{W_\pm}^{\l, \frac 12 + } (\dl)}\Big)\\
\intertext{and}
&\bigg\| \int_0^t e^{i(t-t')\jb{\nb}} \fN_{W}^\o(R^S, R^{S} ) (t') d t'\bigg\|_{X_{W_\pm}^{\l, \frac{1}{2}+ }(\dl) }\les  \dl^{0+}
\Big( 1+  \|  R^S\|_{X_S^{s, \frac 12 + }(\dl)}   + \|  R^S\|_{X_S^{s, \frac 12 + }(\dl)}^2   \Big).
\end{align*}

\noi
By the bilinear structure of the nonlinearity $\fN_S^\o$ and $\fN_W^\o$, a similar estimate holds for the difference of the Duhamel map  $\G^\o (R_1^S, R_1^{W_\pm}) - \G^\o (R_2^S, R_2^{W_\pm} )$, which allows us to conclude that  $\G^\o$ is a contraction on $B_S(1) \times B_{W_\pm}(1) \subset X_S^{s, \frac 12+}(\dl) \times X_{W_\pm}^{\l, \frac 12+}(\dl)$ for $\o \in \O_\dl$.


\subsection{Bilinear estimates for the Schr\"odinger part}
\label{SUBSEC:bis}
In this subsection, we prove \eqref{bilin11} in Lemma \ref{LEM:bi2}.
In view of \eqref{bi0},  in order to prove \eqref{bilin11}, we need to  carry  out case-by-case analysis on  
\begin{align}
\|\chi_{_\dl} \cdot \mathcal{N}_S(u_1, w_1)\|_{X_S^{s, -\frac{1}{2}+}} 
\label{bi2}
\end{align}


\noi
where $u_1$ and $w_1$ are taken to be either of type
\begin{itemize}
\item[(I)] rough random parts: 
\begin{align*}
u_1  &= \wt z^{S}\text{, where $\wt z^S$ is {\it some} extension of }
\chi_{_\dl}\cdot  z^S, \notag \\
w_1  &= \wt z^{W_\pm} \text{, where $\wt z^{W_\pm}$ is {\it some} extension of }
\chi_{_\dl}\cdot  z^{W_\pm},
\end{align*}	
	
\noi
where $z^S$ and $z^{W_\pm}$ denote the random linear solutions defined in \eqref{lin1}, 
	
\vspace{1mm}
	
\item[$(\II)$]	 smoother  `deterministic' remainder (nonlinear) parts: 
\begin{align*}
u_1 &= \wt{R}^S \text{, where  $\wt{R}^S$ is {\it any} extension of $R^S$,}\\
w_1 &= \wt{R}^{W_\pm} \text{, where  $\wt{R}^{W_\pm}$ is {\it any} extension of $R^{W_\pm}$}.
\end{align*}
\end{itemize}

\noi
In the following, when $u_j$ and $w_j$ are of type (I),
we take $\wt z^S = \eta_{_\dl} z^S$ and $\wt z^{W_\pm} = \eta_{_\dl} z^{W_\pm}$.
Thanks to the cutoff function in \eqref{bi2},
we may take $u_j = \eta_{_\dl}\cdot \wt R_j^S $ and $w_j = \eta_{_\dl}\cdot \wt R^{W_\pm}$ in \eqref{bi2}
when $u_j$ and $w_j$ are of type $(\II)$.

\begin{remark} \rm
In the following, we drop the $\pm$ signs and work with one $w_+$ or $w_{-}$ since there is no role of $\pm$. Hence, we set $w:=w_{+}$ and $W:=W_{+}$. 

\end{remark}

\begin{remark}\rm
\label{REM:not}
To estimate $\|\chi_{_\dl} \cdot \mathcal{N}_S(u_1, w_1)\|_{X_S^{s, -\frac{1}{2}+}} $, we need to perform case-by-case analysis of expressions of the form:
\begin{align*}
\int_{\R}\int_{\T^2} \jb{\nb}^s(u_1w_1) \cj v^S dxdt, \qquad \text{where $\|v^{S} \|_{X_S^{0,\frac 12-}}\le 1$. }
\end{align*}

\noi
In the following, for simplicity of notation, we drop the complex conjugate sign and suppress the smooth time-cutoff function $\eta_\dl$; and thus we simply denote them by $z^S, z^W, R^S$, and $R^W$, respectively when there is no confusion. Finally, we dyadically decompose $u_1, w_1$ and $v^S$ such that their spatial frequency supports are $\supp \ft u_1 \subset \{ |n_1| \sim N_1 \}$, $\supp \ft w_1 \subset \{ |n_2| \sim N_2 \}$, and $\supp \ft v \subset \{ |n| \sim N \}$  for some dyadic $N_1,N_2$ and $N \ge 1$.  	
\end{remark}


We now prove Lemmas \ref{LEM:RSRW}, \ref{LEM:zSzW}, \ref{LEM:zSRW}, \ref{LEM:RszW}, which will imply \eqref{bilin11} in Lemma \ref{LEM:bi2} (bilinear estimates for the Schr\"odinger part).
\begin{lemma}[$R^S R^W$-case]
\label{LEM:RSRW}
Let $s>0$ and $\l>0$. Then, we have
\begin{align*}
\|\NN_S(R^S, R^W) \|_{X_{S,\dl}^{s,-\frac 12+}} \les \dl^{0+} \|R^S \|_{X_{S,\dl}^{s,\frac 12+} } \|R^W \|_{X_{W_\pm,\dl}^{\l,\frac 12+} }.
\end{align*}

\end{lemma}

\begin{proof}
We perform the case-by-case analysis:	

\smallskip	
	
\noi
{\bf Case 1:} $N_1 \gg N_2$.

\noi
By writing $\{|n_1| \sim N_1 \}=\bigcup_{\ell_1} J_{1\ell_1}$ and $\{ |n| \sim N \}=\bigcup_{\ell_2}J_{2\ell_2}$, where $J_{1,\ell_1}$ and $J_{2,\ell_2}$ are balls of radius $\sim N_2$, we can decompose $\ft {P_{N_1}R^S}$ and $\ft {P_{N}v^S}$ as
\begin{align*}
\ft {P_{N_1}R^S}=\sum_{\ell_1} \ft { P_{N_1,\ell_1}R^S } \qquad \text{and} \qquad \ft {   P_{N}v^S }=\sum_{\ell_2} \ft { P_{N,\ell_2}v^S }
\end{align*} 

\noi
where $\ft {P_{N_1,\ell_1}R^S }(n_1,t)=\ind_{J_{1\ell_1}}(n_1) \ft { P_{N_1}R^S }(n_1,t) $ 
and $\ft {P_{N,\ell_2}v^S }(n,t)=\ind_{J_{2\ell_2}}(n) \ft { P_{N}v^S }(n,t) $. Given $n_1 \in J_{1\ell_1}$ for some $\ell_1$, there exists $O(1)$ many possible values for $\ell_2=\ell_2(\ell_1)$ such that $n \in J_{2\ell_2}$ under $n_1+n_2=n$. Notice that the number of possible values of $\ell_2$ is independent of $\ell_1$.

From the $L^4$-Strichartz estimate (Lemma \ref{LEM:L4}), the Cauchy-Schwarz inequality in $\l_1$, and Lemma \ref{LEM:dlpower}, we have\footnote{Here, we are assuming that $R^S, R^W$, and $v^S$ have non-negative
Fourier coefficients since the Bourgain spaces are $L^2$-based space.}
\begin{equation*}
\begin{split}
\bigg|\int_{\R}\int_{\T^2} \jb{\nb}^s(P_{N_1} R^S P_{N_2} R^W)  P_{N}v^S dxdt \bigg| &\les \sum_{\ell_1} \sum_{\ell_2} \int_{\R}\int_{\T^2} P_{N_1,\ell_1} (\jb{\nb}^sR^S) P_{N_2}R^W    P_{N,\ell_2}v^Sdxdt  \\
&\les  \sum_{\ell_1} \sum_{\ell_2}   \|P_{N_1,\ell_1} \jb{\nb}^s R^S  \|_{L^4_{t,x}} \| P_{N, \ell_2}v^S\|_{L^4_{t,x}} \| P_{N_2} R^W \|_{L^2_{t,x}}  \\
&\les\sum_{\ell_1}\sum_{\ell_2}  N_2^{\eps} \|  P_{N_1,\ell_1} R^S  \|_{X_S^{s,\frac 12-} }  \|P_{N,\ell_2}v^S \|_{X_S^{0,\frac 12-} }     \|P_{N_2} R^W \|_{X_{W_\pm}^{0,0} }  \\
& \les \sum_{\ell_1}\sum_{\ell_2} N_2^{-\l+\eps} \|  P_{N_1,\ell_1} R^S  \|_{X_S^{s,\frac 12-} }  \|P_{N,\ell_2}v^S \|_{X_S^{0,\frac 12-} }     \|P_{N_2} R^W \|_{X_{W_\pm}^{\l,0} }  \\
&\les \dl^{\frac 12-} N_2^{-\l+\eps}  \| P_{N_1} R^S  \|_{X_S^{s,\frac 12-} } \|P_{N}v^S \|_{X_S^{0,\frac 12-} } \|P_{N_2} R^W \|_{X_{W_\pm}^{\l,\frac 12-} }.
\end{split}
\end{equation*}

\noi
Hence, if $\l>0$, then we can do the dyadic summation over over $N_1 \sim N\ge N_2$.

\smallskip

\noi
{\bf Case 2:} $N_1 \ll N_2$ (non-resonant interaction).

\smallskip

\noi
This interaction includes the high-low interactions where one Schr\"odinger frequency is much greater than the other Schr\"odinger frequency. Hence, it follows from Lemma \ref{LEM:multi}  and \ref{LEM:dlpower} that we have 
\begin{align*}
&\|\NN_S(R^S, R^W) \|_{X_{S}^{s,-\frac 12+}} \\
& \les \sum_{L, L_1, L_2} N_2^s   L^{-\frac 12 +}         L_{\max}^{\frac 12} L_{\min}^\frac 38 L_{\med}^\frac 38    N_1^{\frac 12}  N_2^{-1}      \| P_{N_1,L_1} R^S \|_{ L^2_{t,x} }  \| P_{N_2,L_2} R^W \|_{L^2_{t,x} }  \|P_{N,L}  v^S\|_{L^2_{t,x}}     \\
& \les \dl^{0+} N_2^{s-\l-1} N_1^{\frac 12-s}   \| P_{N_1} R^S \|_{ X_{S}^{s,\frac 12+} } \| P_{N_2} R^W \|_{ X_{W_\pm}^{\l,\frac 12+}  }  \|P_{N}  v^S\|_{X_S^{0, \frac 12-}}   . 
\end{align*}

\noi
Hence, if $(s-\l-1)+\frac 12-s=-\l-\frac 12 < 0$ (i.e.  $\l > -\frac 12 $), we can perform the dyadic summation over $N_2 \sim N \ge N_1$.

\medskip

\noi
{\bf Case 3:} $N_1 \sim N_2 \ges N$ .

\noi
From $L^4$-Strichartz estimate (Lemma \ref{LEM:L4}) and Lemma \ref{LEM:dlpower}, we have
\begin{align*}
&\bigg|\int_{\R}\int_{\T^2} \jb{\nb}^s(P_{N_1} R^S P_{N_2} R^W)  P_{N}v^S dxdt \bigg|\\
&\les \|\jb{\nb}^s P_{N_1}R^S \|_{L^4_{t,x}} \| P_{N_2}R^W  \|_{L^2_{t,x}} \| P_Nv^S \|_{L^4_{t,x}}\\&\les \dl^{\frac 12-} N_1^{-\l+\eps} \| P_{N_1}R^S \|_{X_S^{s,\frac 12-}} \| P_{N_2}R^W \|_{X_{W_\pm}^{\l,\frac 12-}}   \| P_N v^S \|_{X_S^{0,\frac 12-}}.  
\end{align*}

\noi
If $\l \ge \eps$ (i.e. $\l>0$), then we can obtain the desired result by summing over $N_1 \sim N_2\ges N$.

\end{proof}

\begin{remark}\rm
In Case 2, if we proceed as in Case 1 (i.e. only using the $L^4$-Strichartz estimate with the orthogonality argument), then a restriction $\l \ge s$ happens.	
\end{remark}

\begin{lemma}[$z^S z^W$-case]
\label{LEM:zSzW}
Let $0<s<\frac 14-\frac {\g}2$. 
Then, for each small $\dl>0$, we have
\begin{align*}
\|\NN_S(z^S, z^W) \|_{X_{S,\dl}^{s,-\frac 12+}} \les \dl^{0+}
\end{align*}

\noi
outside an exceptional set of probability $< e^{-\frac 1{\dl^c}}$.
\end{lemma}

\begin{proof}
We perform the case-by-case analysis:	

\smallskip

\noi
{\bf   Case 1:} $N_1 \gg N$ or $N \gg N_1$  (non-resonant interaction). 
\smallskip

\noi
By symmetry, we only consider the case $N_1 \gg N$. 
In this interaction, we have
\begin{align}
L_{\max} \ges \big|  |n_1|^2 \pm |n_2|-|n|^2     \big| \gtrsim N_{\max}^2 \sim N_2^2.
\label{LL1}
\end{align}

\noi
Define 
\begin{align*}
L_{1}:=\jb{\tau_1-|n_1|^2},\; L_2:=\jb{\tau_2 \pm |n_2|}, \quad \text{and} \quad L:=\jb{\tau-|n|^2}.
\end{align*}

\noi
First, suppose that $L \sim L_{\max}$. Then, from $L^p_{t,x}L_{t,x}^{2+}L_{t,x}^2$-H\"older's inequality, Lemma \ref{LEM:prstricha} with large $p$, \eqref{LL1}, and Lemma \ref{LEM:dlpower}, we have
\begin{align*}
\bigg |\int_{\R}\int_{\T^2} \jb{\nb}^s(P_{N_1} z^S P_{N_2} z^W)  P_{N}v^S dxdt \bigg| &\les N_2^s \| P_{N_1}z^S \|_{L^p} \| P_{N_2}z^W \|_{L^{2+}} \|P_{N}v^S \|_{L^{2}_{t,x}} \\
&\les L_{\max}^{-\frac 12+} N_2^{\g+s} \| P_{N}v^S \|_{X_{S,}^{0,\frac 12-}}\\
& \les \dl^{0+} N_2^{-1+\g+s+} \| P_{N}v^S \|_{X_{S}^{0,\frac 12-}}
\end{align*}

\noi
outside an exceptional set of probability $< e^{-\frac 1{\dl^c}}$.
If $s < 1-\g$, then we obtain the desired result by summing over $N_1 \sim N_2 \ge N$. We point out that when $N_1\sim N_2\sim 1$, it is possible that $\big|   |n_1|^2 \pm |n_2| -|n|^2 \big| \ll 1$ but still $L_{\max} \ges N_{\max}^2 \sim N_2^2$ is true.

Next, suppose that $L\ll L_{\max}$ and so $\max\{L_1,L_2 \} \sim L_{\max}$. We may assume $L_1 \sim L_{\max}$. Then, we have
\begin{align}
\big|  \ft \eta_\dl(\tau_1-|n_1|^2) \big| \les \frac 1{L_1}\sim N_{\max}^{-2}
\label{highmodu12}
\end{align}

\noi
since $ \ft \eta_\dl(\tau)=\dl \ft \eta (\dl \tau)$. Then, from H\"older's inequality with $p \gg 1$, \eqref{highmodu12}, Young's inequality in $\tau$, and Lemma \ref{LEM:prob}, we have
\begin{align*}
&\bigg |\int_{\R}\int_{\T^2} \jb{\nb}^s(P_{N_1} z^S P_{N_2} z^W)  P_{N}v^S dxdt \bigg|\\
&= \Bigg| \sum_{\substack{n_1, n \in \Z^2: n_1+n_2=n \\ |n_1|\sim N_1, |n_2| \sim N_2 } } \jb{n}^s \frac{g_{n_1}(\o)}{\jb{n_1}} \frac {h_{n_2}(\o)}{\jb{n_2}^{1-\g}} \int_{\tau_1+\tau_2=\tau} \ft \eta_\dl (\tau_1-|n_1|^2) \ft \eta_\dl(\tau_2\pm |n_2|) \ft{P_Nv^S}(n,\tau) d\tau_1 d\tau \Bigg|\\
&\les \| \ft {P_N  v^S } \|_{\ell_n^2L_\tau^{1+}} \Bigg\| \sum_{\substack{n_1 \in \Z^2: n_1+n_2=n \\ |n_1|\sim N_1, |n_2| \sim N_2 } } \jb{n}^s \frac{|g_{n_1}(\o)|}{\jb{n_1}} \frac {|h_{n_2}(\o)|}{\jb{n_2}^{1-\g}} N_2^{-2(1-\eps)} \big\| |\ft \eta_\dl |^\eps   \big\|_{L^p_\tau}  \|\ft \eta_\dl \|_{L^1_\tau} \Bigg\|_{\l_n^2} \\
&\les  \dl^{\eps-\frac 1p-}  N^s N_1^{-1+}N_2^{-1+\g+}N_2^{-2+} N_1^2 N  \| P_N v^S \|_{X_S^{0,\frac 12-}}\\
&\les \dl^{0+} N_2^{-1+s+\g+} \| P_N v^S \|_{X_S^{0,\frac 12-}}
\end{align*}

\noi
outside an exceptional set of probability $< e^{-\frac 1{\dl^c}}$. Hence, if $s < 1-\g$, then we can perform the dyadic summation over $N_1 \sim N_2 \ge N$. The case $L_2 \sim L_{\max}$ follows from the same argument in the case $L_1\sim L_{\max}$.

\medskip

\noi
{\bf   Case 2:} $N_2 \les N_1\sim N$ (resonant interaction).

\smallskip

\noi We split the case into the high and low modulation cases.

\smallskip

\noi
{\bf   Subcase 2.a:} $L_{\max}\ges N_1^{2s+2\g+}$ (high modulation case). 
\smallskip

\noi
First, suppose that $L \sim L_{\max}$. Then, from $L^p_{t,x}L_{t,x}^{2+}L_{t,x}^2$-H\"older's inequality, Lemma \ref{LEM:prstricha} with large $p$, and Lemma \ref{LEM:dlpower}, we have
\begin{align*}
\bigg |\int_{\R}\int_{\T^2} \jb{\nb}^s(P_{N_1}z^S P_{N_2}z^W) P_Nv^S dxdt \bigg|&\les N_1^s \| P_{N_1}z^S \|_{L^p_{t,x}}  \| P_{N_2}z^W \|_{L^{2+}_{t,x}} \| P_N v^S \|_{L_{t,x}^2}\\
&\les N_1^{s+\g+}  \| P_N v^S \|_{L^2_{t,x}}\\
&\les \dl^{\frac 12-} N_1^{0-}  \| P_N v^S \|_{X_S^{0,\frac 12-}}
\end{align*}

\noi
outside an exceptional set of measure $< e^{-\frac 1{\dl^c}}$.
Hence, we can perform the dyadic summation over $N_1 \sim N\ge N_2$.

Next, suppose that $L\ll L_{\max}$ and so $\max\{L_1,L_2 \} \sim L_{\max}$. We may assume $L_1 \sim L_{\max}$. Then, we have
\begin{align}
\big|  \ft \eta_\dl(\tau_1-|n_1|^2) \big| \les \frac 1{L_1}\sim N_{\max}^{-2s-2\g-}.
\label{highmodu13}
\end{align}

\noi
We note that
\begin{align}
&\bigg |\int_{\R}\int_{\T^2} \jb{\nb}^s(P_{N_1}z^S P_{N_2}z^W) P_Nv^S dxdt \bigg| \notag \\
&= \bigg| \int_{\tau \in \R} \sum_{\substack{n\in \Z^2: \\ |n|\sim N}} \jb{n}^s \ft{P_Nv^S}(n,\tau)\Big( \sum_{\substack{n_2\in \Z^2: \\ n_1+n_2=n,\\ |n_1|\sim N_1, |n_2|\sim N_2}} a_{n_1,n_2,n}(\tau) g_{n_1}(\o) h_{n_2}(\o)  \Big) d\tau \bigg|
\label{highreso1}
\end{align}

\noi
where
\begin{align*}
a_{n_1,n_2,n}(\tau)=\ind_{\{n_1+n_2=n \}}\jb{n_1}^{-1}\jb{n_2}^{-1+\g}\int_{\tau_1+\tau_2=\tau }\ft \eta_\dl(\tau_1-|n_1|^2) \ft\eta_\dl(\tau_2\pm|n_2|) d\tau_1.
\end{align*}

\noi
Then, from Lemma \ref{LEM:multigauss}, Minkowski's inequality in $\tau$ (with $p\gg 1$), \eqref{highmodu13}, and Young's inequality, we have
\begin{align}
\Big\|  \sum_{\substack{n_2\in \Z^2: \\ n_1+n_2=n,\\ |n_1|\sim N_1, |n_2|\sim N_2}} a_{n_1,n_2,n}(\tau) \cdot  g_{n_1}(\o) h_{n_2}(\o)  \Big\|_{L^p_\tau} &\les \dl^{0-}N_1^{0+} \Big(\sum_{\substack{n_2\in \Z^2: \\ n_1+n_2=n,\\ |n_1|\sim N_1, |n_2|\sim N_2}}\|a_{n_1,n_2,n}(\tau)\|_{L^p_\tau}^2 \Big)^\frac 12 \notag \\
&\les \dl^{0-}N_1^{-1+}N_2^{-1+\g} \notag \\
&\ \hphantom{X}\times  \Big( \sum_{\substack{n_2\in \Z^2:|n_2|\sim N_2}} N_1^{-2(1-\eps)(2s+2\g+)} \big\| |\ft \eta_\dl|^\eps \|_{L^p_\tau}^2    \| \ft \eta_\dl \|_{L^1_\tau}^2     \Big)^\frac 12 \notag\\
&\les \dl^{\eps-\frac 1p-}N_1^{-1+}N_2^{-1+\g}N_1^{-2s-2\g+}N_2.
\label{gausscan1}
\end{align}

\noi
In \eqref{gausscan1}, we need to make sure that the probability $e^{-c'\frac{N_1^\eps}{\dl^c}}$ of the exceptional sets corresponding to different dyadic blocks and different values of $n_2$ should be summable and bounded by $e^{-\frac 1{\dl^c}}$ i.e. \eqref{gausscan1} holds outside an exceptional set of measure:
\begin{align*}
\sum_{N_1} N_2^2 e^{-\frac{c'N_1^\eps}{\dl^c}}\les e^{-\frac 1{\dl^c}}.
\end{align*}

\noi
From \eqref{highreso1}, H\"older's inequality in $\tau$, \eqref{gausscan1}, and  Cauchy-Schwarz inequality in $n$, we have
\begin{align*}
\text{LHS of \eqref{highreso1}} &\les \sum_{\substack{n\in \Z^2\\ |n|\sim N }} N^s \|\ft{P_Nv^S} \|_{L_\tau^{1+}}  \Big( \sum_{\substack{n_2\in \Z^2: \\ n_1+n_2=n,\\ |n_1|\sim N_1, |n_2|\sim N_2}} \|a_{n_1,n_2,n}(\tau)\|_{L^p_\tau}^2 \Big)^\frac 12 \notag \\
&\les \dl^{\eps-\frac 1p-}N_1^{-1+}N_2^{-1+\g} N^s N_1^{-2s-2\g+} N_2N \| P_Nv^S\|_{X_S^{0,\frac 12-}}  \notag \\
&\les \dl^{\eps-\frac 1p-}N_1^{-s-\g+}  \| P_Nv^S\|_{X_S^{0,\frac 12-}}.
\end{align*}

\noi
Hence, we can perform the dyadic summation over $N_1 \sim N\ge N_2$ if $s+\g>0$, which holds when $s>0$.

\smallskip

\noi
{\bf   Subcase 2.b:} $L_{\max}\les N_1^{2s+2\g+}$ (low modulation case). 
\smallskip

\noi
We split the case into $N_1 \sim N \sim N_2$ and $N_1\sim N \gg N_2$.

\smallskip

\noi
{\bf   Subsubcase 2.b.\textup{(i)}:} $N_1\sim N\sim N_2$.
\smallskip

\noi
We note that
\begin{align}
&\bigg |\int_{\R}\int_{\T^2} \jb{\nb}^s(P_{N_1}z^S P_{N_2}z^W) P_Nv^S dxdt \bigg| \notag \\
&= \bigg| \int_{\tau \in \R} \sum_{\substack{n\in \Z^2: \\ |n|\sim N}} \jb{n}^s \ft{P_Nv^S}(n,\tau)\Big( \sum_{\substack{n_2\in \Z^2: \\ n_1+n_2=n,\\ |n_1|\sim N_1, |n_2|\sim N_2}} a_{n_1,n_2,n}(\tau) g_{n_1}(\o) h_{n_2}(\o)  \Big) d\tau \bigg|
\label{highreso11}
\end{align}

\noi
where
\begin{align*}
a_{n_1,n_2,n}(\tau)=\ind_{\{n_1+n_2=n \}}\jb{n_1}^{-1}\jb{n_2}^{-1+\g}\int_{\tau_1+\tau_2=\tau }\ft \eta_\dl(\tau_1-|n_1|^2) \ft\eta_\dl(\tau_2\pm|n_2|) d\tau_1.
\end{align*}

\noi
Then, from Lemma \ref{LEM:multigauss}, Minkowski's inequality in $\tau$ (with $p\gg 1$), and Young's inequality, we have
\begin{align}
&\Big\|  \sum_{\substack{n_1: n_1+n_2=n,\\ |n_1|^2\pm |n_2|-|n|^2=O(N_1^{2s+2\g+}),\\ |n_1|\sim N_1, |n_2|\sim N_2   }} a_{n_1,n_2,n}(\tau) \cdot  g_{n_1}(\o) h_{n_2}(\o)  \Big\|_{L^p_\tau} \notag \\
&\les \dl^{0-}N_1^{0+} \bigg(\sum_{\substack{n_1: n_1+n_2=n,\\ |n_1|^2\pm |n_2|-|n|^2=O(N_1^{2s+2\g+}),\\ |n_1|\sim N_1, |n_2|\sim N_2   }}\|a_{n_1,n_2,n}(\tau)\|_{L^p_\tau}^2 \bigg)^\frac 12 \notag \\
&\les \dl^{0-}N_1^{-1+}N_2^{-1+\g}   \bigg( \sum_{\substack{n_1: n_1+n_2=n,\\ |n_1|^2\pm |n_2|-|n|^2=O(N_1^{2s+2\g+}),\\ |n_1|\sim N_1, |n_2|\sim N_2   }}  \big\| \ft \eta_\dl \|_{L^p_\tau}^2    \| \ft \eta_\dl \|_{L^1_\tau}^2     \bigg)^\frac 12 \notag\\
&\les \dl^{1-\frac 1p-}N_1^{-1+}N_2^{-1+\g} \Big (\#S_n \Big)^\frac 12
\label{gausscan113}
\end{align}

\noi
outside an exceptional set of probability $<  e^{-c'\frac{N_1^\eps}{\dl^c}}$, where
\begin{align*}
S_{n}:=\Big\{ n_1\in \Z^2: |n_1|^2\pm |n-n_1|-|n|^2=O(N_1^{2s+2\g+}), \; |n_1|\sim N_1, |n-n_1| \sim N_2, \; \text{and} \; |n|\sim N    \Big\}.
\end{align*}

\noi
Notice that in \eqref{gausscan113} we need to make sure that the probability $e^{-c'\frac{N_1^\eps}{\dl^c}}$ of the exceptional sets corresponding to different dyadic blocks and different values of $n_2$ should be summable and bounded by $e^{-\frac 1{\dl^c}}$ i.e. \eqref{gausscan113} holds outside an exceptional set of measure:
\begin{align*}
\sum_{N_1} N_2^2 e^{-\frac{c'N_1^\eps}{\dl^c}}\les e^{-\frac 1{\dl^c}}.
\end{align*}

\noi
Let $n_1 \in S_n$. Then, we have  
	\begin{align*}
	|n_1|^2-|n|^2=O(N_1^{2s+2\g+} + N_1)=O(N_1)
	\end{align*}
	
	\noi
	since $|n-n_2|\sim N_2$ and $N_2\sim N_1$. 
Therefore, we have
\begin{align*}
\big| |n_1|-|n| \big| \les 1
\end{align*}
	
\noi
since we are in the case $|n_1| \sim N_1, |n| \sim N$, $N_1\sim N $, and $s+\g<\frac 12$.   Therefore, $|n_1| \in ( |n|-c, |n|+c)$ for some constants $c$. Let $|n_1|=\sqrt{m} $, where $m \ge 0$. Then, $m\in (|n|^2-2c|n|+c^2, |n|^2+2c|n|+c^2)$ and so the possible number of $m$ is given by $ |n| \sim N$. 
	Hence, we have
	\begin{align}
	\sup_n \#S_n \les N_1^\eps N
	\label{S1}
	\end{align}
	
	\noi
	since if $(x,y)=n_1$, where $x,y \in \Z$, then thanks to Lemma \ref{LEM:circle}, the number of lattice points on a circle is given by  
	\begin{align*}
	\big| \{ (x,y)\in \Z^2: x^2+y^2=m   \}  \big| \les N_1^\eps.
	\end{align*}

\noi
From \eqref{highreso11}, H\"older's inequality in $\tau$, \eqref{gausscan113}, Cauchy-Schwarz inequality in $n$, and \eqref{S1}, we have
\begin{align}
\text{LHS of \eqref{highreso11}} &\les \dl^{0-}N_1^{0+} \sum_{\substack{n\in \Z^2\\ |n|\sim N }} N^s \|\ft{P_Nv^S} \|_{L_\tau^{1+}}  \bigg(\sum_{\substack{n_1: n_1+n_2=n,\\ |n_1|^2\pm |n_2|-|n|^2=O(N_1^{2s+2\g+}),\\ |n_1|\sim N_1, |n_2|\sim N_2   }} \|a_{n_1,n_2,n}(\tau)\|_{L^p_\tau}^2 \bigg)^\frac 12 \notag \\
&\les \dl^{1-\frac 1p-}N_1^{-1+s+}N_2^{-1+\g}  \| P_Nv^S\|_{X_S^{0,\frac 12-}}  \bigg( \sum_{\substack{n\in \Z^2\\ |n|\sim N}}  \sum_{\substack{n_1: n_1+n_2=n,\\ |n_1|^2\pm |n_2|-|n|^2=O(N_1^{2s+2\g+}),\\ |n_1|\sim N_1, |n_2|\sim N_2   }}  \bigg)^\frac 12 \notag \\
&\les  \dl^{1-} N_1^{-1+s+}N_2^{-1+\g} \| P_Nv^S  \|_{X_S^{0,0}}   \Big(N^2 \sup_n \#S_n   \Big)^\frac 12 \notag \\
&\les \dl^{1-}N_1^{-\frac 12+s+\g+}  \|P_N v^S  \|_{X_S^{0,0}}.
\label{HS23}
\end{align}

\noi
\noi
Hence, if $s<\frac 12-\g$, then we can perform the dyadic summation over $N_1 \sim N \sim N_2$.



\smallskip

\noi
{\bf   Subsubcase 2.b.\textup{(ii)}:} $N_1\sim N \gg N_2$.
\smallskip

\noi
In this case, if we proceed as in \eqref{HS23} which is based on the Hilbert-Schmidt norm approach, we can no longer use the $N_2^{-1+\g}$ to perform the dyadic summation over $N_1\sim N$. Therefore, the summation loss  $N^2$ in $n$ (the third line of \eqref{HS23} i.e. the summation in $n$) is a big obstacle to performing the dyadic summation over $N_1\sim N \gg N_2$. However, by using the operator norm approach and random matrix estimates, we can overcome the summation loss.

By taking the Fourier transform, we have
\begin{align*}
\F_xP_{N_2}({ P_{N_1}z^S, P_{N}\jb{\nb}^sv^S })(n_2,t)=  \sum_{n \in \Z^2:n+n_1=n_2}  \jb{n}^s \ft{ P_{N}v^S}(n, t )\eta_\dl(t) H(n,n_2,t),
\end{align*}
 
\noi
where $\eta_\dl(t)=\eta(\dl^{-1}t)$ is from our notation \eqref{eta1}, the random matrix $H(n,n_2,t)$ is defined by 
\begin{align*}
H(n,n_2,t):&=  \sum_{n_1\in \Z^2} \frac{e^{-it|n_1|^2}g_{n_1}(\o) }{\jb{n_1} }  \ind_{ \{n_1=n_2-n \} } \ind_{ \{ \varphi(n_1,n_2,n)=O(N_1^{2s+2\g+}) \} } \ind_{\{|n| \sim N \} }   \prod_{j=1}^2 \ind_{\{|n_j| \sim N_j \} }
\end{align*}
 
\noi
and the phase function $\varphi: (\Z^2)^3 \to \R $ is defined by
\begin{align}
\varphi(n_1,n_2,n)= |n_1|^2 \pm |n_2| -|n|^2.
\label{mm1}
\end{align}

 
\noi
Then, from Cauchy-Schwarz inequality, Lemma \ref{LEM:prob}, and taking the operator norm, we have
\begin{equation}
\begin{split}
 &\bigg| \int_{\T^2 \times \R} \jb{\nb}^s \big[ \NN_S(P_{N_1}z^S, P_{N_2}z^W ) \big]     P_N v^S dxdt  \bigg|\\
 &\les \|  P_{N_2} z^W \|_{L^2_{t,x}}
 \bigg\|  \sum_{n}   \jb{n}^s \ft{ P_{N}v^S}(n, t )\eta_\dl(t) H(n,n_2,t)   \bigg\|_{\l^2_{n_2} L_t^2 }\\
 &\les N_2^{\g+} \dl^{\frac 12-}    \Big \| \| H(n,n_2,t) \|_{\l^2_{n}\to \l^2_{n_2}}  \big\| \jb{n}^s  \ft{ P_{N} v^S}(n, t ) \eta_\dl(t) \big\|_{\l^2_{n} } \Big\|_{L^2_t}   \\
 &\les N_2^{\g+} \dl^{\frac 12-} N^s \sup_{t\in \R} \| H(n,n_2,t) \|_{\l^2_{n} \to \l^2_{n_2}}  \| P_N v^S \|_{L^2_{t,x}},
\end{split}
\label{R00}
\end{equation}
 
\noi
which holds outside an exceptional set of measure $e^{-\frac 1{\dl^c}}$.
The random matrix $H(n,n_2,t)$ can be written with the random tensor $h(n,n_1,n_2,t)$ as follows:
\begin{align*}
H(n,n_2,t)&=  \sum_{n_1\in \Z^2} \frac{e^{-it|n_1|^2}g_{n_1}(\o) }{\jb{n_1} }  \ind_{ \{n_1=n_2-n \} } \ind_{ \{ \varphi(n_1,n_2,n)=O(N_1^{2s+2\g+}) \} }  \ind_{\{|n| \sim N \} } \prod_{j=1}^2 \ind_{\{|n_j| \sim N_j \} } \\
&=\sum_{n_1 \in \Z^2} h(n,n_1, n_2,t) g_{n_1}(\o),
\end{align*}
 
\noi
where 
\begin{align}
h(n,n_1,n_2,t)=e^{-it|n_1|^2}\ind_{ \{n_1=n_2-n \} } \ind_{ \{ \varphi(n_1,n_2,n)=O(N_1^{2s+2\g+}) \} } \jb{n_1}^{-1}  \ind_{\{|n| \sim N \} }   \prod_{j=1}^2 \ind_{\{|n_j| \sim N_j \} }.
\label{rt13}
\end{align}
 
\noi
Then, from Lemma \ref{LEM:tens}, we have 
 \begin{align}
\sup_{t\in \R} \| H(n,n_2,t) \|_{\l^2_{n} \to \l^2_{n_2} } \les N_1^\eps \sup_{t\in \R} \max \big( \| h(t)\|_{n_1n \to n_2}, \| h(t)\|_{n \to n_2n_1}  \big).
 \label{tt1}
 \end{align}
 
\noi
Hence, from \eqref{tt1} and Lemma \ref{LEM:zSzWy}, we have
 \begin{align}
\sup_{t\in \R}\| H(n,n_2,t) \|_{\l^2_{n} \to \l^2_{n_2} }  \les N_1^{s+\g-\frac 12+}N_2^{-\frac 12}+ N_1^{-\frac 12+}. 
\label{R123} 
\end{align}
 
\noi
Hence, by combining \eqref{R00} and \eqref{R123}, we have
\begin{align*}
\text{LHS of \eqref{R00}} \les \dl^{\frac 12-}(N_1^{2s+\g-\frac 12+}N_2^{-\frac 12+\g+}+ {N_1}^{-\frac 12+s+\g+}).
\end{align*}

\noi
Therefore, if $s<\frac 14-\frac {\g}2$ and $\g<\frac 12$, then we can perform the dyadic summation over $N_1\sim N\gg N_2$.

\end{proof}

\begin{lemma}[$z^S R^W$-case ]
\label{LEM:zSRW}
Let $0<s<\min\{\frac 14, \l+1 \}$ and $\l>0$. Then,
for each small $\dl>0$, we have
\begin{align*}
\|\NN_S(z^S, R^W) \|_{X_{S,\dl}^{s,-\frac 12+}} \les \dl^{0+} \|R^W \|_{X_{W_\pm,\dl}^{\l,\frac 12+}}
\end{align*}

\noi
outside an exceptional set of probability $< e^{-\frac 1{\dl^c}}$.	
\end{lemma}

\begin{proof}
We perform the case-by-case analysis:	

\smallskip

\noi
{\bf Case 1:} $N_1 \gg N$ or $N \gg N_1$ (non-resonant interaction).

\smallskip

\noi 
 We may assume $N_1 \gg N$ by symmetry.
In this case, we have 
\begin{align}
L_{\max} \ges \big|  |n_1|^2 \pm |n_2|-|n|^2     \big| \gtrsim N_{\max}^2 \sim N_2^2.
\label{highnon21}
\end{align}

\noi
First, suppose that $\max( L_2,L)  \sim L_{\max}$. Then, from $L^p_{t,x}L_{t,x}^{2}L_{t,x}^{2+}$-H\"older's inequality, Lemma \ref{LEM:prstricha} with large $p$, \eqref{highnon21}, and Lemma \ref{LEM:dlpower}, we have
\begin{align*}
\bigg|\int_{\R}\int_{\T^2} \jb{\nb}^s(P_{N_1} z^S P_{N_2} R^W)  P_{N}v^S dxdt \bigg| &\les N_2^s \| P_{N_1}z^S \|_{L^p_{t,x}} \| P_{N_2}R^W \|_{L^2_{t,x}} \|P_{N}v^S \|_{L^{2+}_{t,x}}\\
&\les N_2^{s-\l+} \|P_{N_2}R^W \|_{X^{\l,0}_{W_\pm}}  \| P_{N}v^S \|_{X_S^{0,0+}}\\
& \les \dl^{\frac 12-} N_2^{-1+s-\l+} \|P_{N_2}R^W \|_{X^{\l,\frac 12+}_{W_\pm}} \| P_{N}v^S \|_{X_S^{0,\frac 12-}}
\end{align*}

\noi
outside an exceptional set of measure $< e^{-\frac 1{\dl^c}}$.
Hence, if $s < \l+1$, then we can perform the dyadic summation over $N_1 \sim N_2 \ge N$.

Next, suppose that $\max( L_2,L)  \ll L_{\max}$ and hence $L_1 \sim L_{\max}$. Then, we have
\begin{align}
\big|  \ft \eta_\dl(\tau_1-|n_1|^2) \big| \les \frac 1{L_1}\sim N_{\max}^{-2}
\label{highmodu122}
\end{align}

\noi
since $ \ft \eta_\dl(\tau)=\dl \ft \eta (\dl \tau)$. Then, from H\"older's inequality with $p \gg 1$, \eqref{highmodu122}, Young's inequality in $\tau$, and Lemma \ref{LEM:prob}, we have
\begin{align*}
&\bigg |\int_{\R}\int_{\T^2} \jb{\nb}^s(P_{N_1} z^S P_{N_2} R^W)  P_{N}v^S dxdt \bigg|\\
&= \Bigg| \sum_{\substack{n_1, n \in \Z^2: n_1+n_2=n \\ |n_1|\sim N_1, |n_2| \sim N_2 } } \jb{n}^s \frac{g_{n_1}(\o)}{\jb{n_1}}  \int_{\tau_1+\tau_2=\tau} \ft \eta_\dl (\tau_1-|n_1|^2) \ft{P_{N_2}R^W}(\tau_2,n_2) \ft{P_Nv^S}(n,\tau) d\tau_1 d\tau \Bigg|\\
&\les \| \ft {P_N  v^S } \|_{\ell_n^2L_\tau^{1+}} \Bigg\| \sum_{\substack{n_2 \in \Z^2: n_1+n_2=n \\ |n_1|\sim N_1, |n_2| \sim N_2 } } \jb{n}^s \frac{|g_{n_1}(\o)|}{\jb{n_1}}  N_2^{-2(1-\eps)} \big\| |\ft \eta_\dl |^\eps   \big\|_{L^p_\tau}  \|\ft { P_{N_2}R^W} \|_{L^1_\tau} \Bigg\|_{\l_n^2} \\
&\les  \dl^{\eps-\frac 1p}  N_1^s N_1^{-1+}N_1^{-2(1-\eps)} N_2 N_2^{-\l} N \| P_{N_2}R^W\|_{X_W^{\l,\frac 12+}} \| P_N v^S \|_{X_S^{0,\frac 12-}}\\
&\les \dl^{0+} N_1^{-1+s-\l+\eps } \| P_{N_2}R^W\|_{X_W^{\l,\frac 12+}} \| P_N v^S \|_{X_S^{0,\frac 12-}}
\end{align*}

\noi
outside an exceptional set of probability $< e^{-\frac 1{\dl^c}}$. Hence, if $s < \l+1$, then we can perform the dyadic summation over $N_1 \sim N_2 \ge N$.

\smallskip

\noi
{\bf Case 2:} $N_2 \les N_1 \sim N$ (resonant interaction).

\smallskip

\noi We split the case into the high and low modulation cases.

\smallskip

\noi
{\bf Subcase 2.a:} $L_{\max}\ges  N^{s+\frac 14+}N_2^{\frac 34}$ (high modulation case). 

\smallskip

\noi
First, suppose that $\max(L_2,L)\sim L_{\max}$. Then, from $L^p_{t,x}L_{t,x}^{2}L_{t,x}^{2+}$-H\"older's inequality, Lemma \ref{LEM:prstricha} with large $p$, and Lemma \ref{LEM:dlpower}, we have
\begin{align*}
\bigg|\int_{\R}\int_{\T^2} \jb{\nb}^s(P_{N_1}z^S P_{N_2}R^W) P_Nv^S dxdt \bigg|&\les N^s \| P_{N_1}z^S \|_{L^p_{t,x}}  \| P_{N_2}R^W \|_{L^{2}_{t,x}} \| P_N v^S \|_{L_{t,x}^{2+}}\\
&\les N^{s+} \| P_{N_2}R^W \|_{L^{2}_{t,x}} \| P_N v^S \|_{L^{2+}_{t,x}}\\
&\les \dl^{\frac 12-} N^{\frac s2-\frac 18-}N_2^{-\l-\frac 38}  \| P_{N_2}R^W \|_{X_W^{\l,\frac 12+}}    \| P_N v^S \|_{X_S^{0,\frac 12-}}
\end{align*}

\noi
outside an exceptional set of probability $< e^{-\frac 1{\dl^c}}$.
Hence, if $s<\frac 14$, we obtain the desired result by summing over $N_1 \sim N\ge N_2$.

Next, suppose that $\max(L_2,L ) \ll L_{\max}$ and so $L_1\sim L_{\max}$. Then, from H\"older's inequality with $p \gg 1$, Young's inequality in $\tau$, and Lemma \ref{LEM:prob}, we have
\begin{align*}
&\bigg |\int_{\R}\int_{\T^2} \jb{\nb}^s(P_{N_1} z^S P_{N_2} R^W)  P_{N}v^S dxdt \bigg|\\
&= \Bigg| \sum_{\substack{n_1, n \in \Z^2: n_1+n_2=n \\ |n_1|\sim N_1, |n_2| \sim N_2 } } \jb{n}^s \frac{g_{n_1}(\o)}{\jb{n_1}}  \int_{\tau_1+\tau_2=\tau} \ft \eta_\dl (\tau_1-|n_1|^2) \ft{P_{N_2}R^W}(\tau_2,n_2) \ft{P_Nv^S}(n,\tau) d\tau_1 d\tau \Bigg|\\
&\les \| \ft {P_N  v^S } \|_{\ell_n^2L_\tau^{1+}} \Bigg\| \sum_{\substack{n_2 \in \Z^2: n_1+n_2=n \\ |n_1|\sim N_1, |n_2| \sim N_2 } } \jb{n}^s \frac{|g_{n_1}(\o)|}{\jb{n_1}}  N^{-s-\frac 14-}N_2^{-\frac 34+} \big\| |\ft \eta_\dl |^\eps   \big\|_{L^p_\tau}  \|\ft { P_{N_2}R^W} \|_{L^1_\tau} \Bigg\|_{\l_n^2} \\
&\les  \dl^{\eps-\frac 1p}  N^sN^{-1} N^{-s-\frac 14-}N_2^{-\frac 34+} N_2 N N_2^{-\l} \| P_{N_2}R^W\|_{X_W^{\l,\frac 12+}} \| P_N v^S \|_{X_S^{0,\frac 12-}}\\
&\les \dl^{0+} N^{0- }N_2^{-\l} \| P_{N_2}R^W\|_{X_W^{\l,\frac 12+}} \| P_N v^S \|_{X_S^{0,\frac 12-}}
\end{align*}

\noi
outside an exceptional set of probability $< e^{-\frac 1{\dl^c}}$. Hence, if $\l>0$, we can perform the dyadic summation over $N_1 \sim N \ge N_2$.

\smallskip

\noi
{\bf   Subcase 2.b:} $L_{\max}\les N^{s+\frac 14+}N_2^{\frac 34}$ (low modulation case). 
\smallskip

\noi 
We first rewrite $\NN_S(P_{N_1}z^S, P_{N_2}R^W )$ as a random operator. By taking the Fourier transform, we have
\begin{align*}
\F_x{ \NN^S(P_{N_1}z^S, P_{N_2}R^W) }(n,t)=  \sum_{n_2}   \ft{ P_{N_2}R^W}(n_2, t )\eta_\dl(t) H(n,n_2,t),
\end{align*}
 
\noi
where  $\eta_\dl(t)=\eta(\dl^{-1}t)$ is from our notation \eqref{eta1}, the random matrix $H(n,n_2,t)$ is defined by 
\begin{align*}
H(n,n_2,t):&=  \sum_{n_1\in \Z^2} \frac{e^{-it|n_1|^2}g_{n_1}(\o) }{\jb{n_1} }  \ind_{ \{n_1=n-n_2 \} } \ind_{ \{ \varphi(n_1,n_2,n)=O(N^{s+\frac 14+}N_2^{\frac 34}) \} }  \ind_{\{|n| \sim N \} } \prod_{j=1}^2 \ind_{\{|n_j| \sim N_j \} },
\end{align*}
 
\noi
and the phase function $\varphi: (\Z^2)^3 \to \R $ is defined by
\begin{align}
\varphi(n_1,n_2,n)= |n_1|^2 \pm |n_2| -|n|^2.
\label{m19}
\end{align}

 
\noi
Then, from Cauchy-Schwarz inequality and taking the operator norm, we have
\begin{equation}
\begin{split}
 &\bigg| \int_{\T^2 \times \R} \jb{\nb}^s \big[ \NN_S(P_{N_1}z^S, P_{N_2}R^W ) \big]     P_N v^S dxdt  \bigg|\\
 &\les \| \jb{\nb}^s P_N v^S \|_{L^2_{t,x}}
 \bigg\|  \sum_{n_2}    \ft{ P_{N_2}R^W }(n_2,t )\eta_\dl(t) H(n,n_2,t)     \bigg\|_{\l^2_n L_t^2 }\\
 &\les N^s \| P_N v^S \|_{L^2_{t,x}}   \Big\|  \| H(n,n_2,t) \|_{\l^2_{n_2}\to \l^2_n}  \big\|   \ft{ P_{N_2} R^W}(n_2, t )  \big\|_{\l^2_{n_2}} \Big\|_{L^2_t}  \\
 &\les N^sN_2^{-\l} \sup_{t\in \R}\| H(n,n_2,t) \|_{\l^2_{n_2} \to \l^2_n} \| P_{N_2}R^W \|_{X_{W_\pm}^{\l,\frac 12+\eps}} \| P_N v^S \|_{L^2_{t,x}} .
\end{split}
\label{R0}
\end{equation}
 
\noi
The random matrix $H(n,n_2,t)$ can be written with the random tensor $h(n,n_1,n_2,t)$ as follows:
\begin{align*}
H(n,n_2,t)&=  \sum_{n_1\in \Z^2} \frac{e^{-it|n_1|^2}g_{n_1}(\o) }{\jb{n_1} }  \ind_{ \{n_1=n-n_2 \} } \ind_{ \{ \varphi(n_1,n_2,n)=O(N^{s+\frac 14+}N_2^{\frac 34}) \} }  \ind_{\{|n| \sim N \} } \prod_{j=1}^2 \ind_{\{|n_j| \sim N_j \} }\\
&=\sum_{n_1 \in \Z^2} h(n,n_1, n_2,t) g_{n_1}(\o),
\end{align*}
 
\noi 
where 
\begin{align}
h(n,n_1,n_2,t)=e^{-it|n_1|^2}\ind_{ \{n_1=n-n_2 \} } \ind_{ \{ \varphi(n_1,n_2,n)=O(N^{s+\frac 14+}N_2^{\frac 34}) \} } \jb{n_1}^{-1}  \ind_{\{|n| \sim N \} }   \prod_{j=1}^2 \ind_{\{|n_j| \sim N_j \} } .
\label{rt}
\end{align}
 
\noi
Then, from Lemma \ref{LEM:tens}, we have 
 \begin{align}
\sup_{t\in\R} \| H(n,n_2,t) \|_{\l^2_{n_2} \to \l^2_n } \les N_1^\eps \sup_{t\in \R} \max \big( \| h(t)\|_{n_1n_2 \to n}, \| h(t)\|_{n_2 \to nn_1}  \big).
 \label{t}
 \end{align}
 
\noi
Hence, from \eqref{t} and Lemma \ref{LEM:Tens1}, we have
 \begin{align}
\sup_{t\in \R} \| H(n,n_2,t) \|_{\l^2_{n_2} \to \l^2_n }  \les N^{\frac s2-\frac 38+}N_2^{-\frac 18}+N_1^{-\frac 12}.
\label{R1} 
\end{align}
 
\noi
Hence, by combining \eqref{R0} and \eqref{R1}, we have
\begin{align*}
\text{LHS of \eqref{R0}} \les N_1^{\frac 32s-\frac 38+}N_2^{-\frac 18-\l}+N_1^{s-\frac 12+}N_2^{-\l}.
\end{align*}

\noi
Therefore, if $s<\frac 14$ and $\l>0$, then we can perform the dyadic summation over $N_1\sim N\ges N_2$.

\end{proof}

\begin{lemma}[$R^S z^W$-case ]
\label{LEM:RszW}
Let $\g<\frac 13$ and $0<s<1-\g$. 
Then, for each small $\dl>0$,  we have
\begin{align}
\|\NN_S(R^S, z^W) \|_{X_{S,\dl}^{s,-\frac 12+}} \les \dl^{0+} \|R^S \|_{X_{S,\dl}^{s,\frac 12+}}
\label{APPbi}
\end{align}
	
\noi
outside an exceptional set of probability $< e^{-\frac 1{\dl^c}}$.		
\end{lemma}	


\begin{proof}
We perform the case-by-case analysis:

\smallskip

\noi
{\bf  Case 1:} $N_1 \gg N$ or $N \gg N_1$ (non-resonant interaction).

\smallskip

\noi 
We first consider the (worse) case $N \gg N_1$.
In this case, we have 
\begin{align}
L_{\max} \ges \big|  |n_1|^2 \pm |n_2|-|n|^2     \big| \gtrsim N_{\max}^2 \sim N_2^2.
\label{highmodu22}
\end{align}

\noi
First suppose that $\max(L_1,L) \sim L_{\max}$. Then, from $L_{t,x}^{2+}L_{t,x}^{p}L_{t,x}^{2}$-H\"older's inequality with $p$ large, Lemma \ref{LEM:prstricha}, \eqref{interop}, \eqref{highmodu22}, and Lemma \ref{LEM:dlpower}, we have
\begin{align*}
\bigg|\int_{\R}\int_{\T^2} \jb{\nb}^s(P_{N_1} R^S P_{N_2} z^W)  P_{N}v^S dxdt \bigg| &\les N_2^s \| P_{N_1}R^S \|_{L^{2+}_{t,x}} \| P_{N_2}z^W \|_{L^p_{t,x}} \|P_{N}v^S \|_{L^2_{t,x}}\\
&\les \dl^{\frac 12-} N_2^{-1+s+\g+}  \|P_{N_1}R^s \|_{X^{0,\frac 12+}_S}  \| P_{N}v^S \|_{X_S^{0,\frac 12-}}\\
& \les \dl^{\frac 12-} N_2^{-1+\g+s+}N_1^{-s}   \|P_{N_1}R^S \|_{X^{s,\frac 12+}_S} \| P_{N}v^S \|_{X^{0,\frac 12-}_S}
\end{align*}

\noi
for an exceptional set of measure $<e^{-\frac 1{\dl^c}}$.
Hence, if $s+\g < 1$, then we can perform the dyadic summation over $N \sim N_2 \ge N_1$. Notice that in the case $N_1 \gg N$, it suffices to assume $\g<1$.

Next, suppose that $\max(L_1,L )\ll L_{\max}$ and so $L_2\sim L_{\max}$. Then, from H\"older's inequality with $p \gg 1$, Young's inequality in $\tau$, and Lemma \ref{LEM:prob}, we have
\begin{align*}
&\bigg |\int_{\R}\int_{\T^2} \jb{\nb}^s(P_{N_1} R^S P_{N_2} z^W)  P_{N}v^S dxdt \bigg|\\
&= \Bigg| \sum_{\substack{n_1, n \in \Z^2: n_1+n_2=n \\ |n_1|\sim N_1, |n_2| \sim N_2 } } \jb{n}^s \frac{h_{n_2}(\o)}{\jb{n_2}^{1-\g}}  \int_{\tau_1+\tau_2=\tau} \ft{P_{N_1}R^S}(\tau_1,n_1) \ft \eta_\dl (\tau_2\pm|n_2|)  \ft{P_Nv^S}(n,\tau) d\tau_1 d\tau \Bigg|\\
&\les \| \ft {P_N  v^S } \|_{\ell_n^2L_\tau^{1+}} \Bigg\| \sum_{\substack{n_1 \in \Z^2: n_1+n_2=n \\ |n_1|\sim N_1, |n_2| \sim N_2, |n|\sim N } } \jb{n}^s \frac{|h_{n_2}(\o)|}{\jb{n_2}^{1-\g}}  N_2^{-2(1-\eps)} \big\| |\ft \eta_\dl |^\eps   \big\|_{L^p_\tau}  \|\ft { P_{N_1}R^S} \|_{L^1_\tau} \Bigg\|_{\l_n^2} \\
&\les  \dl^{\eps-\frac 1p}  N^sN_2^{-1+\g+} N_2^{-2+} N_1 N  \| P_{N_1}R^S\|_{X_S^{0,\frac 12+}} \| P_N v^S \|_{X_S^{0,\frac 12-}}\\
&\les \dl^{0+} N^{s-1+\g+} N_1^{-s} \| P_{N_1}R^S\|_{X_S^{0,\frac 12+}} \| P_N v^S \|_{X_S^{0,\frac 12-}}
\end{align*}

\noi
outside an exceptional set of probability $< e^{-\frac 1{\dl^c}}$. Hence, if $s<1-\g$, we can perform the dyadic summation over $ N \sim  N_2 \ge N_1$. Notice that in the case $N_1 \gg N$, it suffices to assume $\g<1$.

\smallskip

\noi
{\bf   Case 2:} $N_2 \les N_1\sim N$ (resonant interaction).

\smallskip

\noi We split the case into the high and low modulation cases.

\smallskip

\noi
{\bf   Subcase 2.a:} $L_{\max}\ges N_2^{1+\g+}$ (high modulation case). 
\smallskip

\noi
When $N_1\gg N_2$, we first wirte $\{|n_1| \sim N_1 \}=\bigcup_{\ell_1} J_{1\ell_1}$ and $\{ |n| \sim N \}=\bigcup_{\ell_2}J_{2\ell_2}$, where $J_{1,\ell_1}$ and $J_{2,\ell_2}$ are balls of radius $\sim N_2$, we can decompose $\ft {P_{N_1}R^S}$ and $\ft {P_{N}v^S}$ as
\begin{align*}
\ft {P_{N_1}R^S}=\sum_{\ell_1} \ft { P_{N_1,\ell_1}R^S } \qquad \text{and} \qquad \ft {   P_{N}v^S }=\sum_{\ell_2} \ft { P_{N,\ell_2}v^S }
\end{align*} 

\noi
where $\ft {P_{N_1,\ell_1}R^S }(n_1,t)=\ind_{J_{1\ell_1}}(n_1) \ft { P_{N_1}R^S }(n_1,t) $ 
and $\ft {P_{N,\ell_2}v^S }(n,t)=\ind_{J_{2\ell_2}}(n) \ft { P_{N}v^S }(n,t) $. Given $n_1 \in J_{1\ell_1}$ for some $\ell_1$, there exists $O(1)$ many possible values for $\ell_2=\ell_2(\ell_1)$ such that $n \in J_{2\ell_2}$ under $n_1+n_2=n$. Notice that the number of possible values of $\ell_2$ is independent of $\ell_1$. 

First suppose that $\max(L_1,L) \sim L_{\max}$. Then, from $L_{t,x}^{2+}L_{t,x}^{p}L_{t,x}^{2}$-H\"older's inequality with $p$ large, Lemma \ref{LEM:prstricha}, \eqref{interop}, and Lemma \ref{LEM:dlpower}, we have 
\begin{align*}
\bigg|\int_{\R}\int_{\T^2} \jb{\nb}^s(P_{N_1,\l_1} R^S P_{N_2} z^W)  P_{N,\l_2}v^S dxdt \bigg| &\les N_1^s \| P_{N_1,\l_1}R^S \|_{L^{2+}_{t,x}} \| P_{N_2}z^W \|_{L^p_{t,x}} \|P_{N,\l_2}v^S \|_{L^2_{t,x}} \\
&\les \dl^{\frac 12-} N_2^{\g+}N_2^{-\frac 12-\frac \g2-}  \| P_{N_1,\l_1}R^S \|_{X^{s,\frac 12+}_S}  \|P_{N,\l_2}v^S \|_{X^{0,\frac 12-}_S}\\
& = \dl^{\frac 12-} N_2^{-\frac 12+\frac \g2-} \| P_{N_1,\l_1}R^S \|_{X^{s,\frac 12+}_S}  \|P_{N,\l_2}v^S \|_{X^{0,\frac 12-}_S}
\end{align*}

\noi
for an exceptional set of measure $<e^{-\frac 1{\dl^c}}$.
Hence, we can perform the the Cauchy-Schwarz inequality in $\l_1$ and the dyadic summation over $N_1 \sim N \ges N_2$ if $\g<1$.

Next, suppose that $\max(L_1,L )\ll L_{\max}$ and so $L_2\sim L_{\max}$. Then, from H\"older's inequality with $p \gg 1$, Young's inequality in $\tau$, Lemma \ref{LEM:prob}, and Cauchy-Schwarz inequality in $n_1\in J_{1\l_1}$, we have
\begin{align*}
&\bigg |\int_{\R}\int_{\T^2} \jb{\nb}^s(P_{N_1,\l_1} R^S P_{N_2} z^W)  P_{N,\l_2}v^S dxdt \bigg|\\
&= \Bigg| \sum_{\substack{n_1, n \in \Z^2: n_1+n_2=n \\ n_1\in J_{1\l_1}, n\in J_{2\l_2}, |n_2|\sim N_2 }  } \jb{n}^s \frac{h_{n_2}(\o)}{\jb{n_2}^{1-\g}}  \int_{\tau_1+\tau_2=\tau} \ft{P_{N_1,\l_1}R^S}(\tau_1,n_1) \ft \eta_\dl (\tau_2\pm|n_2|)  \ft{P_{N,\l_2}v^S}(n,\tau) d\tau_1 d\tau \Bigg|\\
&\les \| \ft {P_{N,\l_2}  v^S } \|_{\ell_n^2L_\tau^{1+}} \Bigg\| \sum_{\substack{n_1 \in \Z^2: n_1+n_2=n \\ n_1\in J_{1\l_1}, n \in J_{2\l_2}, |n_2|\sim N_2 } } \jb{n}^s \frac{|h_{n_2}(\o)|}{\jb{n_2}^{1-\g}}  N_2^{-(1+\g+)(1-\eps)} \big\| |\ft \eta_\dl |^\eps   \big\|_{L^p_\tau}  \|\ft { P_{N_1,\l_1}R^S} \|_{L^1_\tau} \Bigg\|_{\l_n^2} \\
&\les  \dl^{\eps-\frac 1p-}  N^s N_2^{-1+\g+\eta}N_2^{-1-\g-2\eta} N_2 N_2  \| P_{N_1,\l_1}R^S\|_{X_S^{0,\frac 12+}} \| P_{N,\l_2} v^S \|_{X_S^{0,\frac 12-}}\\
&\les \dl^{0+} N_2^{0-} \| P_{N_1,\l_1}R^S \|_{X^{s,\frac 12+}_S}  \|P_{N,\l_2}v^S \|_{X^{0,\frac 12-}_S}
\end{align*}

\noi
outside an exceptional set of probability $< e^{-\frac 1{\dl^c}}$. Hence, we can perform the Cauchy-Schwarz inequality in $\l_1$ and the dyadic summation over $ N \sim  N_1 \ge N_2$.




\smallskip

\smallskip

\noi
{\bf  Subcase 2.b:} $L_{\max}\les  N_2^{1+\g+}$ (low modulation case). 
\smallskip

\noi
We rewrite $\NN_S(P_{N_1,\l_1}R^S, P_{N_2}z^W )$ as a random operator. By taking the Fourier transform, we have
\begin{align*}
\F_{x} \NN_S(P_{N_1,\l_1}R^S, P_{N_2}z^W)(n,t)=  \sum_{n_1} \ft{ P_{N_1,\l_1}R^S}(n_1, t) \eta_\dl(t) H(n,n_1,t),
\end{align*}

\noi
where $\eta_\dl(t)=\eta(\dl^{-1}t)$ is from our notation \eqref{eta1}, the random matrix $H(n,n_1,t)$ is defined by
\begin{align}
H(n,n_1,t)&:=  \sum_{n_2 \in \Z^2} \frac{e^{-it|n_2|}h_{n_2}(\o) }{\jb{n_2}^{1-\g} } \ind_{ \{n_2=n-n_1 \} } \ind_{ \{ \varphi(n_1,n_2,n)=O(N_2^{1+\g+}) \} }  \ind_{\{ n_1 \in J_{1\l_1}  \} } \ind_{\{ n \in J_{2\l_2}  \} }  \ind_{ \{|n_2| \sim N_2 \}},
\label{r2}
\end{align}

\noi
and the phase function $\varphi: (\Z^2)^3 \to \R $ is defined by
\begin{align}
\varphi(n_1,n_2,n)= |n_1|^2 \pm |n_2| -|n|^2.
\label{m1}
\end{align}


\noi
Then, from Cauchy-Schwarz inequality, Minkowski's inequality in $\tau$ and taking the operator norm, we have
\begin{equation}
\begin{split}
&\bigg| \int_{\T^2 \times \R} \jb{\nb}^s \big[\NN_S(P_{N_1,\l_1}R^S, P_{N_2}z^W)\big]     P_{N,\l_2} v^S dxdt  \bigg|\\
&\les \| \jb{\nb}^s P_{N,\l_2} v^S \|_{L^2_{t,x}}
\bigg\|  \sum_{n_1} \ft{ P_{N_1,\ell_1}R^S }(n_1, t) \eta_\dl(t) H(n,n_1)     \bigg\|_{\l^2_n L_t^2 }\\
&\les N^s \| P_{N,\l_2} v^S \|_{L^2_{t,x}}    \Big\|  \| H(n,n_1,t) \|_{\l^2_{n_1}\to \l^2_n}  \big\|   \ft { P_{N_1,\l_1}R^S}(n_1, t )  \big\|_{\l^2_{n_1} } \Big\|_{L_t^2}    \\
&\les   \sup_{t\in \R}\| H(n,n_1,t) \|_{\l^2_{n_1} \to \l^2_n} \| P_{N,\l_2} v^S \|_{X^{0,\frac 12-}_S} \| P_{N_1,\l_1}R^S \|_{X_{S}^{s,\frac 12+}}.
\end{split}
\label{s2}
\end{equation}

\noi
The random matrix $H(n,n_1,t)$ in \eqref{r2} can be written with the random tensor $h(n,n_1,n_2,t)$ as follows:
\begin{align*}
H(n,n_1,t)&=\sum_{n_2 \in \Z^2} h(n,n_1, n_2,t) h_{n_2}(\o),
\end{align*}

\noi
where 
\begin{equation}
\begin{split}
h(n,n_1,n_2,t) &=e^{-it|n_2|}\ind_{ \{n_2=n-n_1,\; |n_2| \sim N_2 \} } \ind_{ \{ \varphi(n_1,n_2,n)=O(N_2^{1+\g+}) \} } \ind_{\{n_1 \in J_{1\ell_1}\}} \ind_{\{n \in J_{2\ell_2}\}}   \jb{n_2}^{-1+\g}.
\end{split}
\label{rt1}
\end{equation}

\noi
Then, from Lemma \ref{LEM:tens} and \ref{LEM:tens0}, we have 
\begin{align}
\| H(n,n_1,t) \|_{\l^2_{n_1} \to \l^2_n } \les N_{2}^\eps \max \big( \| h(t)\|_{n_1n_2 \to n}, \| h(t)\|_{n_1 \to nn_2 }  \big).
\label{t1}
\end{align}

\noi
Hence, from \eqref{t1} and Lemma \ref{LEM:Tens2}, we have
\begin{align}
\sup_{t\in \R}\| H(n,n_1,t) \|_{\l^2_{n_1} \to \l^2_n }  \les  N_2^{-\frac 12+\frac 32\g+}. 
\label{t11}
\end{align}

\noi
From \eqref{s2}, \eqref{t11}, and the Cauchy-Schwarz inequality in $\l_1$, we have 
\begin{align*}
\text{LHS of \eqref{s2}}\les N_2^{-\frac 12+\frac 32\g+} \| P_{N}v^S\|_{X_S^{0,\frac 12-}} \| P_{N_1} R^s \|_{X_S^{s,\frac 12+}},
\end{align*}

\noi
where we used the fact that the number of possible value of $\l_2=\l_2(\l_1)$ is independent of $\l_1$.
Therefore, if $\g<\frac 13$, we can perform the dyadic summation over $N_1\sim N\gg N_2$.  

When $N_1 \sim N_2$, the above argument used in the case $N_1 \gg N_2$ is itself applicable without using the orthogonality argument.

\end{proof}

\subsection{Bilinear estimates for the wave part}
\label{SUBSEC:biw}
In this subsection, we prove \eqref{bilin22} in Lemma \ref{LEM:bi2}.
In view of \eqref{bi0},  in order to prove \eqref{bilin22}, we need to  carry  out case-by-case analysis on  
\begin{align}
\|\chi_{_\dl} \cdot \mathcal{N}_{W}(u_1, u_2)\|_{X_{W_\pm}^{\l, -\frac{1}{2}+}} 
\label{bi00}
\end{align}


\noi
where $u_j$ is taken to be either of type
\begin{itemize}
\item[(I)] rough random part: 
\begin{align*}
u_j  &= \wt z^{S}\text{, where $\wt z^S$ is {\it some} extension of }
\chi_{_\dl}\cdot  z^S
\end{align*}	
	
\noi
where $z^S$ denotes the random linear solution defined in \eqref{lin1}, 
	
\vspace{1mm}
	
\item[$(\II)$]	 smoother  `deterministic' remainder (nonlinear) part: 
\begin{align*}
u_j &= \wt{R}^S_j \text{, where  $\wt{R}^S_j$ is {\it any} extension of $R^S$,}
\end{align*}
\end{itemize}

\noi
In the following, when $u_j$ is of type (I),
we take $\wt z^S = \eta_{_\dl} z^S$.
Thanks to the cutoff function in \eqref{bi00},
we may take $u_j = \eta_{_\dl}\cdot \wt R_j^S $ in \eqref{bi00}
when $u_j$ is of type $(\II)$.



\begin{remark} \rm
	In the following, we drop the $\pm$ signs and work with one $w_+$ or $w_{-}$ since there is no role of $\pm$. Hence, we set $w:=w_{+}$ and $W:=W_{+}$. 
	
\end{remark}

\begin{remark}\rm
To estimate $\|\chi_{_\dl} \cdot \mathcal{N}_W(u_1, u_2)\|_{X_{W_\pm}^{\l, -\frac{1}{2}+}} $,  we need to perform case-by-case analysis of expressions of the form:
\begin{align*}
\int_{\R}\int_{\T^2} \jb{\nb}^{\l-1+2\g} \NN_W(u_1,u_2) \cj v^W dxdt,
\end{align*}
	
\noi
where $\|v^{W} \|_{X_{W_\pm}^{0,\frac 12-}}\le 1$. As in Remark \ref{REM:not},
for simplicity of notation, we drop the complex conjugate sign and suppress the smooth time-cutoff function $\eta_\dl$ and thus simply denote them by $z^S$ and $R^S$, respectively when there is no confusion. We dyadically decompose $u_1$ and $u_2$ such that their spatial frequency support are $\supp \ft u_1 \subset \{ |n_1| \sim N_1 \}$, $\supp \ft u_2 \subset \{ |n_2| \sim N_2 \}$, and $\supp \ft{v^W} \subset \{ |n| \sim N \}$  for some dyadic $N_1,N_2$ and $N \ge 1$. Lastly, we point out that $n_1 \neq n_2$ thanks to the renormalization (see \eqref{nonlinear}).   	
\end{remark}

We now prove Lemmas \ref{LEM:RSRS}, \ref{LEM:zSzS}, and \ref{LEM:RSzS} which will imply Lemma \ref{LEM:bi2}'s second part \eqref{bilin22} (bilinear estimates for the wave part).

\begin{lemma}[$R^S R^S$-case]
\label{LEM:RSRS}
Let $\l<1-2\g+s$ and $s>0$. Then, we have
\begin{align*}
\|\NN_W(R^S, R^S) \|_{X_{W_\pm,\dl}^{\l,-\frac 12+}} \les \dl^{0+} \|R^S \|_{X_{S,\dl}^{s,\frac 12+} } \|R^S \|_{X_{S,\dl}^{s,\frac 12+} }.
\end{align*}
	
\end{lemma}

\begin{proof}
We may assume $N_1 \ge N_2$ by symmetry.

\smallskip

\noi
{\bf  Case 1:} $N_1 \gg N_2$ 
	
\smallskip

\noi
Let ${\bf P}_{\neq 0}$ is the projection onto non-zero frequencies: ${\bf P}_{\neq 0}f:=f-\int_{\T^2} f$. Then, from the boundedness\footnote{${\bf P}_{\neq 0}$ is clearly bounded on $L^p(\T^2)$, $1\le p \le \infty$.} of ${\bf P}_{\neq 0}$ on $L^p(\T^2)$, $L_{t,x}^{4}L_{t,x}^{4}L_{t,x}^{2}$-H\"older's inequality, the $L^4$-Strichartz estimate (Lemma \ref{LEM:L4}), and Lemma \ref{LEM:dlpower}, we have
\begin{align*}
&\bigg|\int_{\R}\int_{\T^2} \jb{\nb}^{\l-1+2\g}{\bf P}_{\neq 0}(P_{N_1} R^S P_{N_2} R^S)  P_{N}v^W dxdt \bigg|\\
&\les N^{\l-1+2\g} N_1^{-s+} N_2^{-s+}   \| P_{N_1} R^S  \|_{X_S^{s,\frac 12-} } \|P_{N_2}R^S \|_{X_S^{s,\frac 12-} } \|P_Nv^W \|_{X_W^{0,0} }\\
&\les \dl^{\frac 12-}  N^{\l-1+2\g-s+} N_2^{-s+}   \| P_{N_1} R^S  \|_{X_S^{s,\frac 12-} } \|P_{N_2}R^S \|_{X_S^{s,\frac 12-} } \|P_Nv^W \|_{X_W^{0,\frac 12-} }.
\end{align*}

\noi
Hence, if $\l<1-2\g+s$ and $s>0$, we can perform the dyadic summation over $N_1 \sim N \ges N_2$.

\smallskip

\noi
{\bf  Case 2:} $N_1 \sim N_2 \ges N$ 
\smallskip

\noi
By $L_{t,x}^{4}L_{t,x}^{4}L_{t,x}^{2}$-H\"older's inequality, the $L^4$-Strichartz estimate (Lemma \ref{LEM:L4}), and Lemma \ref{LEM:dlpower}, we have
\begin{align*}
&\bigg|\int_{\R}\int_{\T^2} \jb{\nb}^{\l-1+2\g}{\bf P}_{\neq 0}(P_{N_1} R^S P_{N_2} R^S)  P_{N}v^W dxdt \bigg|\\
&\les \dl^{\frac 12-} N^{\l-1+2\g} N_1^{-2s+}   \| P_{N_1} R^S  \|_{X_S^{s,\frac 12-} } \|P_{N_2}R^S \|_{X_S^{s,\frac 12-} } \|P_Nv^W \|_{X_W^{0,\frac 12-} }.
\end{align*}

\noi
Hence, if $\l<1-2\g+2s$ and $s>0$, we can perform the dyadic summation over $N_1 \sim N_2 \ges N$.

\end{proof}

\begin{lemma}[$z^S z^S$-case]
\label{LEM:zSzS}
Let $\l<1-2\g$.
Then, for each small $\dl>0$, we have
\begin{align*}
\|\NN_W(z^S, z^S) \|_{X_{W_\pm ,\dl}^{\l,-\frac 12+}} \les \dl^{0+}
\end{align*}

\noi
outside an exceptional set of probability $< e^{-\frac 1{\dl^c}}$.
\end{lemma}

\begin{proof}
We may assume $N_1 \ge N_2$ by symmetry.

\smallskip

\noi
{\bf Case 1:} $N_1 \gg N_2$ (non-resonant interaction). 
\smallskip

\noi
In this case, we have 
\begin{align}
L_{\max} \ges \big|  |n_1|^2 - |n_2|^2\pm |n|     \big| \gtrsim N_{\max}^2 \sim N_1^2.
\label{highmodu222}
\end{align}

\noi
First suppose that $L\sim L_{\max}$. Then, from the boundedness of ${\bf P}_{\neq 0}$ on $L^p(\T^2)$, $L_{t,x}^{p}L_{t,x}^{2+}L_{t,x}^{2}$-H\"older's inequality with $p$ large, Lemma \ref{LEM:prstricha}, \eqref{highmodu222}, and Lemma \ref{LEM:dlpower}, we have
\begin{align*}
\bigg|\int_{\R}\int_{\T^2} \jb{\nb}^{\l-1+2\g}{\bf P}_{\neq 0}(P_{N_1} z^S \cj{P_{N_2} z^S})  P_{N}v^W dxdt \bigg| &\les N_1^{\l-1+2\g} \| P_{N_1}z^S \|_{L^p} \| P_{N_2}z^S \|_{L^2+} \|P_{N}v^W \|_{L^{2}_{t,x}}\\
&\les \dl^{0+} L_{\max}^{-\frac 12+} N_1^{\l-1+2\g+}    \| P_{N_1}v^W \|_{X_W^{0,\frac 12-}}\\
& \les \dl^{0+} N_1^{\l-2+2\g+} \| P_{N}v^W \|_{X_W^{0,\frac 12-}}
\end{align*}

\noi
outside an exceptional set of measure $< e^{-\frac 1{\dl^c}}$.
Hence, if $\l < 2-2\g$, we can perform the dyadic summation over $N_1 \sim N \ge N_2$.

Next, suppose that $L\ll L_{\max}$ and so $\max(L_1,L_2)\sim L_{\max}$. Then, from H\"older's inequality with $p \gg 1$, Young's inequality in $\tau$, and Lemma \ref{LEM:prob}, we have
\begin{align*}
&\bigg |\int_{\R}\int_{\T^2} \jb{\nb}^{\l-1+2\g} {\bf P}_{\neq 0}(P_{N_1} z^S \cj{P_{N_2} z^S})  P_{N}v^W dxdt \bigg|\\
&= \Bigg| \sum_{\substack{n_2, n \in \Z^2: n_1-n_2=n \\ |n_1|\sim N_1, |n_2| \sim N_2, n_1\neq n_2 } } \jb{n}^s \frac{g_{n_1}(\o)}{\jb{n_1}} \frac {\cj {g_{n_2} }(\o)}{\jb{n_2}} \int_{\tau_1+\tau_2=\tau} \ft \eta_\dl (\tau_1-|n_1|^2) \ft \eta_\dl(\tau_2- |n_2|^2) \ft{P_Nv^W}(n,\tau) d\tau_1 d\tau \Bigg|\\
&\les \| \ft {P_N  v^W } \|_{\ell_n^2L_\tau^{1+}} \Bigg\| \sum_{\substack{n_2 \in \Z^2: n_1-n_2=n \\ |n_1|\sim N_1, |n_2| \sim N_2, n_1 \neq n_2 } } \jb{n}^{\l-1+2\g} \frac{|g_{n_1}(\o)|}{\jb{n_1}} \frac {|g_{n_2}(\o)|}{\jb{n_2}} N_1^{-2(1-\eps)} \big\| |\ft \eta_\dl |^\eps   \big\|_{L^p_\tau}  \|\ft \eta_\dl \|_{L^1_\tau} \Bigg\|_{\l_n^2} \\
&\les  \dl^{\eps-\frac 1p-}  N^{\l-1+2\g} N_1^{-1+}N_2^{-1+}N_1^{-2+} N_2^2 N  \| P_N v^W \|_{X_S^{0,\frac 12-}}\\
&\les \dl^{0+} N_1^{\l-2+2\g+}  \| P_N v^W \|_{X_W^{0,\frac 12-}}.
\end{align*}

\noi
Hence, if $\l<2-2\g$, we can perform the dyadic summation over $N_1\sim N\ges N_2$.

\smallskip

\noi
{\bf Case 2:} $N_1\sim N_2 \ges N$ (resonant interaction).

\smallskip

\noi
We note that
\begin{align}
&\bigg |\int_{\R}\int_{\T^2} \jb{\nb}^{\l-1+2\g}{\bf P}_{\neq 0}(P_{N_1}z^S \cj{P_{N_2} z^S}) P_Nv^W dxdt \bigg| \notag \\
&= \bigg| \int_{\tau \in \R}  \sum_{\substack{n\in \Z^2: \\ |n|\sim N}} \jb{n}^{\l-1+2\g} \ft{P_Nv^W}(n,\tau)\Big(\sum_{\substack{n_1\in \Z^2: \\ n_1-n_2=n,\\ |n_1|\sim N_1, |n_2|\sim N_2, n_1 \neq n_2}} a_{n_1,n_2,n}(\tau) g_{n_1}(\o) \cj {g_{n_2} }(\o) \Big) d\tau \bigg|
\label{highreso111}
\end{align}

\noi
where
\begin{align*}
a_{n_1,n_2,n}(\tau)=\jb{n_1}^{-1}\jb{n_2}^{-1}  \int_{\tau_1+\tau_2=\tau }\ft \eta_\dl(\tau_1-|n_1|^2) \ft\eta_\dl(\tau_2-|n_2|^2) d\tau_1.
\end{align*}

\noi
Then, from Lemma \ref{LEM:multigauss}, Minkowski's inequality in $\tau$ (with $p\gg 1$), \eqref{highmodu13}, Young's inequality, and no pairing condition $n_1\neq n_2$, we have
\begin{align}
\Big\| \sum_{\substack{n_1\in \Z^2: \\ n_1-n_2=n,\\ |n_1|\sim N_1, |n_2|\sim N_2, n_1 \neq n_2}} a_{n_1,n_2,n}(\tau) g_{n_1}(\o) \cj {g_{n_2} }(\o)  \Big\|_{L^p_\tau} &\les \dl^{0-}N_1^{0+} \Big(\sum_{\substack{n_2\in \Z^2: \\ n_1-n_2=n,\\ |n_1|\sim N_1, |n_2|\sim N_2}}\|a_{n_1,n_2,n}(\tau)\|_{L^p_\tau}^2 \Big)^\frac 12 \notag \\
&\les \dl^{0-}N_1^{-1+}N_2^{-1} \notag \\
&\times  \Big( \sum_{\substack{n_2\in \Z^2:|n_2|\sim N_2}}  \big\| \ft \eta_\dl \|_{L^p_\tau}^2    \| \ft \eta_\dl \|_{L^1_\tau}^2     \Big)^\frac 12 \notag\\
&\les \dl^{1-\frac 1p-}N_1^{-1+}N_2^{-1}N_2.
\label{gausscan11}
\end{align}

\noi
In \eqref{gausscan11}, we need to make sure that the probability $e^{-c'\frac{N_1^\eps}{\dl^c}}$ of the exceptional sets corresponding to different dyadic blocks and different values of $n_2$ should be summable and bounded by $e^{-\frac 1{\dl^c}}$ i.e. \eqref{gausscan11} holds outside an exceptional set of measure:
\begin{align*}
\sum_{N_1} N_2^2 e^{-\frac{c'N_1^\eps}{\dl^c}}\les e^{-\frac 1{\dl^c}}.
\end{align*}

\noi
From \eqref{highreso111}, H\"older's inequality in $\tau$, \eqref{gausscan11}, and  Cauchy-Schwarz inequality in $n$, we have
\begin{align*}
\text{LHS of \eqref{highreso111}} &\les \sum_{\substack{n\in \Z^2\\ |n|\sim N }} N^{\l-1+2\g} \|\ft{P_Nv^W} \|_{L_\tau^{1+}}  \Big( \sum_{\substack{n_2\in \Z^2: \\ n_1-n_2=n,\\ |n_1|\sim N_1, |n_2|\sim N_2}}\|a_{n_1,n_2,n}(\tau)\|_{L^p_\tau}^2 \Big)^\frac 12 \notag \\
&\les \dl^{1-\frac 1p-}N^{\l-1+2\g}N_1^{-1+}N_2^{-1}  N_2N \| P_Nv^W\|_{X_W^{0,\frac 12-}}  \notag \\
&\les \dl^{1-\frac 1p-}N_1^{\l-1+2\g+} \| P_Nv^W\|_{X_W^{0,\frac 12-}}.
\end{align*}

\noi
Hence, if $\l<1-2\g$, we can perform the dyadic summation over $N_1 \sim N_2\ges N$.
\end{proof}

\begin{lemma}[$ z^SR^S$-case]
\label{LEM:RSzS}
Let $\l< 1-2\g $ and $s>0$. Then, for each small $\dl>0$, we have
\begin{align*}
\|\NN_W(z^S, R^S) \|_{X_{W_\pm,\dl}^{\l,-\frac 12+}} \les \dl^{0+} \|R^S \|_{X_{S,\dl}^{s,\frac 12+} }
\end{align*}	

\noi
outside an exceptional set of probability $< e^{-\frac 1{\dl^c}}$.
\end{lemma}

\begin{proof}
We perform the case-by-case analysis:	

\smallskip

\noi
{\bf Case 1:} $N_1 \gg N_2$ (non-resonant interaction).
\smallskip

\noi 
In this case, we have 
\begin{align*}
L_{\max} \ges \big|  |n_1|^2 - |n_2|^2 \pm |n|     \big| \gtrsim N_{\max}^2\sim N_1^2.
\end{align*}

\noi
First suppose that $\max(L_2,L) \sim L_{\max}$. Then, from the boundedness of ${\bf P}_{\neq 0}$ on $L^p(\T^2)$, $L_{t,x}^{p}L_{t,x}^{2+}L_{t,x}^{2}$-H\"older's inequality with $p$ large, Lemma \ref{LEM:prstricha}, \eqref{interop}, and Lemma \ref{LEM:dlpower}, we have 
\begin{align*}
&\bigg| \int_{\R}\int_{\T^2} \jb{\nb}^{\l-1+2\g}{\bf P}_{\neq 0}(P_{N_1} z^S P_{N_2} R^S)  P_{N}v^W dxdt \bigg|\\
&\les N^{\l-1+2\g} \| P_{N_1}z^S \|_{L^p_{t,x}} \| P_{N_2}R^S \|_{L^{2+}_{t,x}} \|P_{N}v^W \|_{L^{2}_{t,x}}\\
&\les N^{\l-1+2\g+} N_2^{-s+} \|P_{N_2}R^S \|_{X^{s,0+}_S}  \| P_{N}v^W \|_{X_{W_\pm}^{0,0}}\\
&\les \dl^{\frac 12-} L_{\max}^{-\frac 12+}  N^{\l-1+2\g+} N_2^{-s+}    \|P_{N_2}R^S \|_{X^{s,\frac 12+}_S} \| P_{N_1}v^W \|_{X^{0,\frac 12-}_{W_\pm}}\\
&\les \dl^{\frac 12-} N_1^{\l-2+2\g+} N_2^{-s+} \|P_{N_2}R^S \|_{X^{s,\frac 12+}_S} \| P_{N}v^W \|_{X^{0,\frac 12-}_{W_\pm}}
\end{align*}

\noi
outside an exceptional set of measure $< e^{-\frac 1{\dl^c}}$.
Hence, if $\l<2-2\g$ and $s>0$, then we can perform the dyadic summation over $N_1\sim N \ge N_2$.

Next, suppose that $\max(L_2,L )\ll L_{\max}$ and so $L_1\sim L_{\max}$. Then, from H\"older's inequality with $p \gg 1$, Young's inequality in $\tau$, Lemma \ref{LEM:prob}, and Cauchy-Schwarz inequality in $n_2$, we have
\begin{align*}
&\bigg |\int_{\R}\int_{\T^2} \jb{\nb}^{\l-1+2\g} {\bf P}_{\neq 0}(P_{N_1} z^S P_{N_2} R^S)  P_{N}v^W dxdt \bigg|\\
&= \Bigg| \sum_{\substack{n_2, n \in \Z^2: n_1-n_2=n \\ |n_1|\sim N_1, |n_2| \sim N_2, n_1\neq n_2 } } \jb{n}^s \frac{g_{n_1}(\o)}{\jb{n_1}}  \int_{\tau_1+\tau_2=\tau} \ft \eta_\dl (\tau_1-|n_1|^2) \ft{P_{N_2}R^S}(n_2, \tau_2-|n_2|^2) \ft{P_Nv^W}(n,\tau) d\tau_1 d\tau \Bigg|\\
&\les \| \ft {P_N  v^W } \|_{\ell_n^2L_\tau^{1+}} \Bigg\| \sum_{\substack{n_2 \in \Z^2: n_1-n_2=n \\ |n_1|\sim N_1, |n_2| \sim N_2, |n|\sim N } } \jb{n}^{\l-1+2\g} \frac{|g_{n_1}(\o)|}{\jb{n_1}}  N_1^{-2(1-\eps)} \big\| |\ft \eta_\dl |^\eps   \big\|_{L^p_\tau}  \| \ft{P_{N_2}R^S} \|_{L^1_\tau} \Bigg\|_{\l_n^2} \\
&\les  \dl^{\eps-\frac 1p-}  N^{\l-1+2\g} N_1^{-1+}N_1^{-2+} N_2 N \| P_{N_2} R^S \|_{X_S^{0,\frac 12+}} \| P_N v^W \|_{X_W^{0,\frac 12-}}\\
&\les \dl^{0+} N_1^{\l-2+2\g+} N_2^{-s} \| P_{N_2} R^S \|_{X_S^{s,\frac 12+}} \| P_N v^W \|_{X_W^{0,\frac 12-}}.
\end{align*}

\noi
Hence, if $\l<2-2\g$ and $s>0$, then we can perform the dyadic summation over $N_1\sim N \ge N_2$.

\smallskip

\noi
{\bf  Case 2:} $N_1 \ll N_2$  (non-resonant interaction).
\smallskip

\noi 
In this case, we have 
\begin{align*}
L_{\max} \ges \big|  |n_1|^2 - |n_2|^2 \pm |n|     \big| \gtrsim N_{\max}^2 \sim N_2^2.
\end{align*}

\noi
We point out that in Case 2, it suffices to assume $\l<2-2\g+s$ by proceeding with the proof of Case 1 ($N_1\gg N_2$).

\smallskip

\noi
{\bf Case 3:} $N_1 \sim N_2 \ges N$ (resonant interaction).
\smallskip

\noi
From $L_{t,x}^{p}L_{t,x}^{2+}L_{t,x}^{2}$-H\"older's inequality with $p$ large, Lemma \ref{LEM:prstricha}, \eqref{interop}, and Lemma \ref{LEM:dlpower}, we have
\begin{align*}
&\bigg|\int_{\R}\int_{\T^2} \jb{\nb}^{\l-1+2\g}{\bf P}_{\neq 0}(P_{N_1}z^S P_{N_2}R^S) P_Nv^W dxdt\bigg|\notag\\
&\les N^{\l-1+2\g} \| P_{N_1}z^S \|_{L^p_{t,x}}  \| P_{N_2}R^S \|_{L^{2+}_{t,x}} \| P_N v^W \|_{L_{t,x}^2}\\
&\les \dl^{\frac 12-} N^{\l-1+2\g}  N_2^{-s+}  \| P_{N_2}R^S \|_{X_S^{s,\frac 12+}}   \| P_N v^W \|_{X_{W_\pm}^{0,\frac 12-}}
\end{align*}

\noi
outside an exceptional set of measure $< e^{-\frac 1{\dl^c}}$.
Hence, if $\l<1-2\g$ and $s>0$, we can perform the dyadic summation over $N_1 \sim N_2 \ge N$.
\end{proof}

\section{Random tensor theory }
\label{SEC:rant}
In this section, we present the proof of random tensor estimates which were used in the proof of Lemmas \ref{LEM:zSRW}, \ref{LEM:RszW}, and \ref{LEM:RSzS}).

\subsection{Random tensors}
In this subsection, we provide the basic definition and 
some lemmas on (random) tensors from~\cite{DNY, Bring, OWZ, BDNY}.
See~\cite[Sections 2 and 4]{DNY} and \cite[Section 4]{Bring}
for further discussion.

\begin{definition} \label{def_tensor} \rm
	Let $A$ be a finite index set. We denote by $n_A$ the tuple $ (n_j : j \in A)$. 
	A tensor $h = h_{n_A}$ is a function: $(\Z^2)^{A} \to \mathbb{C} $ with the input variables $n_A$. Note that the tensor $h$ may also depend on $\o \in \O$. 
	The support of a tensor $h$ is the set of $n_A$ such that $h_{n_A} \neq 0$. 
	
	Given a finite index set  $A$, 
	let $(B, C)$ be a partition of $A$. We define the norms 
	$\| \cdot \|_{n_A}$ and 
	$\| \cdot \|_{n_{B} \to n_{C}}$ by 
	\[ \| h \|_{n_A}  = \|h\|_{\l^2_{n_A}} = \bigg(\sum_{n_A} |h_{n_A}|^2\bigg)^\frac{1}{2}\]
	and
	\begin{align}
	\| h \|^2_{n_{B} \to n_{C}} = \sup \bigg\{ 
	\sum_{n_{C}} \Big| \sum_{n_{B}} h_{n_A} f_{n_{B}} \Big|^2 :  \| f \|_{\l^2_{n_{B}}} =1  \bigg\},  
	\label{Z0a}
	\end{align}

	\noi
	where  we used the short-hand notation $\sum_{n_Z}$ for $\sum_{n_Z \in (\Z^2)^Z}$ for a finite index set $Z$.
	Note that, by duality, we have  $\| h \|_{n_{B} \to n_{C}} = \| h \|_{n_{C} \to n_{B}} 
	= \| \cj h \|_{n_{B} \to n_{C}}$ for any tensor $h = h_{n_A}$. 
	If $B = \varnothing$ or $C = \varnothing$,  then we have
	$  \| h \|_{n_{B} \to n_{C}} = \| h \|_{n_A}$.
\end{definition}

For example, when $A = \{1, 2\}$, 
the norm  $\| h \|_{n_{1} \to n_{2}}$ denotes the usual operator norm
$\| h \|_{\l^2_{n_{1}} \to \l^2_{n_{2}}}$
for an infinite dimensional matrix operator $\{h_{n_1 n_2}\}_{n_1, n_2 \in \Z^2}$.
By bounding the matrix operator norm by the Hilbert-Schmidt norm (= the Frobenius norm), we have
\begin{align*}
\| h \|_{\l^2_{n_{1}} \to \l^2_{n_{2}}} \le \| h\|_{\l^2_{n_1, n_2}}
\end{align*}

Let $(B, C)$ be a partition of $A$.
Then, 
by duality, we can write \eqref{Z0a} as 
\begin{align*}
\| h \|_{n_{B} \to n_{C}} = \sup \bigg\{ \Big|
\sum_{n_{B}, n_C} h_{n_A} f_{n_{B}} g_{n_C}\Big| : 
\| f \|_{\l^2_{n_{B}}} =  \| g \|_{\l^2_{n_{C}}} =1   \bigg\}, 
\end{align*}

\noi
from which we obtain
\begin{align*}
\sup_{n_A}|h_{n_A}|
= \sup_{n_B, n_C}|h_{n_Bn_C}|
\le   \| h \|_{n_{B} \to n_{C}}.
\end{align*}

Before we state the main lemma in this subsection (Lemma \ref{LEM:tens}), we first present the following notations. 
For a complex number $z$, we define $z^{+}=z$ and $z^{-}=\cj z$. Let $A$ be a finite index set. For each $j\in A$, we associate $j$ with a sign $\zeta_j \in \{\pm \}$. For $j_1,j_2 \in A$, we say that $(n_{j_1}, n_{j_2})$ is a pairing if $n_{j_1}=n_{j_2}$ and $\zeta_{j_1}=-\zeta_{j_2}$. Let $\{g_n\}_{n\in \Z^2}$ be a set of independent standard complex-valued Gaussian random variables. We write in polar coordinates 
\begin{align*}
g_k(\omega)=\rho_k(\omega)\eta_k(\omega)
\end{align*}

\noi
where $\rho_k=|g_k|$ and $\eta_k=\rho_k^{-1}g_k$. Then all the $\rho_k$ and $\eta_k$ are independent, and each $\eta_k$ is uniformly distributed on the unit circle of $\C$. 

We now present the following random tensor estimate. For the proof, see Proposition 4.14
in \cite{DNY}.

\begin{lemma}[Proposition 4.14 in \cite{DNY}] 
	\label{LEM:tens} 	
	Let $\dl>0$, $A$ be a finite set and $h_{bcn_A}=h_{bcn_A}(\omega)$ be a random tensor, where each $n_j\in \Z^2$ and $(b,c)\in (\Z^2)^q$ for some integer $q\geq 2$. Given signs $\zeta_j\in\{\pm\}$, we also assume that $\langle b\rangle,\langle c\rangle\lesssim M$ and $\langle n_j\rangle\lesssim M$ for all $j\in A$, where $M$ is a dyadic number, and that in the support of $h_{bcn_A}$, there is no pairing in $n_A$. Define the tensor 
	\begin{equation}\label{contract}
	H_{bc}=\sum_{n_{A}}h_{bcn_A}\prod_{j\in A}\eta_{n_j}^{\zeta_j},
	\end{equation}
	
	\noi
	where we assume that $\{h_{bcn_A}\}$ is independent with $\{\eta_n\}_{ n\in \Z^2 }$. Then, there exists constants $C,c>0$ such that we have
	\begin{equation*}\label{contbd}
\|H_{bc}\|_{b\to c}\les \dl^{-\theta}M^{\theta} \cdot\max_{(B,C)}\|h\|_{bn_B\to cn_C},
	\end{equation*} 
	
	\noi
outside an exceptional set of probability $\le C\exp(-\frac {cM}{\dl^\dr } )$ with $\dr>0$, where $(B,C)$ runs over all partitions of $A$.
\end{lemma}

For example, under the independence assumption\footnote{
Lemma 5.2 requires the independence of $\{h_{bcn_A}\}$ and $\{ \eta_n \}$. Since $\rho_{n}=|g_{n}|$ and $\eta_{n}=\rho_{n}^{-1}g_{n}$ are independent, we know that $\{h_{bdac}\rho_{a} \rho_{c} \}  $ and $\{\eta_{n}\}$ are independent, which satisfies the assumption in Lemma \ref{LEM:tens}}, with high probability we have
\[\|H_{bd}\|_{b\to d}\lesssim\max(\|h\|_{abc\to d},\|h\|_{ab\to cd},\|h\|_{bc\to ad},\|h\|_{b\to acd}),\quad \mathrm{where}\quad H_{bd}=\sum_{a,c}h_{bdac}g_a^\pm g_c^\pm.\]

We also present the following variant of Lemma \ref{LEM:tens}. For the proof, see Proposition 4.15 in \cite{DNY}.

\begin{lemma}[Proposition 4.15 in \cite{DNY}]\label{LEM:tens0} 
Consider the same setting as in Lemma \ref{LEM:tens} with the following differences: 
\begin{itemize}
\item[(1)]  We only restrict $\langle n_j\rangle\lesssim M$ for all $j\in A$ but do not impose any condition on $\langle b\rangle$ or $\langle c\rangle$.

\item[(2)] We assume that $b,c\in \Z^2$ and that in the support of the random tensor $h_{bcn_A}$ we have $|b-\zeta c|\lesssim M$ where $\zeta\in\{\pm\}$.

\item[(3)] The random tensor $h_{bcn_A}$ only depends on $b-\zeta c$, $|b|^2-\zeta |c|^2$, and $n_A$, and is supported in the set where $\big||b|^2-\zeta|c|^2\big|\leq M^{5}$.
\end{itemize}

\noi
Then, there exists constants $C,c>0$ such that we have
\begin{equation*}\label{contbd}
\|H_{bc}\|_{b\to c}\les \dl^{-\theta}M^{\theta} \cdot\max_{(B,C)}\|h\|_{bn_B\to cn_C},
\end{equation*} 
	
\noi
outside an exceptional set of probability $\le C\exp(-\frac {cM}{\dl^\dr } )$ with $\dr>0$, where $(B,C)$ runs over all partitions of $A$.
\end{lemma}

\subsection{Deterministic tensor estimates}
\label{SUBSEC:dettens}
In this subsection, we present the deterministic tensor estimates (Lemma \ref{LEM:zSzWy}, \ref{LEM:Tens1}, and \ref{LEM:Tens2}).

\begin{remark}\rm
Lemma \ref{LEM:zSzWy}, \ref{LEM:Tens1}, and \ref{LEM:Tens2} play an important role in proving Lemma \ref{LEM:zSzW} (in particular, Subsubcase 2.b.(ii)), Lemma \ref{LEM:zSRW} (in particular, Subcase 2.b), and Lemma \ref{LEM:RszW} (in particular, Subcase 2.b), respectively.
\end{remark}

\begin{lemma}[First deterministic tensor estimate]\label{LEM:zSzWy}
Let $h_{n,n_1,n_2}(t):=h(n,n_1,n_2,t)$ be the random tensor defined in \eqref{rt13}:
\begin{align*}
h(n,n_1,n_2,t)=e^{-it|n_1|^2}\ind_{ \{n_1=n_2-n \} } \ind_{ \{ \varphi(n_1,n_2,n)=O(N_1^{2s+2\g+}) \} } \jb{n_1}^{-1}  \ind_{\{|n| \sim N \} }   \prod_{j=1}^2 \ind_{\{|n_j| \sim N_j \} },
\end{align*}

\noi
where the phase function $\varphi(n_1,n_2,n)$ is given in \eqref{mm1} and $N_1\sim N \gg N_2$. Then, we have
\begin{align}
\sup_{t\in \R}\| h(t)\|_{n_1n \to n_2} &\les N_1^{s+\g-\frac 12+}N_2^{-\frac 12}+ N_1^{-\frac 12+} \label{w41},\\
\sup_{t\in \R}\| h(t)\|_{n \to n_2n_1} &\les N_1^{-\frac 12+} \label{w42}. 
\end{align}

\end{lemma}

\begin{proof}
We first prove \eqref{w41}. By the Schur's test, we have that
\begin{align}
\begin{split}
\|h(t)\|_{n_1n \to n_2}^2 & \les  \Big(\sup_{n_2} \sum_{n_1,n} |h_{n_1n_2n}(t)| \Big) \Big(\sup_{n_1,n} \sum_{n_2} |h_{n_1n_2n}(t)| \Big) \\ 
& \les  N_1^{-2} \sup_{\substack{n_2 \\ |n_2|\sim N_2} } \Big| \big\{(n_1, n): n_2 = n_1 + n , \varphi (n_1,n_2,n) = O(N_1^{2s+2\g+}),\; |n_1| \sim N_1,\;  |n|\sim N \big\} \Big|\\
& \times \sup_{n_1,n} \Big| \big\{n_2: n_2 = n_1 +n, \varphi (n_1,n_2,n) = O(N_1^{2s+2\g+})\big\}\Big|
\end{split}
\label{QQ1Q}
\end{align}
	
\noi
Since $n_2$ is uniquely determined by $n_1$ and $n$, the last factor can be bounded by 1.
As for the remaining part, we have
\begin{align}
& \sup_{\substack{n_2 \\ |n_2|\sim N_2} } \Big| \big\{(n_1, n): n_2 = n_1 + n , \varphi (n_1,n_2,n) = O(N_1^{2s+2\g+}),\; |n_1| \sim N_1,\; \text{and}\; |n|\sim N \big\} \Big| \notag \\
&\les \sup_{\substack{n_2\\ |n_2|\sim N_2} } \Big|\big\{n_1:  |n_1|^2 \pm |n_2| - |n_2-n_1|^2=  O(N_1^{2s+2\g+}),\; |n_1| \sim N_1,\; \text{and} \; |n_2-n_1|\sim N \big\}\Big| \notag \\
& = \sup_{\substack{n_2\\ |n_2|\sim N_2}} \#S_{n_2}.
\label{S2QQ}
\end{align}
	
\noi
Let $n_1 \in S_{n_2}$. Then, we have  $|n_1|^2\pm |n_2|-|n_2-n_1|^2=O(N_1^{2s+2\g+})$ and so $|n_2| \sim N_2 $ implies that
\begin{align*}
\frac {n_2}{|n_2|} \cdot n_1= -\frac {|n_2|}2 \pm \frac 12 + O\Big(\frac {N_1^{2s+2\g+}}{N_2} \Big)
\end{align*}

\noi
i.e. the component of $n_1$ parallel  to $n_2$ is restricted in an interval of length $O\Big(\frac {N_1^{2s+2\g+}}{N_2} \Big)$. Since $|n_1| \sim N_1$, we have 
\begin{align}
& \sup_{\substack{n_2\\ |n_2|\sim N_2}} \#S_{n_2} \les  \Big(\frac {N_1^{2s+2\g+}}{N_2} +1  \Big)N_1\les N_1^{2s+2\g+1+}N_2^{-1}+ N_1
\label{cru11}
\end{align}

\noi
Therefore, from \eqref{QQ1Q}, \eqref{S2QQ}, and \eqref{cru11}, we have
\begin{align*}
\begin{split}
\|h\|_{ n_1 n \to n_2}^2 
& \les  N_1^{-2}\big(N_1^{2s+2\g+1+}N_2^{-1}+ N_1 \big)\les N_1^{2s+2\g-1+}N_2^{-1}+N_1^{-1}
\end{split}
\end{align*}
	
\noi
which proves \eqref{w41}.

We now prove \eqref{w42}. By the Schur's test, we have that
\begin{align}
\begin{split}
\|h(t)\|_{n \to n_2n_1}^2  &\les \Big(\sup_n \sum_{n_1,n_2} |h_{n_1n_2n}(t)| \Big) \Big(\sup_{n_1,n_2} \sum_{n} |h_{n_1n_2n}(t)| \Big)\\
& \les  N_1^{-2} \sup_{\substack{n \\ |n|\sim N} } \Big| \big\{(n_1, n_2): n_2 = n_1 + n , \varphi (n_1,n_2,n) = O(N_1^{2s+2\g+}),\; |n_1| \sim N_1,\;  |n_2|\sim N_2 \big\} \Big|\\
& \times \sup_{n_1,n_2} \Big| \big\{n: n_2 = n_1 +n, \varphi (n_1,n_2,n) = O(N_1^{2s+2\g+})\big\}\Big|
\label{s11y}
\end{split}
\end{align}
	
\noi
We first consider the last factor in \eqref{s11y}. Since $n$ is uniquely determined by $n_1$ and $n_2$,  the last factor in \eqref{s11y} can be bounded by $1$:
\begin{align}
\sup_{n_1,n_2} \Big| \big\{n: n_2 = n_1 +n, \varphi (n_1,n_2,n) = O(N_1^{2s+2\g+})\big\}\Big| \les 1.
\label{s12y}
\end{align}

\noi	
As for the remaining part, we have
\begin{align}
& \sup_{\substack{n \\ |n|\sim N} } \Big| \big\{(n_1, n_2): n_2 = n_1 + n , \varphi (n_1,n_2,n) = O(N_1^{2s+2\g+}),\; |n_1| \sim N_1,\; \text{and}\; |n_2|\sim N_2 \big\} \Big| \notag \\
&\les \sup_{\substack{n\\ |n|\sim N} } \Big|\big\{n_1:  |n_1|^2 \pm |n_1+n| - |n|^2=  O(N_1^{2s+2\g+}),\; |n_1|\sim N_1,\; \text{and} \; |n_1+n|\sim N_2 \big\}\Big| \notag \\
& = \sup_{\substack{n\\ |n|\sim N}} \#S_n.
\label{s13y}
\end{align}
	
	\noi
	Let $n_1 \in S_n$. Then, we have  
	\begin{align*}
	|n_1|^2-|n|^2=O(N_1^{2s+2\g+} + N_1)=O(N_1)
	\end{align*}
	
	\noi
	since $|n_1+n| \sim N_2 \les N_1$, and $2s+2\g<1$. Therefore, we have 
\begin{align*}
	\big| |n_1|-|n| \big| \les 1,
	\end{align*}
	
	\noi
	which implies that $|n_1| \in ( |n|-c, |n|+c)$ for some constants $c$. Let $|n_1|=\sqrt{m} $, where $m \ge 0$. Then, $m\in (|n|^2-2c|n|+c^2, |n|^2+2c|n|+c^2)$ and so the possible number of $m$ is given by $ |n| \sim N$. 
	Hence, we have
	\begin{align}
	\sup_{\substack{n\\ |n|\sim N}} \#S_n \les N_1^\eps N,
	\label{s14y}
	\end{align}
	
	\noi
	since if $(x,y)=n_1$, where $x,y \in \Z$, then thanks to Lemma \ref{LEM:circle}, the number of lattice points on a circle is given by  
	\begin{align*}
	\big| \{ (x,y)\in \Z^2: x^2+y^2=m   \}  \big| \les N_1^\eps.
	\end{align*}
	
	\noi
	Hence, from \eqref{s11y}, \eqref{s12y}, \eqref{s13y}, and \eqref{s14y}, we have 
	\begin{align*}
	\| h \|^2_{ n \to n_2n_1} \les N_1^{-2}N_1^\eps N\les N_1^{-1+\eps},
	\end{align*}
	
	\noi
	which proves \eqref{w42}.

\end{proof}

\begin{lemma}[Second deterministic tensor estimate]
\label{LEM:Tens1}
Let $h_{n,n_1,n_2}(t):=h(n,n_1,n_2,t)$ be the random tensor defined in \eqref{rt}:
\begin{align*}
h(n,n_1,n_2,t)=e^{-it|n_1|^2}\ind_{ \{n_1=n-n_2 \} } \ind_{ \big\{ \varphi(n_1,n_2,n)=O\big(N^{s+\frac 14+}N_2^{\frac 34}\big) \big\} } \bigg(\prod_{j=1}^2 \ind_{\{|n_j| \sim N_j \} } \bigg) \ind_{\{|n| \sim N \} } \jb{n_1}^{-1},
\end{align*}
 
\noi
where the phase function $\varphi(n_1,n_2,n)$ is given in \eqref{m19} and $N_1\sim N \ges N_2$.
Then, we have 
\begin{align}
\sup_{t\in \R}\|h(t)\|_{n_1 n_2 \to n} & \les N_1^{\frac s2-\frac 12+}, 
\label{A1}\\
\sup_{t\in\R}\|h(t)\|_{n_2 \to n_1 n} & \les N_1^{\frac s2-\frac 38+}N_2^{-\frac 18}+N_1^{-\frac 12}.  
\label{A2}
\end{align}
	
\end{lemma}

\begin{proof}
	We first prove \eqref{A1}. By the Schur's test, we have that
	\begin{align}
	\begin{split}
	\|h(t)\|_{n_1 n_2 \to n}^2  &\les  \Big(\sup_{n_1,n_2} \sum_{n} |h_{n_1n_2n}(t)| \Big)\Big(\sup_n \sum_{n_1,n_2} |h_{n_1n_2n}(t)| \Big)\\
	& \les N_1^{-2} 
	\sup_{n_1,n_2} \Big|\big\{n: n = n_1 + n_2,  \; |\varphi (n_1,n_2, n)|=O(N^{s+\frac 14+}N_2^{\frac 34})  \big\}\Big|\\
\hphantom{X}&\times \sup_{\substack{n \\ |n|\sim N} } \Big|\big\{(n_1,n_2): n = n_1 +n_2, \;  |\varphi (n_1,n_2,n)|=O(N^{s+1+}),\; |n_1|\sim N_1, \; |n_2|\sim N_2  \big\}\Big|
	\label{s11}
	\end{split}
	\end{align}
	
	\noi
	We first consider the first factor in \eqref{s11}. Since $n$ is uniquely determined by $n_1$ and $n_2$,  the first factor in \eqref{s11} can be bounded by $1$:
	\begin{align}
	\sup_{n_1,n_2} |\{n: n = n_1 + n_2,  \; |\varphi (n_1,n_2,n)|=O(N^{s+1+})  \}| \les 1.
	\label{s12}
	\end{align}

\noi	
As for the remaining part, we have
	\begin{align}
	& \sup_{\substack{n \\ |n|\sim N} } \Big|\big\{(n_1,n_2): n = n_1 +n_2, \;  |\varphi (n_1,n_2,n)|=O(N^{s+1+}),\; |n_1|\sim N_1, \; \text{and} \; |n_2|\sim N_2  \big\}\Big| \notag \\
	&\les \sup_{\substack{n\\ |n|\sim N} } \Big|\big\{n_1:  |n_1|^2 \pm |n- n_1| - |n|^2=  O(N^{s+1+}),\; |n_1|\sim N_1,\; \text{and} \; |n-n_1|\sim N_2 \big\}\Big| \notag \\
	& = \sup_{\substack{n\\ |n|\sim N}} \#S_n.
	\label{s13}
	\end{align}
	
	\noi
	Let $n_1 \in S_n$. Then, we have  
	\begin{align*}
	|n_1|^2-|n|^2=O(N^{s+1+} + N_2)
	\end{align*}
	
	\noi
	Therefore, we have 
\begin{align*}
|n|-cN^{s+} \le|n_1|\le |n|+cN^{s+}
\end{align*}
	
\noi
for some constants $c>0$. Since $|n|\sim N$ and $N\gg cN^{s+}$ with $s<1$, $n_1$ is contained in an annulus of thickness $\sim N^{s+}$ and a ball of radius $\sim N_1$. Hence, we have
\begin{align}
\sup_{\substack{n\\ |n|\sim N}} \#S_n \les N^{s+}N_1\sim N^{s+1+}
\label{s14}
\end{align}
	
	
\noi
Hence, from \eqref{s11}, \eqref{s12}, \eqref{s13}, and \eqref{s14}, we have 
\begin{align*}
\sup_{t\in \R}\| h(t) \|^2_{n_1n_2 \to n} \les N_1^{-2} N^{s+1+}\les N_1^{s-1+},
\end{align*}
	
\noi
which proves \eqref{A1}.
	
As for \eqref{A2}, by the Schur's test, we have that
\begin{align}
\begin{split}
\|h(t)\|_{n_2 \to n_1  n}^2 & \les  \Big(\sup_{n_2} \sum_{n_1,n} |h_{n_1n_2n}(t)| \Big) \Big(\sup_{n_1,n} \sum_{n_2} |h_{n_1n_2n}(t)| \Big) \\ 
& \les  N_1^{-2} \sup_{n_1,n} \Big| \big\{n_2: n = n_1 +n_2 , \varphi (n_1,n_2,n) = O(N^{s+\frac 14+}N_2^{\frac 34})\big\}\Big|\\
&\times \sup_{\substack{n_2 \\ |n_2|\sim N_2} } \Big| \big\{(n_1, n): n = n_1 + n_2 , \varphi (n_1,n_2,n) = O(N^{s+\frac 14+}N_2^{\frac 34}),\; |n_1| \sim N_1,\; |n|\sim N \big\} \Big|.
\end{split}
\label{QQ1}
\end{align}
	
\noi
Since $n_2$ is uniquely determined by $n_1$ and $n$, the first factor can be bounded by 1.
As for the remaining part, we have
\begin{align}
& \sup_{\substack{n_2 \\ |n_2|\sim N_2} } \Big|\big\{(n_1,n): n = n_1 +n_2, \;  |\varphi (n_1,n_2,n)|=O(N^{s+\frac 12+}N_2^{\frac 12}),\; |n_1| \sim N_1, \; \text{and} \; |n|\sim N \big\}\Big| \notag \\
&\les \sup_{\substack{n_2\\ |n_2|\sim N_2} } \Big|\big\{n_1:  |n_1|^2 \pm |n_2| - |n_1+n_2|^2=  O(N^{s+\frac 12+}N_2^{\frac 12}),\; |n_1| \sim N_1,\; \text{and} \; |n_1+n_2|\sim N \big\}\Big| \notag \\
& = \sup_{\substack{n_2\\ |n_2|\sim N_2}} \#S_{n_2}.
\label{S2Q}
\end{align}
	
\noi
Let $n_1 \in S_{n_2}$. Then, we have  $|n_1|^2\pm |n_2|-|n_1+n_2|^2=O(N^{s+\frac 14+}N_2^{\frac 34})$ and so $|n_2| \sim N_2 $ implies that
\begin{align*}
\frac {n_2}{|n_2|} \cdot n_1= -\frac {|n_2|}2 \pm \frac 12 + O(N^{s+\frac 14+}N_2^{-\frac 14})
\end{align*}

\noi
i.e. the component of $n_1$ parallel  to $n_2$ is restricted in an interval of length $O(N^{s+\frac 14+}N_2^{-\frac 14})$. Since $|n_1| \sim N_1$, we have 
\begin{align}
& \sup_{\substack{n_2\\ |n_2|\sim N_2}} \#S_{n_2} \les  (N^{s+\frac 14+}N_2^{-\frac 14} +1  )N_1\les N^{s+\frac 54+}N_2^{-\frac 14}+ N_1.
\label{cru1}
\end{align}

\noi
Therefore, from \eqref{QQ1}, \eqref{S2Q}, and \eqref{cru1}, we have
\begin{align*}
\begin{split}
\sup_{t\in\ R}\|h(t)\|_{n_2 \to n_1 n}^2 
& \les  N_1^{-2}\big(N^{s+\frac 54+}N_2^{-\frac 14}+ N_1 \big)\les N^{s-\frac34+}N_2^{-\frac 14}+N_1^{-1},
\end{split}
\end{align*}
	
\noi
which proves \eqref{A2}.
	
\end{proof}

\begin{lemma}[Third deterministic tensor estimate]
\label{LEM:Tens2}
Let $h_{n,n_1,n_2}(t):=h(n,n_1,n_2,t)$ be the random tensor defined in \eqref{rt1}:
\begin{equation*}
\begin{split}
h(n,n_1,n_2,t) &=e^{-it|n_2|} \ind_{ \{n_2=n-n_1, |n_2|\sim N_2\} } \ind_{ \{ \varphi(n_1,n_2,n)=O(N_2^{1+\g+}) \} } \ind_{\{n_1 \in J_{1\ell_1}\}} \ind_{\{n \in J_{2\ell_2}\}}   \jb{n_2}^{-1+\g},
\end{split}
\end{equation*}

\noi
where the phase function $\varphi(n_1,n_2,n)$ is given in \eqref{m1} and $N_1\sim N \ges N_2$.
Then, we have 
\begin{align}
\sup_{t\in \R}\|h(t)\|_{n_1 n_2 \to n} & \les N_2^{-\frac 12+\frac 32\g +}, \label{A11}\\
\sup_{t\in \R}\|h(t)\|_{n_1 \to  nn_2} & \les 
N_2^{-\frac 12+ \frac 32\g +}. 
\label{A21}
\end{align}
	
\end{lemma}

\begin{proof}
	We first prove \eqref{A11}. By the Schur's test, we have that
	\begin{align}
	\begin{split}
	\|h(t)\|_{n_1 n_2 \to n}^2  &\les \Big(\sup_n \sum_{n_1,n_2} |h_{n_1n_2n}(t)| \Big) \Big(\sup_{n_1,n_2} \sum_{n} |h_{n_1n_2n}(t)| \Big)\\
	& \les \sup_{\substack{n\in J_{2\l_2}}} \Big|\big\{(n_1,n_2): n = n_1 +n_2, \;  |\varphi (n_1,n_2,n )|=O(N_2^{1+\g+}), n_1\in J_{1\l_1}, |n_2|\sim N_2 \big\}\Big|\\
	& \times N_2^{-2+2\g} \sup_{n_1,n_2} \Big|\big\{n: n = n_1 + n_2,  \; |\varphi (n_1,n_2,n)|=O(N_2^{1+\g+}) \big\}\Big| .
	\label{s111}
	\end{split}
	\end{align}
	
	\noi
	We first consider the last factor in \eqref{s111}. Since $n$ is uniquely determined by $n_1$ and $n_2$,  the last factor in \eqref{s111} can be bounded by $1$:
	\begin{align}
	\sup_{n_1,n_2} |\{n: n = n_1 + n_2,  \; |\varphi (n_1,n_2,n)|=O(N_2^{1+\g+})  \}| \les 1.
	\label{s121}
	\end{align}
	
\noi	
As for the remaining part, we have
	\begin{align}
	& \sup_{\substack{n\in J_{2\l_2}}} \Big|\big\{(n_1,n_2): n = n_1 +n_2, \;  |\varphi (n_1,n_2,n )|=O(N_2^{1+\g+}),\; n_1\in J_{1\l_1},\; |n_2|\sim N_2 \big\}\Big| \notag \\
	&\les \sup_{\substack{n\in J_{2\l_2}}} \Big|\big\{n_1:  |n_1|^2 \pm |n- n_1| - |n|^2=  O(N_2^{1+\g+}),\; n_1\in J_{1\l_1},\; |n-n_1|\sim N_2 \big\}\Big| \notag \\
	& = \sup_{\substack{n\in J_{2\l_2}}} \# S_n.
	\label{s131}
	\end{align}
	
	\noi
	Let $n_1 \in S_n$. Then, we have  
	\begin{align*}
	|n_1|^2-|n|^2=O(N_2^{1+\g+} + N_2)=O(N_2^{1+\g+} )
	\end{align*}
	
\noi
Therefore, we have 
\begin{align}
\big| |n_1|-|n| \big| \les N_2^{1+\g+}N_1^{-1}.
\label{d1}
\end{align}
	
\noi
We now split the case into $N_2^{1+\g+} \ll N_1$ and $N_2^{1+\g+} \ges  N_1$.

\smallskip

\noi
{\bf Case 1:} $N_2^{1+\g+} \ll N_1$ 
\smallskip

\noi
From \eqref{d1}, in this case we have 
\begin{align*}
|n|-\eps \le|n_1|\le |n|+\eps
\end{align*}
	
\noi
for some $0<\eps\ll 1$. Hence, $n_1$ is contained in an annulus of thickness $\sim \eps$ and a ball of radius $\sim N_2$, which means that we have
\begin{align}
\sup_{\substack{n\in J_{2\l_2}}} \# S_n &\les \big| \{n_1:n_1 \in J_{1\l_1} \} \cap \{n_1: |n|-\eps \le |n_1|\le|n|+\eps \} \big| \notag \\
&\les  \eps N_2.
\label{s141}
\end{align}

\smallskip

\noi
{\bf Case 2:} $N_2^{1+\g+} \ges  N_1$ 
\smallskip

\noi
From \eqref{d1}, in this case we have 
\begin{align*}
|n|-cN_2^{\g+} \le|n_1|\le |n|+cN_2^{\g+}
\end{align*}
	
\noi
for some constants $c>0$. Since $|n|\sim N$ and $N\gg cN_2^{\g+}$ with $\g<1$, $n_1$ is contained in an annulus of thickness $\sim N_2^{\g+}$ and a ball of radius $\sim N_2$. Hence, we have
\begin{align}
\sup_{\substack{n\in J_{2\l_2}}} \# S_n &\les \big| \{n_1:n_1 \in J_{1\l_1} \} \cap \{n_1: |n|-N_2^{\g+} \le |n_1|\le|n|+N_2^{\g+} \} \big| \notag \\
&\les  N_2^{1+\g+}.
\label{y34} 
\end{align}

	

\noi
Hence, from \eqref{s111}, \eqref{s121}, \eqref{s131}, \eqref{s141}, and \eqref{y34}, we have 
\begin{align*}
\sup_{t\in \R}\| h(t) \|^2_{n_1n_2 \to n} \les N_2^{-2+2\g}N_2^{1+\g+}\les N_2^{-1+3\g+},
\end{align*}
	
\noi
which proves \eqref{A11}.
	
As for \eqref{A21}, by the Schur's test, we have that
\begin{align*}
\begin{split}
\|h(t)\|_{n_1 \to n  n_2}^2 & \les  \Big(\sup_{n_1} \sum_{n,n_2} |h_{n_1n_2n}(t)| \Big) \Big(\sup_{n,n_2} \sum_{\substack{n_1}} |h_{n_1n_2n}(t)| \Big) \\ 
& \les N_2^{-2+2\g} \sup_{n,n_2} \Big|\big\{n_1: n = n_1 +n_2 , \varphi (n_1,n_2, n) = O(N_2^{1+\g+})  \big\}\Big|\\
\hphantom{X}&\times \sup_{\substack{n_1\in J_{1\l_1}}} \Big|\big\{(n, n_2): n = n_1 + n_2,\; \varphi (n_1,n_2,n) = O(N_2^{1+\g+}),\; n\in J_{2\l_2},\; |n_2| \sim N_2 \big\}\Big|\\
\end{split}
\end{align*}
	
\noi
Since $n_1$ is uniquely determined by $n$ and $n_2$, 
the first factor can be bounded by 1.
As for the remaining part, we have
\begin{align*}
& \sup_{\substack{n_1\in J_{1\l_1}}} \Big|\big\{(n,n_2): n = n_1 +n_2, \;  |\varphi (n_1,n_2,n )|=O(N_2^{1+\g+}),\; n\in J_{2\l_2},\; |n_2|\sim N_2 \big\}\Big| \notag \\
&\les \sup_{\substack{n_1\in J_{1\l_1}}} \Big|\big\{n:  |n_1|^2 \pm |n- n_1| - |n|^2=  O(N_2^{1+\g+}),\; n\in J_{2\l_2},\; |n-n_1|\sim N_2 \big\}\Big| \notag \\
& = \sup_{\substack{n_1\in J_{1\l_1}}} \# S_{n_1}.
\end{align*}

\noi
As for the counting estimate $\# S_{n_1}$, thanks to the symmetry, we can proceed as in the case \eqref{s131}, which implies
\begin{align*}
\begin{split}
\sup_{t\in \R}\| h(t) \|^2_{n_1 \to nn_2} \les N_2^{-2+2\g}N_2^{1+\g+}\les N_2^{-1+3\g+}.
\end{split}
\end{align*}
	
\noi
This proves \eqref{A21}.
	
\end{proof}

\section{Global well-posedness and invariance of the Gibbs measure}
\label{SEC:INV}
In this section, we extend the local solutions constructed in Theorem \ref{THM:2} to global solutions and prove invariance of the Gibbs measure \eqref{Gibbs4} under the flow of the
renormalized Zakharov-Yukawa system \eqref{Zak2}. We exploit the Bourgain’s invariant measure argument \cite{BO94, BO96, BT2, OTzW}.

\subsection{Invariance of the Gibbs measure under the truncated Zakharov-Yukawa system}
In this subsection, we present frequency-truncated systems and invariance of the Gibbs measure along the flow. For fixed $\eps>0$, $\mu_1 \otimes \mu_{1-\g} \otimes \mu_{-\g}$ is a measure on 
\begin{align*}
H^{-\eps}(\T^2)\times \vec H^{-\g-\eps}(\T^2)
\end{align*}   

\noi
where $\vec H^{-\g-\eps}(\T^2)=H^{-\g-\eps}(\T^2)\times H^{-1-\g-\eps}(\T^2)$.
Given $N\in \N$, we also define the finite-dimensional Gaussian measures $\muu_{\g,N}=(\vec \pi_N)_*(\mu_1 \otimes \mu_{1-\g} \otimes \mu_{-\g})$ on $\vec E_N=\text{span}\big\{(e^{in_1\cdot x }, e^{in_2\cdot x }, e^{in_3\cdot x }): |n_j|\le N,\, j=1,2,3 \big\}$ as the pushforward of $\mu_1 \otimes \mu_{1-\g} \otimes \mu_{-\g}$ under $\vec \pi_N$, where $\vec \pi_N$ is the Dirichlet projector onto the frequencies $\{(n_1,n_2,n_3): |n_j|\le N,\, j=1,2,3 \}$. 

Consider the frequency-truncated system of the renormalized Zakharov-Yukawa system \eqref{Zak2}:
\begin{align}
\begin{cases}
i \dt u^N +\Dl u^N = \pi_{N}\NN_S(u^N,w^N)  \\
\dt^2 w^N +(1- \Dl) w^N = \pi_N \NN_W(u^N,u^N)  \\
\big(u^N, w^N, \dt w^N \big)|_{t=0}=\big(\pi_N u_0, \pi_N w_0, \pi_N v_0 \big),
\label{truZak1}
\end{cases}
\end{align}

\noi
where $\NN_S$ and $\NN_W$ denote the nonlinearity defined in \eqref{bi0} and \eqref{Zak4} and $\big(\pi_N u_0, \pi_N w_0, \pi_N v_0 \big)\in \Vec E_N$. The truncated system \eqref{truZak1} is the finite-dimensional system \eqref{truZak1} of nonlinear ODEs on the Fourier coefficiets of $(u_N, w_N, \dt w_N)$. Therefore, we can conclude by the Cauchy-Lipschitz theorem
that the system of ODEs is locally well-posed. Furthermore, we can extend these solutions
globally-in-time since $\|   (u^N(t), w^N(t), \dt w^N(t))  \|_{H^{1}(\T^2)\times \vec H^{1-\g}(\T^2)}$ can be controlled uniformly in time (but not in $N$) by using the conservation of Hamiltonian and $L^2$-mass of the Schr\"odinger component: for any fixed $N \ge 1$, we have\footnote{Because of the focusing nature of the potential energy in the Hamiltonian, we cannot directly obtain the uniform (in $N$) estimate in the energy class $H^1(\T^2)\times \vec H^{1-\g}(\T^2)$ by only using the energy conservation. We also point out that the $L^2$-norm is only preserved for the component $u_N$, which means that there is no a priori uniform (in both $N$ and $t$) control of solutions $(u^N(t), w^N(t), \dt w^N(t))$ at the $L^2$-level. }
\begin{align*}
\|   (u^N(t), w^N(t), \dt w^N(t))  \|_{H^{1}(\T^2)\times \vec H^{1-\g}(\T^2)}^2 \les H(\pi_N u_0, \pi_N w_0, \pi_N v_0 \big)+N^2M(\pi_N u_0)
\end{align*}

\noi
uniformly in $t\in \R$, where we used the Sobolev and Young's inequality.
In this section, one of our goals is to establish the uniform (in $N$) control of solutions in the support of the Gibbs measure.

Let $\wt {\vec\Phi}_N(t)$ denote the flow map for \eqref{truZak1}. 
Let $d\wt{ \vec  \rho}_{N,\g}$ denote the finite dimensional Gibbs measure associated with the density: 
\begin{align*}
\begin{split}
d\wt{ \vec  \rho}_{N,\g}=Z_{N}^{-1}e^{-Q_N(u, w)} \ind_{ \{|\int_{\T^2} 
: | u^N|^2 :   dx| \le K\}} d\muu_{\g,N}(u,w,\dt w)
\end{split}
\end{align*}

\noi
where $Q_N(u,w)$ is the interaction potential defined in \eqref{poten} and $d\muu_{\g,N}=d(\mu_{1,N} \otimes \mu_{1-\g, N} \otimes \mu_{-\g, N})$ with $\mu_{s,N}:=(\pi_N)_*\mu_s$. 
From the conservation of Hamiltonian, $L^2$-mass of $u^N$ and the Liouville theorem, we can know that the truncated Gibbs measure $d\wt {\vec  \rho}_{N,\g}$ is an invariant measure under the truncated system $\wt {\vec \Phi}_N(t)$.
We also consider the extension of \eqref{truZak1} to infinite dimensions, where the higher modes evolve according to linear dynamics: 
\begin{align}
\begin{cases}
i \dt u^N +\Dl u^N = \pi_{N}\NN_S(\pi_Nu^N,\pi_Nw^N)  \\
\dt^2 w^N +(1- \Dl) w^N = \pi_N \NN_W(\pi_N u^N, \pi_N u^N)  \\
\big(u^N, w^N, \dt w^N \big)|_{t=0}=\big(u_0, w_0, v_0 \big),
\label{truZak2}
\end{cases}
\end{align}

\noi
where $\big(u_0, w_0, v_0 \big) \in H^{-\eps}(\T^2) \times \vec H^{-\g-\eps}(\T^2)$. In other words, $\eqref{truZak2}$ allows us to discuss the two decoupled flows, where the high frequency part evolves linearly and the low frequency part corresponds to the finite-dimensional system \eqref{truZak1} of nonlinear ODEs. 
Let $\vec \Phi_N(t)$ be the flow map for \eqref{truZak2}. Then, we have 
\begin{align*}
\vec \Phi_N(t)=\wt {\vec \Phi}_N(t)\vec \pi_N+\vec {S}(t) \vec \pi_N^\perp
\end{align*}

\noi
where $\vec \pi_N^\perp:=\Id-\vec \pi_N$. We denote by $\vec E_N^\perp$ the orthogonal complement of $\vec E_N$ in $H^{-\eps}(\T^2) \times \vec H^{-\g-\eps}(\T^2)$. Let $\vec \mu_{\g,N}^\perp$ be the Gaussian field on $\vec E_N^\perp$ i.e. $\muu_{\g,N}^\perp=(\vec \pi_N^\perp)_*(\mu_1 \otimes \mu_{1-\g} \otimes \mu_{-\g})$ on $\vec E_N^\perp$ as the pushforward of $\mu_1 \otimes \mu_{1-\g} \otimes \mu_{-\g}$ under $\vec \pi_N^\perp$.   
We define the truncated Gibbs measure $d\vec \rho_{\g,N}$ as follows:
\begin{align*}
d\vec \rho_{\g,N}=d\wt{ \vec  \rho}_{N,\g} \otimes d\vec \mu_{\g,N}^\perp.
\end{align*}

\noi
From the invariance of $d\wt {\vec \rho}_{\g,N}$ under the flow $\wt{\vec \Phi}_{N}(t)$ and the invariance of the Gaussian measures $d\vec \mu_{\g,N}^\perp$  under rotations $\vec S(t)$, we conclude the following invariance of $d\vec \rho_{\g,N}$ under the truncated system $\vec \Phi_N(t)$.
\begin{lemma}\label{LEM:truninva}
For each $t\in \R$, the Gibbs measure $d\vec \rho_{\g, N}$ is invariant under the flow map $\vec \Phi_N(t)$ on $H^{-\eps}(\T^2) \times \vec H^{-\g-\eps}(\T^2) $.
\end{lemma}

\subsection{Almost sure global well-posedness}
In this subsection, by using the invariance of the Gibbs measure for \eqref{truZak2} (Lemma \ref{LEM:truninva}) and a standard PDE approximation argument, we obtain the almost sure global well-posedness.

\begin{lemma}\label{LEM:truapprox}
There exist small $0<\eps<\eps_1\ll 1$ and $\be>0$ such that given any small $\kappa>0$ and $T>0$, there exists a measurable set $\Si_{\kappa, T}\subset H^{-\eps}(\T^2)\times \vec H^{-\g-\eps}(\T^2)$ such that \textup{(i)} $\vec \rho_{\g}(\Si_{\kappa,T}^c)<\kappa$ and \textup{(ii)} for any $(u_0,w_0, v_0)\in \Si_{\kappa, T}$, there exists a (unique) solution 
\begin{align*}
z^{S,\o} +X^{s,b}_S(T)  \subset   & C\big([-T, T ]; H^{-\eps}(\T^2)  \big)  \\ 
z^{W,\o} +X^{\l,b}_{W}(T) \subset   & C\big([-T, T ]; H^{-\g-\eps}(\T^2)  \big) \cap C^1\big([-T,T]; H^{-\g-1-\eps}(\T^2) \big)
\end{align*} 

\noi
to the renormalized Zakharov-Yukawa system \eqref{Zak2} with $(u,w,\dt w)|_{t=0}=(u_0,w_0,v_0)$ for some $s>0$ and $\l>0$ in Theorem \ref{THM:2}.
Furthermore, given any large $N\gg 1$, we have
\begin{align*}
&\Big\| \big(u(t), w(t), \dt w(t) \big)- \vec \Phi_N(t)(u_0,w_0,v_0) \Big \|_{C([-T,T]; H^{-\eps_1}(\T^2) \times \vec H^{-\g-\eps_1}(\T^2) )}\les C(\kappa,T)N^{-\beta},
\end{align*}

\noi
where $\vec \Phi_N(t)$ denotes the flow map for \eqref{truZak2}.
\end{lemma}

\begin{proof}
Once we have almost sure local well-posedness (Theorem \ref{THM:2}), the proof of Lemma \ref{LEM:truapprox}
is by now standard. In the following, we only sketch key parts of the argument and refer to \cite{BO94, BO96, BT2, Richards0, Richards} for further details.

It is convenient to reduce the Zakharov-Yukawa system \eqref{Zak4} to a first-order system
by setting $ w_{\pm}:=w \pm i\jb{\nb}^{-1} \dt w $ as in the proof of Theorem \ref{THM:2}; see also Subsection \ref{SUBSEC:reduce}. An observation is that convergence properties of $w_{\pm}$ can be directly converted to convergence properties of $w$ by taking the real part; $w=\Re w_{\pm} $. In the following, we drop the $\pm$ signs from $w_{\pm}$ and work with one $w_{+}$ or $w_{-}$ since there is no role of $\pm$.

For $M\ge N \ge 1$, we can write
\begin{align}
(u^M,w^M)-(u^N,w^N)&= (\pi_Mu^M-\pi_Nu^N, \pi_Mw^M-\pi_N w^N) \notag \\
&\hphantom{X}+(\pi^\perp_Mu^M, \pi^\perp_M w^M)-(\pi^\perp_N u^N, \pi^\perp_N w^N).
\label{decomp4}
\end{align}

\noi
The convergence of 
\begin{align*}
(\pi_Mu^M- \pi_Me^{it\Dl}u_0^\o)&-(\pi_Nu^N-\pi_Ne^{it\Dl}u_0^\o) \too 0\\
(\pi_Mw^M- \pi_Me^{it\jb{\nb}}w_0^\o)&-(\pi_Nw^N-\pi_Ne^{it\jb{\nb}}w_0^\o)\too 0
\end{align*}

\noi
can be shown exactly as in the proof of Theorem \ref{THM:2}, locally in time, i.e. in $ X_S^{s, \frac 12+}(\dl) \times X_{W_\pm}^{\l, \frac 12+}(\dl) \subset C\big([-\dl,\dl]; H^{s}(\T^2) \times  H^{\l}(\T^2) \big) $ for some $s>0$ and $\l>0$, which implies the following convergence:
\begin{align*}
(\pi_Mu^M-\pi_Nu^N, \pi_Mw^M-\pi_N w^N) \too 0  \quad \text{as} \quad N \to \infty
\end{align*}

\noi
in $C\big([-\dl,\dl]; H^{-\eps_1}(\T^2) \times  H^{-\g-\eps_1}(\T^2) \big)$ 
since we also have 
\begin{align*}
( \pi_Me^{it\Dl}u_0^\o-\pi_Ne^{it\Dl}u_0^\o, \pi_Me^{it\jb{\nb} }w_0^\o-\pi_Ne^{it\jb{\nb}}w_0^\o )  \too 0  \quad \text{as} \quad N \to \infty
\end{align*}

\noi
in $C\big([-\dl,\dl]; H^{-\eps_1}(\T^2) \times  H^{-\g-\eps_1}(\T^2) \big)$. 

On the other hand, the second and third terms in \eqref{decomp4} decay like $N^{-\beta}$ for some $\be>0$ thanks to the high frequency projections.
The remaining part of the argument leading to the proof of Lemma \ref{LEM:truapprox} is contained in \cite{BO94, BO96, BT2, Richards0, Richards}. In particular, see the proof of Proposition 3.5 in \cite{Richards0} for details in a setting analogous to our work.



\end{proof}

Once we obtain Lemma \ref{LEM:truapprox}, the almost sure global well-posedness follows from the Borel-Cantelli lemma. Given $\kappa>0$, let $T_j=2^{j}$ and $\kappa_j=\frac {\kappa}{2^j}$, $j\in \N$. By exploiting Lemma \ref{LEM:truapprox}, we can construct a set $\Si_{\kappa_j,T_j}$ and set
\begin{align}
\Si_\kappa:=\bigcap_{j=1}^\infty \Si_{\kappa_j,T_j}.
\label{sikap}
\end{align}

\noi
Then, we have $\vec\rho_\g(\Si_\kappa^c)<\kappa$ and for any $(u_0,w_0,v_0) \in \Si_{\kappa}$, there exists a unique global-in-time solution to the renormalized Zakharov-Yukawa system \eqref{Zak2} with $(u,w,\dt w)|_{t=0}=(u_0,w_0,v_0)$. Finally, we set
\begin{align}
\Si:=\bigcup_{n=1}^\infty \Si_{\frac 1n}.
\label{fullpro1}
\end{align}

\noi
Then, we have $\vec \rho_\g(\Si^c)=0$ and for any $(u_0,w_0,v_0) \in \Si$, there exists a unique global-in-time solution to the renormalized Zakharov-Yukawa system \eqref{Zak2} with $(u,w,\dt w)|_{t=0}=(u_0,w_0,v_0)$, which means that we prove almost sure global well-posedness.

\subsection{Invariance of the Gibbs measure}
Let $\vec \Phi(t)$ be the flow map for the renormalized Zakharov-Yukawa system \eqref{Zak2} defined on the set $\Si$ of full probability constructed in \eqref{fullpro1}. The main goal of this subsection is to establish the invariance of the Gibbs measure along the flow $\vec \Phi(t)$:
\begin{align}
\int_{\Si}F(\vec \Phi(t)(u,w,v) )d\vec \rho_\g(u,w,v)=\int_{\Si} F(u,w,v) d\vec \rho_\g(u,w,v)
\label{Invar1}
\end{align} 

\noi
for any $F \in L^1(H^{-\eps}(\T^2)\times \vec H^{-\g-\eps}(\T^2), d\vec \rho_\g)$ and any $t\in \R$. From a density argument, it is enough to show \eqref{Invar1} for continuous and bounded $F$.
Fix $t\in \R$. Note that 
\begin{align}
&\bigg| \int_{\Si}F(\vec \Phi(t)(u,w,v) )d\vec \rho_\g(u,w,v)- \int_{\Si} F(u,w,v) d\vec \rho_\g(u,w,v) \bigg| \notag \\
&\le \bigg| \int_{\Si}F(\vec \Phi(t)(u,w,v) )d\vec \rho_\g(u,w,v)- \int_{\Si} F(\Phi(t)(u,w,v) ) d\vec \rho_{\g,N}(u,w,v) \bigg| \notag \\
&\hphantom{X}+\bigg| \int_{\Si} F(\Phi(t)(u,w,v) ) d\vec \rho_{\g,N}(u,w,v)- \int_{\Si} F(\Phi_N(t)(u,w,v) d\vec \rho_{\g,N}(u,w,v) \bigg| \notag  \\
&\hphantom{X}+\bigg| \int_{\Si} F(\Phi_N(t)(u,w,v) d\vec \rho_{\g,N}(u,w,v)- \int_{\Si} F(u,w,v) d\vec \rho_{\g,N}(u,w,v) \bigg| \notag \\
&\hphantom{X}+\bigg|\int_{\Si} F(u,w,v) d\vec \rho_{\g,N}(u,w,v)- \int_{\Si} F(u,w,v) d\vec \rho_\g(u,w,v) \bigg| \notag \\
&=I_N+II_N+III_N+IV_N.
\label{INVASOG}
\end{align}

\noi
From Lemma \ref{LEM:truninva}, we have
\begin{align}
\int_{\Si}F(\vec \Phi_N(t)(u,w,v) )d\vec \rho_{\g,N}(u,w,v)=\int_{\Si} F(u,w,v) d\vec \rho_{\g,N}(u,w,v),
\label{Invar12}
\end{align}

\noi
which implies
\begin{align*}
III_N=0. 
\end{align*}

\noi
Thanks to Theorem \ref{THM:1} (especially, \eqref{c11}), we have \footnote{In fact, we have the total variation convergence of measures.} that $d\rhoo_{\g,N}$ converges weakly to $ d\rhoo_{\g}$, which implies
\begin{align*}
I_N+IV_N \too 0 \qquad \text{as} \qquad N\to \infty.
\end{align*}

\noi
Let $\dl>0$. Then, the boundedness of $F$ and the total variation convergence of $d\rhoo_{\g,N}$ to $d\rhoo_\g$ imply that for any sufficiently small $\kappa>0$ and sufficiently large $N \gg 1$, we have 
\begin{align}
\Bigg| \int_{\Si_\kappa^c}F(\vec \Phi(t)(u,w,v) )d\vec \rho_{\g,N}(u,w,v) \Bigg|+\Bigg| \int_{\Si^c_\kappa}F(\vec \Phi_N(t)(u,w,v) )d\vec \rho_{\g,N}(u,w,v) \Bigg| <\dl,
\label{Inva13}
\end{align}

\noi
where $\Si_\kappa$ is defined in \eqref{sikap}. Let us fix one such $\kappa>0$. Then, from  \eqref{INVASOG} and \eqref{Inva13}, we have
\begin{align}
II_N &\le \dl+\int_{\Si_\kappa} \big| F(\vec \Phi(t)(u,w,v) )- F(\vec \Phi_N(t)(u,w,v) ) \big| d(\vec \rho_{\g,N}-\vec \rho_\g)(u,w,v)   \notag \\
&\hphantom{X}+\int_{\Si_\kappa} \big| F(\vec \Phi(t)(u,w,v) )- F(\vec \Phi_N(t)(u,w,v) ) \big| d\vec \rho_\g(u,w,v).  
\label{Inva14}
\end{align}

\noi
Thanks to the boundedness of $F$ and the total variation convergence of $d\rhoo_{\g,N}$ to $d\rhoo_\g$, we have 
\begin{align}
\int_{\Si_\kappa} \big| F(\vec \Phi(t)(u,w,v) )- F(\vec \Phi_N(t)(u,w,v) ) \big| d(\vec \rho_{\g,N}-\vec \rho_\g)(u,w,v) \too 0
\label{TVTVcon}
\end{align}

\noi
as $N \to \infty$. From Lemma \ref{LEM:truapprox}, we have
\begin{align}
\| \vec \Phi(t)(u,w\,v)- \vec \Phi_N(t)(u,w, v) \|_{H^{-\eps}(\T^2) \times \vec H^{-\g-\eps}(\T^2)  } \le C(\kappa, t) N^{-\beta}
\label{APPINVA}
\end{align}

\noi
for any $(u,w,v)\in \Si_\kappa$ and sufficiently large $N \gg 1$. Hence, from \eqref{APPINVA}, the continuity of $F$ and the dominated convergence theorem, we have
\begin{align}
\int_{\Si_\kappa} \big| F(\vec \Phi(t)(u,w,v) )- F(\vec \Phi_N(t)(u,w,v) ) \big| d\vec \rho_\g(u,w,v) \too 0 
\label{dominated10}
\end{align}

\noi
as $N\to \infty$. From \eqref{Inva13}, \eqref{Inva14}, \eqref{TVTVcon}, \eqref{dominated10}, and taking $\dl \to 0$, we have
\begin{align*}
II_N \too 0 \qquad \text{as} \qquad N\to \infty,
\end{align*}

\noi
which proves \eqref{Invar1}.

\appendix

\section{Hilbert-Schmidt norm approach}
\label{SEC:A}

In this appendix, we give a brief discussion on the Hilbert-Schmidt norm approach
and compare its result with when using the operator norm bound with the random tensor theory. 
In particular, we handle 
the kernel (random) matrix   by using
the Hilbert-Schmidt norm with the Wiener chaos estimate (Lemma~\ref{LEM:multigauss}). 
In the following, we look into Subcase 2.b in Lemma \ref{LEM:RszW}, which gave the restriction $\g<\frac 13$ in the operator norm approach with the random tensor estimates. We will see that by using the Hilbert-Schmidt norm approach, Lemma \ref{LEM:RszW} is no longer satisfied for any $\g \ge 0$, which shows that the approach with the random tensor theory is essential to obatin the main result. Let us first recall Subcase 2.b in Lemma \ref{LEM:RszW}:

\smallskip
\noi
{\bf  Subcase 2.b in Lemma \ref{LEM:RszW}:} $L_{\max}\les N_2^{1+\g+}$ (low modulation)  and $N_1 \sim N \ges N_2$ (resonant interaction). 

\smallskip
\noi
This time we try to prove the bilinear estimate \eqref{APPbi} with the Hilbert-Schmidt norm approach.
When $N_1\gg N_2$, we first wirte $\{|n_1| \sim N_1 \}=\bigcup_{\ell_1} J_{1\ell_1}$ and $\{ |n| \sim N \}=\bigcup_{\ell_2}J_{2\ell_2}$, where $J_{1,\ell_1}$ and $J_{2,\ell_2}$ are balls of radius $\sim N_2$, we can decompose $\ft {P_{N_1}R^S}$ and $\ft {P_{N}v^S}$ as
\begin{align*}
\ft {P_{N_1}R^S}=\sum_{\ell_1} \ft { P_{N_1,\ell_1}R^S } \qquad \text{and} \qquad \ft {   P_{N}v^S }=\sum_{\ell_2} \ft { P_{N,\ell_2}v^S }
\end{align*} 

\noi
where $\ft {P_{N_1,\ell_1}R^S }(n_1,t)=\ind_{J_{1\ell_1}}(n_1) \ft { P_{N_1}R^S }(n_1,t) $ 
and $\ft {P_{N,\ell_2}v^S }(n,t)=\ind_{J_{2\ell_2}}(n) \ft { P_{N}v^S }(n,t) $. Given $n_1 \in J_{1\ell_1}$ for some $\ell_1$, there exists $O(1)$ many possible values for $\ell_2=\ell_2(\ell_1)$ such that $n \in J_{2\ell_2}$ under $n_1+n_2=n$. Notice that the number of possible values of $\l_2$ is independent of $\l_1$. 

From the Cauchy-Schwarz inequality, we have
\begin{align}
&\bigg|\int_{\R}\int_{\T^2} \jb{\nb}^s(P_{N_1,\l_1}R^S P_{N_2}z^W) P_{N,\l_2}v^S dxdt\notag\bigg|\\ 
&=\bigg| \sum_{\substack{ n \in J_{2\l_2}}} \int_\R \jb{n}^s\ft{ P_{N,\l_2} v^S}(n,t) \sum_{\substack{n_1\in J_{1\l_1}\\ n_1+n_2=n,\; |n_2|\sim N_2}} \ft{P_{N_1,\l_1} R^S}(n_1,t) \jb{n_2}^{-1+\g} e^{-it|n_2|}h_{n_2}(\o)\eta_\dl(t) dt \bigg|\notag \\
&\les  \bigg(\sum_{ n \in J_{2\l_2}   } \| \jb{n}^s\ft {P_{N,\l_2}v^S}(n,t)   \|_{L_t^2}^2  \bigg)^\frac 12
\bigg(  \sum_{n \in J_{2\l_2} } \Big\|  \sum_{\substack{n_1 \in J_{1\l_1}: n_1+n_2=n}} \s^\o(n,n_1,t) \eta_\dl(t) \ft{P_{N_1,\l_1}R^s}(n_1,t)     \Big  \|_{L_t^2}^2   \bigg)^\frac 12 \notag \\
&\sim N^s \| P_{N,\l_2} v^S  \|_{X_S^{0,0}}  \bigg(  \sum_{n \in J_{2\l_2} } \Big\|  \sum_{\substack{n_1 \in J_{1\l_1}: n_1+n_2=n}} \s^\o(n,n_1,t) \eta_\dl(t) \ft{P_{N_1,\l_1}R^s}(n_1,t)     \Big  \|_{L_t^2}^2   \bigg)^\frac 12,
\label{c1}
\end{align}

\noi
where the random matrix $\s^\o(n,n_1,t)$ is given by
\begin{align*}
\begin{split}
\s^\o(n,n_1,t)=
\begin{cases}
\frac {e^{-it|n_2|}h_{n_2}(\o)}{\jb{n_2}^{1-\g}}, \quad &\text{if} \quad  |n_1|^2\pm |n-n_1|-|n|^2=O(N_2^{1+\g+}), \; n_1+n_2=n, \; |n_2|\sim N_2\\
0,  \quad &\text{otherwise}.
\end{cases}
\end{split}
\end{align*}

\noi
Then, from Lemma \ref{LEM:matrix}, we have
\begin{align}
&\sum_{n \in J_{2\l_2} } \Big\|  \sum_{\substack{n_1 \in J_{1\l_1}: n_1+n_2=n}} \s^\o(n,n_1,t) \eta_\dl(t) \ft{P_{N_1,\l_1}R^s}(n_1,t)     \Big  \|_{L_t^2}^2  \notag \\
&\les \sup_{t\in \R}\Bigg[  \underbrace{\max_{\substack{n\in J_{2\l_2} }}\sum_{n_1\in S_{n}} |\s^\o(n,n_1,t)|^2}_{\text{diagonal term}}+\underbrace{\bigg( \sum_{\substack{n_1 \neq n_1' }}  \Big|\sum_{\substack{n\in \Z^2 }}  \s^\o(n,n_1,t) \cj{\s^\o}(n, n_1',t)  \Big|^2    \bigg)^\frac 12}_{\text{non-diagonal term}}\Bigg]  \|P_{N_1,\l_1}R^S \|_{X_S^{0,0}}^2,
\label{c2}
\end{align}

\noi
where
\begin{align}
S_{n}=\big\{ n_1 \in J_{1\l_1}: |n_1|^2\pm |n-n_1|-|n|^2=O(N_2^{1+\g+}),\; |n-n_1|\sim N_2 \big\}
\label{Sn} 
\end{align}

\noi
We now split the case into $N_2^{1+\g+} \ll N_1$ and $N_2^{1+\g+} \ges  N_1$.

\smallskip

\noi
{\bf Case 1:} $N_2^{1+\g+} \ll N_1$ 
\smallskip

\noi
In this case, from \eqref{Sn} we have 
\begin{align*}
|n|-\eps \le|n_1|\le |n|+\eps
\end{align*}
	
\noi
for some $0<\eps\ll 1$. Hence, $n_1$ is contained in an annulus of thickness $\sim \eps$ and a ball of radius $\sim N_2$. Hence, we have
\begin{align}
\sup_{\substack{n\in J_{2\l_2}}} \# S_n &\les \sup_{\substack{n\in J_{2\l_2}}} \big| \{n_1:n_1 \in J_{1\l_1} \} \cap \{n_1: |n|-\eps \le |n_1|\le|n|+\eps \} \big| \notag \\
&\les  \eps N_2.
\label{App1}
\end{align}

\smallskip

\noi
{\bf Case 2:} $N_2^{1+\g+} \ges  N_1$ 
\smallskip

\noi
In this case, from \eqref{Sn} we have 
\begin{align*}
|n|-cN_2^{\g+} \le|n_1|\le |n|+cN_2^{\g+}
\end{align*}
	
\noi
for some constants $c>0$. Since $|n|\sim N$ and $N\gg cN_2^{\g+}$ with $\g<1$, $n_1$ is contained in an annulus of thickness $\sim N_2^{\g+}$ and a ball of radius $\sim N_2$. Hence, we have
\begin{align}
\sup_{\substack{n\in J_{2\l_2}}} \# S_n &\les \sup_{\substack{n\in J_{2\l_2}}} \big| \{n_1:n_1 \in J_{1\l_1} \} \cap \{n_1: |n|-N_2^{\g+} \le |n_1|\le|n|+N_2^{\g+} \} \big| \notag \\
&\les  N_2^{1+\g+}.
\label{App2}
\end{align}

\noi
By counting the number of lattice points as in \eqref{App1} and \eqref{App2}, the diagonal term can be estimated as follows:
\begin{align}
\sup_{\substack{n\in J_{2\l_2} }}\sum_{n_1\in S_{n}} |\s^\o(n,n_1,t)|^2 &\les N_2^{-2+2\g+} \sup_{\substack{n\in J_{2\l_2} }}\#S_{n}\notag\\
& \les N_2^{-2+2\g+} N_2^{1+\g+}= N_2^{-1+3\g+} 
\label{c3}
\end{align} 

\noi
outside an exceptional set of probability $< e^{-\frac 1{\dl^c}}$, where we used Lemma \ref{LEM:prob} in the first step of \eqref{c3}.

Next, we  consider the non-diagonal term. Note that  
\begin{align}
&\bigg( \sum_{\substack{n_1 \neq n_1' }}  \Big|\sum_{\substack{n\in \Z^2 }}  \s^\o(n,n_1,t) \cj{\s^\o}(n, n_1',t)  \Big|^2       \bigg)^\frac 12 \notag\\
&= \bigg( \sum_{\substack{n_1 \neq n_1' }}  \Big|       \sum_{\substack{n \in \Z^2 :\\|n_1|^2\pm |n-n_1|-|n|^2=O(N_2^{1+\g+}), \\ |n_1'|^2\pm |n-n_1'|-|n|^2=O(N_2^{1+\g+}) } }a_{n,n_1,n_1'}\cdot h_{n-n_1}(\o) \cj{h}_{n-n_1'}(\o)  \Big|^2  \bigg)^\frac 12 
\label{RszW}
\end{align}

\noi
where 
\begin{align*}
a_{n,n_1,n_1'}=\jb{n-n_1}^{-1+\g}\jb{n-n_1'}^{-1+\g}\ind_{ \big\{ |n-n_1|\sim N_2, \; |n-n_1'|\sim N_2 \big\}}
\end{align*}
From Lemma \ref{LEM:multigauss} and using the independence of $\{h_n \}$ with the condition $n_1 \neq n_1'$, we have
\begin{align}
&\Big|        \sum_{\substack{n \in \Z^2 :\\|n_1|^2\pm |n-n_1|-|n|^2=O(N_2^{2\g+})\\ |n_1'|^2\pm |n-n_1'|-|n|^2=O(N_2^{2\g+}) } } a_{n,n_1,n_1'}\cdot h_{n-n_1}(\o) \cj{h}_{n-n_1'}(\o)  \Big| \notag \\
&\les \dl^{0-} N_2^{0+}
\Bigg(\E_\PP \bigg| 
 \sum_{\substack{n \in \Z^2 :\\|n_1|^2\pm |n-n_1|-|n|^2=O(N_2^{2\g+})\\ |n_1'|^2\pm |n-n_1'|-|n|^2=O(N_2^{1+\g+}) } } a_{n,n_1,n_1'}\cdot  h_{n-n_1}(\o) \cj{h}_{n-n_1'}(\o)      \bigg|^2 \Bigg)^\frac 12 \notag\\
&\les \dl^{0-} N_2^{0+}  \Big( \sum_{\substack{n \in \Z^2:\\|n_1|^2\pm |n-n_1|-|n|^2=O(N_2^{1+\g+})\\ |n_1'|^2\pm |n-n_1'|-|n|^2=O(N_2^{1+\g+}) } } |a_{n,n_1,n_1'}|^2    \Big)^\frac 12.
\label{RszW1}
\end{align}

\noi
Hence, from \eqref{RszW}, \eqref{RszW1}, and counting the number of lattice points as in \eqref{App1} and \eqref{App2}, we have
\begin{align}
&\sum_{\substack{n_1 \neq n_1' }}  \Big|\sum_{\substack{n\in \Z^2 }}  \s^\o(n,n_1,t) \cj{\s^\o}(n, n_1',t)  \Big|^2  \notag \\
&\les N_2^{-4+4\g+} \# \Big\{ (n_1,n_1',n): |n_1|^2\pm |n-n_1|-|n|^2=O(N_2^{1+\g+}), \; n_1 \in J_{1\l_1}, \; n\in J_{2\l_2}, \; |n-n_1|\sim N_2 \notag \\
&\hphantom{XXXXXXXXXXXXX}\hspace{1mm} |n_1'|^2\pm |n-n_1'|-|n|^2=O(N_2^{1+\g+}),\; n_1'\in J_{1\l_1},\; |n-n_1'|\sim N_2  \Big \}\notag \\
&\les N_2^{-4+4\g+} \sum_{n_1'\in J_{1\l_1}} \hspace{2mm} \sum_{\substack{n_1\in J_{1\l_1}\\ |n_1|^2-|n_1'|^2=O(N_2^{1+\g+}) }} \hspace{2mm} \sum_{\substack{n\in J_{2\l_2}\\ |n_1|^2 -|n|^2=O(N_2^{1+\g+}) }}\notag \\
&\les N_2^{-4+4\g+} N_2^2 N_2^{1+\g+}N_2^{1+\g+} \sim N_2^{6\g+}.
\label{S}
\end{align}


\noi
Therefore, by combining \eqref{c1}, \eqref{c2}, \eqref{c3}, and \eqref{S}, we have 
\begin{align*}
&\bigg|\int_{\R}\int_{\T^2} \jb{\nb}^s(P_{N_1,\l_1}R^S P_{N_2}z^W) P_{N,\l_2}v^S dxdt\notag\bigg|\\
&\les  \sup_{t\in \R}\Bigg[  \underbrace{\max_{\substack{n\in J_{2\l_2} }}\sum_{n_1\in S_{n}} |\s^\o(n,n_1,t)|^2}_{\text{diagonal term}}+\underbrace{\bigg( \sum_{\substack{n_1 \neq n_1' }}  \Big|\sum_{\substack{n\in \Z^2 }}  \s^\o(n,n_1,t) \cj{\s^\o}(n, n_1',t)  \Big|^2    \bigg)^\frac 12}_{\text{non-diagonal term}}\Bigg]^\frac 12\\
& \hphantom{XXXX}  \times \|P_{N_1,\l_1}R^S \|_{X_S^{s,0}} \|P_{N,\l_2} v^S \|_{X_S^{0,0}}\\
&\les \big(N_2^{-\frac 12+\frac 32\g+}+ N_2^{\frac 23\g+}\big) \|   P_{N_1,\l_1}R^S  \|_{X_S^{s,0}} \|P_{N,\l_2} v^S \|_{X_S^{0,0}}.
\end{align*}

\noi
Therefore, even if $\g=0$, we cannot perform the dyadic summation over $N_1 \sim N \ges N_2$,
which shows that Lemma \ref{LEM:RszW} cannot be proven by using the Hilbert-Schmidt norm approach with the lattice counting method used in this paper for any $\g \ge 0$. In other words, the operator norm approach with the random tensor theory is essential to obtain the main result.


\begin{ackno}\rm
The author would like to kindly thank his Ph.D. advisor Tadahiro Oh for suggesting the problem and continued support. 
The author is grateful to Yuzhao Wang and Younes Zine for helpful discussions. In addition,  the author would like to appreciate a number of comments  given by the anonymous referee,
which improved the presentation of the paper.
K.S.~was partially supported by National Research Foundation of Korea (grant NRF-2019R1A5A1028324).	
\end{ackno}

\end{document}